\documentclass[11pt]{article}

\usepackage{geometry}
\usepackage[dvips]{graphicx}
\usepackage{float}
\usepackage{textcomp}
\usepackage{amsthm}
\usepackage{amsmath}
\usepackage{amssymb}
\usepackage{subfigure}
\usepackage[hidelinks]{hyperref}
\usepackage{cite}
\usepackage[nottoc]{tocbibind}
\usepackage{color}
\usepackage{multirow}
\usepackage{tabularx}
\usepackage{verbatim}
\usepackage[longend,ruled]{algorithm2e}
\newtheorem{theorem}{Theorem}[section]

\newtheorem{proposition}[theorem]{Proposition}
\newtheorem{remark}[theorem]{Remark}
\newtheorem{definition}[theorem]{Definition}

\DeclareMathOperator{\cl}{cl}
\DeclareMathOperator{\interior}{int}

\newcommand{\R}{\mathbb{R}}

\newcommand{\set}[1]{\{#1\}}

\definecolor{electricgreen}{rgb}{0.0, 1.0, 0.0}

\begin{document}
\title{Integral Global Optimality Conditions and an Algorithm for  Multiobjective Problems}

\author{Everton J. Silva\thanks{Ph.D. Program in Mathematics, NOVA School Sciences and Technology (NOVA-SST) of New University of Lisbon, Center of Mathematics and Applications (CMA), 2829-516 Caparica, Portugal \texttt{(ejo.silva@campus.fct.unl.pt)}.}
  \and
Elizabeth W. Karas\thanks{Department of Mathematics, Federal University of Paran\'a, CP 19096, 81531-980, Curitiba, PR, Brazil.}
\and
 Lucelina  B. Santos{$^\dagger$}}

\maketitle
%   Paper
\begin{abstract}
In this work, we propose integral  global optimality conditions for multiobjective problems not necessarily differentiable.
The integral characterization, already known for  single objective problems, are extended to  multiobjective problems by  weighted sum and Chebyshev weighted scalarizations.
 Using this last scalarization, we propose an algorithm for obtaining an approximation of the weak Pareto front whose effectiveness is illustrated by solving a collection of multiobjective test problems.
\\[5pt]
{\bf Keywords:} Multiobjective optimization; Pareto front; Weighted sum scalarization; Chebyshev weighted scalarization; Integral global optimality conditions.\\[5pt]
{\bf AMS Classification:} 90C29, 65K05, 49M37.
\end{abstract}
%https://mathscinet.ams.org/msc/msc2020.html?t=65K05&s=49M37&btn=Search&ls=Ct

%-----------------------------------------------------------
\section{Introduction}
\label{sec:int}
%-----------------------------------------------------------
The multiobjective optimization addresses problems of Decision Making which are characterized by multiple and possibly conflicting objective functions to be optimized simultaneously on a set of feasible decisions.  Examples of these problems appear in several applications, for instance, Finance \cite{sakar2014}, Biology \cite{soleimanidamaneh2011}, Management Science \cite{ignatius2011}, Game Theory \cite{pardalos2008}, Engineering \cite{gobbi2015}, among other fields.

The first results in multiobjective optimization are due to V. Pareto, who, in his famous work \textit{``Cours d'Economie Politique''}\ \cite{pareto1886} introduced the concept of an efficient solution.
This notion of optimality has been widely used in Economics because it is closely related to the Theory of Social Welfare. After the Second World War (a time that coincides with the apogee of Operational Research), numerous studies appeared in this field. Necessary and sufficient conditions for the determination of efficient points were studied. Since then, these problems have been extensively studied in the literature, being treated both from 
%a
theoretical and applied point of view. For more historical information about this theme, see \cite{stadler1987}.

Formally, a multiobjective problem admits the following formulation:
\begin{equation}
\begin{array}{l}
\text{minimize }  F(x)=(f_{1}(x),\cdots ,f_{r}(x)) \\ \text{subject to } x\in X,
\end{array}
\label{MOP}
\tag{MOP}
\end{equation}
where $f_{\ell}:\mathbb{R}^{n}\to\mathbb{R},$ $\ell=1,\cdots ,r$, are
given functions and $X$ is a nonempty subset of $\mathbb{R}^{n}.$

Due to the conflicting nature of the objectives, an optimal solution that simultaneously minimizes all the objectives is usually not available. For vectorial functions, the minimum can be defined in terms of efficient solutions. In this paper we use the following notions of optimality:

\begin{itemize}
\item weak Pareto optimality: a point $\overline{x}\in X$ is a weak Pareto optimal  (or weakly efficient) solution of the problem \eqref{MOP} if there is no other feasible point $x\in X$ such that  $f_{\ell}(x)<f_{\ell}(\overline{x})$ for all $\ell=1,\ldots,r$.

\item Pareto optimality: a point $\overline{x}\in X$ is a Pareto optimal  (or efficient) solution of \eqref{MOP} if there is no other feasible point $x\in X$ such that $f_{\ell}(x)\leq f_{\ell}(\overline{x})$  for all $\ell=1,\ldots ,r,$ with strict inequality valid for some $\ell_{0}.$
\end{itemize}

The set of the values of all  Pareto optimal solutions to \eqref{MOP} forms the so-called  Pareto front. In this work we present integral global optimality conditions to the problem \eqref{MOP} and based on these we propose an algorithm to compute an approximation of the weak Pareto front. Other optimality characterization for  multiobjective problems  are discussed in several works. For the differentiable case, necessary first order conditions can be found in \cite{burachik2012,chankong1983,miettinen1999}; second order conditions are discussed in \cite{biggi2000,feng2019,hachimi2007,huyKimTuyen,HuyTuyen,ivanov2015,RizviNasser,santos2013,wang1984}; sufficient conditions under generalized convexity assumptions are proposed in \cite{HernandezEtAl,Luc,Osuna2,osuna1998}. Optimality conditions for non-smooth problems can be found in \cite{BhatiaAggarwal,clark1983,Nobakhtian,Staib}, for instance. Following a different approach, we present a characterization of optimality via integration, inspired  by Falk \cite{falk1973} who proposed it in 1973 for single objective problems.  As this approach only requires the continuity of the objective function and the compactness of the feasible set, it can be applied to a larger variety of problems. In the context of single objective problems,   the works \cite{cui2006,hong1988,kostrevaZheng,wu2005,zheng1985I,zheng1985II,zheng1990I,zheng1990II,zheng1991,zheng1995} also use integration techniques, with weakness hypotheses of continuity and compactness. The characterization of optimality occurs through the concepts of mean value and variance on the level sets of the objective function. From this characterization,  the authors of \cite{hong1988}  proposed an algorithm to obtain  global minimizers of single objective problems and some numerical tests were carried out to illustrate the performance of the method.

We apply these ideas to the problem \eqref{MOP} 
 by applying scalarization techniques to transform the multiobjective problem into a single objective (scalar) problem, in a such way that the solutions of the multiobjective problem can be obtained by solving a classical nonlinear programming problem. There are several techniques for scalarization of multiobjective problems. Among these methods, perhaps the best known is the weighted sum scalarization. This technique was introduced by Gass and Saaty \cite{gass1955} in 1955 and it is probably the most used due to its simplicity. The weighted sum technique is a simple way to generate different Pareto optimal solutions. The failure of this method is that not all Pareto optimal points can be found if the problem is nonconvex. Another scalarization method is the weighted Chebyshev technique, introduced by Bowman \cite{bowman1976} in 1976, which allows us to ensure that any weak Pareto optimal solution of the multiobjective problem \eqref{MOP} is solution of the weighted Chebyshev problem for some choice of weights. This fact is central for our results related to global optimality conditions for multiobjective problems.

Such results are obtained by applying these weighted scalarization techniques to the problem \eqref{MOP} and using integral global optimality conditions obtained by Cui, Wang and Zheng \cite{cui2006}, Hong and Zheng \cite{hong1988}, Wu, Cui and Zheng \cite{wu2005}, Zheng \cite{zheng1985I,zheng1985II,zheng1990I,zheng1990II,zheng1991} and Zheng and Zhuang \cite{zheng1995} to the scalarized problem. In addition, based on the integral characterization of optimality, we extend to multiobjective problems, the algorithm proposed by Hong and Zheng in  \cite{hong1988} for single objective problems. We perform  numerical experiments to illustrate the effectiveness of the proposed algorithm for solving multiobjective problems.

The paper is organized as follows. Sec. \ref{sec:tecres} recalls integral optimality conditions for scalar problems and extends them to multiobjective problems. Based on these conditions, we  propose in Sec. \ref{sec:alg} an algorithm to solve multiobjective problems and  prove its global convergence. Sec. \ref{sec:numexp} is dedicated to numerical experiments to illustrate the performance of the algorithm. Some conclusions are presented in Sec. \ref{sec:conclusion}.

%-----------------------------------------------------------
\section{Integral Global Optimality Conditions}

\label{sec:tecres}

%-----------------------------------------------------------
In this section we present  integral global optimality conditions for multiobjective problems. First, we recall integral optimality conditions for single objective problems and then we extend it to multiobjective problems, one of the main contributions of this paper.

\subsection{Single objective problem}

Consider the following single objective (scalar) optimization problem: 
\begin{equation}
\begin{array}{l}
\text{minimize }  f(x) \\ 
\text{subject to } x\in X%
\end{array}
\tag{P}  \label{P}
\end{equation}
where $f:\mathbb{R}^{n}\to\mathbb{R}$ is a given function and $X$ is a nonempty subset of $\mathbb{R}^{n}.$ In the sequence, we evoke some results on the integral characterization of global optimality for the scalar problem \eqref{P}.  First, we present a result  proposed originally by Falk   \cite{falk1973} for maximization problems and rewritten now for our context.

\begin{theorem}
\label{teo4}
\label{RIWS}
Consider $X\subset \mathbb{R}^{n}$ a
compact set with nonempty interior, $f:X\rightarrow (-\infty ,0)$ a
continuous function and $\overline{x}\in X$ such that $f(\overline{x})=-1.$
The integral  $\Upsilon (t)=\displaystyle\int_{X}[-f(x)]^{t}\, dx $ converges, when $%
t\rightarrow \infty $ if, and only if, $\overline{x}$ is a global solution
of the problem \eqref{P}.
\end{theorem}

We are interested on the global minimization of functions, not necessarily continuous.
In this context,  the concepts of level sets
and robustness are essential. The level set of the function $f$ is defined, for each real number $c$, by
\[
H_{c}=\{x\in \mathbb{R}^{n}\mid f(x)\leq c\}.
\]
The concept of robustness is a generalization of that of openness. 
A set $D\subset \mathbb{R}^{n}$ is robust if its closure coincides with the closure of its interior, $\cl(D)=\cl(\interior{D})$. Clearly, any open set $G$ is robust since $G =\interior\, G$.  On the other hand, a closed set may be nonrobust. In fact, the set with a single point is closed in $\mathbb{R}^n$, but it is nonrobust. Furthermore, the concept of the robustness of a set is closely related to its topological structure. For instance, the set $D = \{1,2\}$ is nonrobust on $\mathbb{R}^1$, but it is robust in $\mathbb{Z}$ with the discrete topology, \cite{hong1988,zheng1991}. 
Next, we have some useful properties of robust sets.  

\begin{remark}[Q. Zheng \protect\cite{zheng1991}]
\label{rem1} The following properties hold for robust sets:

\begin{enumerate}
\item The union of robust sets is robust,

\item The intersection of a robust set and an open set is robust,

\item If $D$ is robust then,  its closure $cl(D)$ is also robust.
\end{enumerate}
\end{remark}

From these concepts, we say that a function $f:\mathbb{R}^{n}\mathbb{\rightarrow R}$ is upper robust over $X$ if, and only if, the set $\{x\in X\mid f(x)<c\}$ is robust, for each real number $c$. 
 For more details on robustness, see Q. Zheng \cite{zheng1990I}-\cite{zheng1991}. From now on, we assume the following assumptions on the problem \eqref{P}:

\begin{itemize}
\item[\textsc{\textbf{A1.}}] $X$ is robust,

\item[\textsc{\textbf{A2.}}]  The function $f:X\rightarrow \mathbb{R}$ is  lower semicontinuous and upper robust,

\item[\textsc{\textbf{A3.}}] There exists $c\in \mathbb{R}$ such that $H_c\cap X$ is a compact set.
\end{itemize}

\medskip

Under these assumptions, we present some definitions  that are fundamental for the sequence of the work.

\begin{definition}\cite[Def. 5.1]{zheng1991}
\label{def1} Suppose that Assumptions \textsc{\textbf{A1}}, \textsc{\textbf{A2}} and \textsc{\textbf{A3}} hold. Consider $\overline{c}=\min\limits_{x\in X}f(x)$ and let $c>%
\overline{c}$. We define the mean value, variance and modified variance of the
function $f$ over $H_c\cap X,$ respectively,
as follows: 
\begin{eqnarray}
M(f,c,X) &=&\displaystyle\frac{1}{\mu (H_{c}\cap X)}\int_{H_{c}\cap
X}f(x)\, d\mu ,  \label{5.4} \\
V(f,c,X) &=&\displaystyle\frac{1}{\mu (H_{c}\cap X)}\int_{H_{c}\cap
X}(f(x)-M(f,c,X))^{2}\, d\mu ,  \label{5.5} \\
V_{1}(f,c,X) &=&\displaystyle\frac{1}{\mu (H_{c}\cap X)}\int_{H_{c}\cap
X}(f(x)-c)^{2}\, d\mu ,  \label{5.6}
\end{eqnarray}%
where $\mu $ denotes the Lebesgue measure in $\mathbb{R}^{n}$.
\end{definition}

\medskip

Under Assumptions \textsc{\textbf{A1}}, \textsc{\textbf{A2}} and \textsc{\textbf{A3}}, it can be proved that $\mu (H_{c}\cap X)>0$, for $c>\overline{c}$ and the function $f$ is measurable on $H_{c}\cap X$ (see \cite[Lemma 5.1]{zheng1991}). Therefore, in this case, the mean value, variance and modified variance are well defined. Furthermore, for $c=\overline{c}$, these definitions can be extended by a limit process as follows.

\begin{definition}\cite[Def. 5.2]{zheng1991}
\label{def2} Under the assumptions of Definition \ref{def1}, we can extend
it to $c\geq \overline{c}$ by: 
\begin{eqnarray*}
M(f,c,X) &=&\displaystyle\lim_{c_{k}\downarrow c}\frac{1}{\mu (H_{c_{k}}\cap
X)}\int_{H_{c_{k}}\cap X}f(x)\, d\mu, \\
V(f,c,X) &=&\displaystyle\lim_{c_{k}\downarrow {c}}\frac{1}{\mu
(H_{c_{k}}\cap X)}\int_{H_{c_{k}}\cap X}(f(x)-M(f,c,X))^{2}\, d\mu, \\
V_{1}(f,c,X) &=&\displaystyle\lim_{c_{k}\downarrow c}\frac{1}{\mu
(H_{c_{k}}\cap X)}\int_{H_{c_{k}}\cap X}(f(x)-c)^{2}\, d\mu .
\end{eqnarray*}
\end{definition}

According to \cite{hong1988}, under the assumptions, these limits exist and are independent of choices of the decreasing sequence $\{c_{k}\}$. With these concepts we can characterize the integral global optimality for the problem \eqref{P} as
follows.

\begin{theorem}
\label{teo5} \cite[Thm. 5.1]{zheng1991} Suppose that Assumptions \textsc{\textbf{A1}}, \textsc{\textbf{A2}} and \textsc{\textbf{A3}} hold. The following statements are equivalent:

\begin{itemize}
\item[(i)] $\overline{x}\in X$ is a global minimizer of \eqref{P} and $\overline{c}=f(\overline{x})$ is the global minimum value  of $f$ over $X$,

\item[(ii)] $M(f,\overline{c},X)=\overline{c}$,

\item[(iii)] $V(f,\overline{c},X)=0$,

\item[(iv)] $V_{1}(f,\overline{c},X)=0$.
\end{itemize}
\end{theorem}

\medskip

Next, we will extend these integral characterizations for global optimality to the multiobjective problem (MOP).

\subsection{Multiobjective Problem}

In this section we return our attention to the multiobjective problem \eqref{MOP} to extend the results seen in last section. We assume that $F$ is a continuous function and $X\subset \mathbb{R}^{n}$ is a 
compact set with nonempty interior.

First, we recall some results regarding the scalarization of multiobjective
problems. Define the sets of weighting vectors
\begin{equation}
W =\{w\in \mathbb{R}^{r}\mid w_{\ell}\geq 0,\, \ell=1,\ldots,r\text{ and }
\|w\|_{1}=1\}  \label{w}
\end{equation}
and 
\begin{equation}
W^* =\{w\in \mathbb{R}^{r}\mid w_{\ell}>0,\, \ell=1,\ldots,r\text{ and }
\|w\|_{1}=1\},  \label{wset}
\end{equation}
where $\|w\|_{1}=\displaystyle\sum_{\ell=1}^{r}|w_{\ell}|$. For each $w\in W$%
, we define the weighted sum scalarization function $\Phi_w:\mathbb{R}^n\to%
\mathbb{R}$ by 
\begin{equation}  \label{Phi}
\Phi _{w}(x) =\displaystyle\sum_{\ell=1}^{r}w_{\ell}f_{\ell}(x)
\end{equation}
and we consider the following weighted sum  problem: 
\begin{equation}
\begin{array}{l}
\text{minimize } \Phi _{w}(x) \\ 
\text{subject to }x\in X.%
\end{array}
\tag{$WS_{w}$}  \label{WS}
\end{equation}
The connections between the solutions of the weighted sum problem %
\eqref{WS} and the (weak) Pareto optimal solutions of the problem \eqref{MOP} are given in the following theorems.

\begin{theorem} \cite[Thm. 3.1.1 and 3.1.2]{miettinen1999}
\label{teo1} If there exists $w\in W$ (respectively, $w\in W^*$) such that $\overline{x}\in X$ is a
solution of \eqref{WS} then $\overline{x}$ is a weak Pareto optimal solution (respectively, Pareto optimal solution) of \eqref{MOP}.
\end{theorem}

\medskip Now we will define the weighted Chebyshev scalarization. For that,
let $F^*\in \mathbb{R}^{r}$ be the ideal objective vector, where its
components $f_{\ell}^*$ are obtained by minimizing each objective function
individually subject to the constraints, that is, for each $%
\ell=1,\ldots,r,$ 
\begin{equation}
f_{\ell}^*=\min\limits_{x\in X}f_{\ell}(x).  \label{1}
\end{equation}
If there exists $\overline{x}\in X$, such that $F(\overline{x})=F^*$, then $\overline{x}$ would be a solution of the multiobjective problem \eqref{MOP} and the Pareto optimal set would be reduced to it. In general, the ideal objective vector can be used as a lower bound for the objective function at the Pareto optimal set. Now, given $\xi\in\mathbb{R}^r_{+}$ with small positive components, we consider the utopian objective vector  $u^\ast=F^\ast-\xi$ and for $w\in W$, we define the weighted Chebyshev scalar function $\Psi_w:\mathbb{R}^n\to\mathbb{R}$ by 
\begin{equation}  \label{Psi}
\Psi_{w}(x) =\displaystyle\max_{\ell=1,\ldots,r}
\{w_{\ell}(f_{\ell}(x)-{u_\ell^\ast})\} 
\end{equation}
and we solve the following problem: 
\begin{equation}
\begin{array}{l}
\text{minimize } \Psi_{w}(x) \\ 
\text{subject to } x\in X.%
\end{array}
\tag{$WCS_{w}$}  \label{WCS}
\end{equation}

The convexity (or generalized convexity) of the multiobjective optimization problem \eqref{MOP} is sufficient to ensure that all Pareto optimal solutions can be found using the weighted sum scalarization. See Theorem 3.1.4 in \cite{miettinen1999}, Lemma 2 in \cite{wang1984} and Theorems 3.2 and 3.3 in \cite{Osuna2}. On the other hand, next theorem shows that all weak Pareto optimal solutions can be found by the weighted Chebyshev technique, without any additional hypotheses. 
 
\begin{theorem}\cite[Thm. 3.4.2 and 3.4.5]{miettinen1999}
\label{teo3}
The point $\overline{x}\in X$ is a weak Pareto optimal solution of the multiobjective problem \eqref{MOP} if, and only if, $\overline{x}$ is a solution of \eqref{WCS} for some weighting vector $w\in W^*$.
\end{theorem}

\medskip

It is interesting to note that if, for $w\in W^\ast$, the problem \eqref{WCS} has a unique solution, then it will be a Pareto optimal point \cite[Cor. 3.4.4]{miettinen1999}. In addition, if the set of Pareto solutions is uniformly dominant\footnote{The efficient set is uniformly dominant if for every non-efficient point $x^\prime$ there exists an efficient point $x^\ast$ such that $f_\ell(x^\prime)>f_\ell(x^\ast)$ for all $\ell=1,\ldots,r$.}, then every Pareto point can be obtained through the Chebyshev scalarization \cite[Thm. 3 and 4]{bowman1976}. For more details on scalarization methods, see Chankong and Haimes \cite{chankong1983}, Jahn \cite{jahn2011} and Miettinen \cite{miettinen1999}.

Now, we will present integral characterizations of global optimality for multiobjective problems \eqref{MOP} from these scalarization techniques. As $F$ is a continuous function on the compact set $X$, the functions $\Phi _{w}$ and $\Psi _{w}$, defined by \eqref{Phi} and \eqref{Psi}, respectively, are continuous. From Weierstrass Theorem, it follows that there exist constants $M_{1}$ and $M_{2}$ such that $\Phi_{w}(x)< M_{1}$ and $\Psi_{w}(x)< M_{2}$ for all $x\in X$. Define the functions, for $x\in\mathbb{R}^n$, by
\begin{eqnarray*}
\widetilde{\Phi}_{w}(x) &=&\Phi _{w}(x)-M_{1} \\
\widetilde{\Psi}_{w}(x) &=&\Psi _{w}(x)-M_{2}.
\end{eqnarray*}
These functions are continuous on $X$ and $\widetilde{\Phi}_{w}(x), 
\widetilde{\Psi}_{w}(x)< 0$ for all $x\in X$. 

\begin{remark}
\label{rem3}
As a consequence, a point  $\overline{x}\in X$ is a global minimizer of $\Phi _{w}$ over $X$ if, and only if, $\overline{x}$ minimizes the function $x\mapsto -\frac{\widetilde{\Phi}%
_{w}(x)}{\widetilde{\Phi }_{w}(\overline{x})}$ on $X$. A similar result holds for the
function ${\Psi}_{w}.$
\end{remark}

Now we state the results inherited from  last section by the application of the weighted scalarization techniques to the problem \eqref{MOP}.

\begin{theorem}
\label{teo6}
Consider $\overline{x}\in X$ and  $w\in W$ (respectively, $w\in W^{\ast })$. If $\Upsilon_w(t)=\displaystyle\int_{X}\left[\frac{\widetilde{\Phi}_{w}(x)}{\widetilde{\Phi }_{w}(\overline{x})}\right]^{t}d\mu $ converges as $t\rightarrow \infty$, then $\overline{x}$ is a weak Pareto optimal solution
(respectively,  Pareto optimal solution) for the problem \eqref{MOP}.

\end{theorem}
\begin{proof}
The set $X$ is compact and the function $x\mapsto - \frac{\widetilde{\Phi}_w(x)}{\widetilde{\Phi}_w(\bar{x})}$ satisfies the hypotheses of Theorem  \ref{RIWS}.
Thus,  $\overline{x}\in X$
is a global minimizer of this function over $X$, and consequently, a global solution  of the problem \eqref{WS}, 
by Remark \ref{rem3}. So, the result follows from  Theorem \ref{teo1}.
\end{proof}

\begin{theorem}
\label{teo7} 
A point $\overline{x}\in X$ is a weak Pareto optimal solution of \eqref{MOP} if, and only if, there exists $w\in W^*$ such that the function defined by $\Upsilon_w(t)=\displaystyle\int_{X}\left[ \frac{\widetilde{\Psi }_{w}(x)}{\widetilde{\Psi}_{w}(\overline{x})}\right]^{t} d\mu$ converges when $t\to \infty$.
\end{theorem}

\begin{proof}
By Theorem \ref{teo3}, $\overline{x}\in X$ is a weak Pareto solution of \eqref{MOP} if, and only if, there exists $\overline{w}\in W^*$ such that $\overline{x}$ is a solution of \eqref{WCS},  which is equivalent to 
\begin{equation}
\label{aux1}
\widetilde\Psi_{\overline{w}}(\overline{x})\leq \widetilde\Psi_{\overline{w}}(x)\, \text{ for all } x \in X.
\end{equation}
Since
$\widetilde{\Psi}_{\overline{w}}(x)<0$, for all $x\in X$, \eqref{aux1} is equivalent to say that $\overline{x}$ is a global minimizer, in $X$, of the function $x\mapsto -\dfrac{\widetilde{\Psi}_{\overline{w}}(x)}{\widetilde{\Psi}_{\overline{w}}(\overline{x})}$. As this function satisfies the hypotheses of Theorem \ref{RIWS},  the proof is concluded.
\end{proof}

% ---------
Now we will discuss optimality conditions for the multiobjective problem \eqref{MOP} from the concepts of mean, variance and modified variance. In particular, the next theorem establishes global  optimality necessary conditions to the problem \eqref{MOP} using the weighted sum scalarization.

\begin{theorem}
\label{mvaluevariance_Phi} Assume that Assumption \textsc{\textbf{A1}} holds. Suppose that
 there exists  $w\in W$ (respectively, $w\in W^\ast$) such that the function $\Phi_w$ satisfies  Assumptions \textsc{\textbf{A2}} and \textsc{\textbf{A3}}. 
 Consider $\overline{x}\in X$ and  $\overline{c}=\Phi _{w}(\overline{x})$. Then  the following conditions are equivalent:
\begin{itemize}
\item[(i)] $\overline{x}\in X$ is a
solution of the problem \eqref{WS},

\item[(ii)]  $M(\Phi _{w},\overline{c},X)=%
\overline{c}$,

\item[(iii)] $V(\Phi _{w},\overline{c}%
,X)=0$,

\item[(iv)]  $V_{1}(\Phi _{w},\overline{c},X)=0$,
\end{itemize}
where 
$M$, $V$ and $V_1$ are, respectively, the mean value, variance and modified variance of $\Phi_{w}$.
Moreover, in these equivalent situations, $\overline{x}$ is a weak Pareto optimal solution (respectively, Pareto optimal solution) of \eqref{MOP}.
\end{theorem}
\begin{proof}
The result is an immediate consequence of Theorems \ref{teo5} and \ref{teo1}.
\end{proof}
%===========

Analogous result holds for the weighted Chebyshev scalarization \eqref{Psi}. However, for this scalarization, we have stronger global optimality conditions by considering  the following assumptions:

\begin{description}
\item[\textsc{A1}$^{\prime }.$] $X$ is a robust and closed set,
\item[\textsc{A2}$^{\prime }.$] The functions $f_{\ell},$ $\ell=1,\cdots ,r$, are continuous,
\item[\textsc{A3}$^{\prime }.$] There exist an index $\ell_{0}$ and  $c_0\in\mathbb{R}$
such that the set $\{x\in X\mid f_{\ell_{0}}(x)\leq c_{0}\}$ is compact.
\end{description}

Next proposition ensures that if the problem \eqref{MOP} satisfies these assumptions, then  \textbf{A1, A2} and \textbf{A3} hold for weighted sum problem \eqref{WS} and weighted Chebyshev problem \eqref{WCS},  for all $w\in W^{\ast}.$

\begin{proposition}
\label{prop} Suppose that \textsc{\textbf{A1}}$^{\prime}$, \textsc{\textbf{A2}}$^{\prime }$ and \textsc{\textbf{A3}}$^{\prime }$ hold. Then, \textsc{\textbf{A1}} holds and for all $w\in W^{\ast }$, the functions $\Phi_w$ and $\Psi _{w}$ satisfies Assumptions \textsc{\textbf{A2}} and \textsc{\textbf{A3}}.
\end{proposition}

\begin{proof}
Assumption \textsc{\textbf{A1}} follows trivially from \textsc{\textbf{A1$^{\prime }$}}. Consider $w\in W^{\ast}$. Using the Assumption \textsc{\textbf{A2$^{\prime}$}}, the functions $\Phi_w$ and $\Psi_w$ are continuous. Furthermore, for each $c\in \mathbb{R}$, the sets $\{x\in \mathbb{R}^{n}\mid\Phi_{w}(x)<c\}$ and $\{x\in \mathbb{R}^{n}\mid\Psi_{w}(x)<c\}$ are open. By  Assumption \textsc{\textbf{A1$^{\prime }$}} and Remark \ref{rem1}, their intersections with $X$ are robust. Consequently, $\Phi_w$ and $\Psi_{w}$ are upper robust functions and \textsc{\textbf{A2}} holds. 

For each  $c\in \mathbb{R}$, consider the level set  $H_c=\{x\in \mathbb{R}^n\mid\Phi_{w}(x)\leq c\}$. Assumption \textsc{\textbf{A1}}$^\prime$ implies that $ H_c\cap X$ is a  closed set. Furthermore, as $w\in W^{\ast}$, we have,  in particular to $\ell_0$ given in Assumption  \textsc{\textbf{A3}}$^\prime$, that
$$H_c\cap X  = \left\{x\in X\mid \displaystyle\sum_{\ell=1}^{r}w_\ell f_\ell (x)\leq c\right\}=\left\{x\in X\mid f_{\ell_0}(x)\leq \dfrac{1}{w_{\ell_0}}\left(c-\sum_{\ell=1 \atop \ell\neq \ell_0}^{r} w_\ell f_\ell(x)\right)\right\}.$$

Taking $c=w_{\ell_0}c_0 +\displaystyle \sum_{\ell=1 \atop \ell\neq \ell_0}^{r} w_\ell f_\ell(x)$, with $c_0$ given in Assumption \textsc{\textbf{A3}}$ ^\prime$, the set $H_c\cap X$ is compact.

Analogously, for each $c\in\mathbb{R}$, consider the level set $H_c=\{x\in \mathbb{R}^n\mid\Psi _{w}(x)\leq c\}$.  Assumption \textsc{\textbf{A1}}$^\prime$ implies that $H_c\cap X$ is a  closed set. Furthermore, as $w\in W^{\ast}$, we have,  in particular to $\ell_0$ given in Assumption  \textsc{\textbf{A3}}$^\prime$, that  $$ H_c\cap X \subset  \{x\in X\mid w_{\ell_{0}}(f_{\ell_{0}}(x)-u_{\ell_{0}}^{\ast })\leq c\}=\left\{x\in X\mid f_{\ell_{0}}(x)\leq \frac{c+w_{\ell_{0}}u_{\ell_{0}}^{\ast }}{w_{\ell_{0}}}\right\}.$$ 
Taking $c=w_{\ell_{0}}(c_{0}-u_{\ell_{0}}^{\ast})$, with  $c_0$ given in Assumption \textsc{\textbf{A3}}$ ^\prime$, the  set $H_c\cap X$ is compact, which proves  \textsc{\textbf{A3}} for both functions and concludes the proof. 
\end{proof}

Next theorem ensures necessary and sufficient global optimality conditions of \eqref{MOP} using the weighted Chebyshev scalarization \eqref{Psi}, while Theorem \ref{mvaluevariance_Phi} establishes only necessary  conditions for the weighted sum scalarization.

\begin{theorem}
\label{mvaluevariance} Suppose that the problem \eqref{MOP} satisfies \textsc{\textbf{A1$^{\prime }$, A2}}$^{\prime }$ and \textsc{\textbf{A3}}$^{\prime }$. Consider $\overline{x}\in X$. Then, the following conditions are equivalent:
\begin{itemize}
\item[(i)] $\overline{x}$ is a weak Pareto optimal solution of (\ref{MOP}),

\item[(ii)] there exists $w\in W^{\ast }$ such that $\overline{x}$ minimizes 
$\Psi _{w}$ over  $X$ and $\overline{c}=\Psi _{w}(\overline{x})$,

\item[(iii)] there exists $w\in W^{\ast }$ such that $M(\Psi _{w},\overline{c},X)=%
\overline{c}$, with $\overline{c}=\Psi _{w}(\overline{x})$,

\item[(iv)] there exists $w\in W^{\ast }$ such that $V(\Psi _{w},\overline{c}%
,X)=0$, with $\overline{c}=\Psi _{w}(\overline{x})$,

\item[(v)] there exists $w\in W^{\ast }$ such that $V_{1}(\Psi _{w},\overline{c},X)=0$,
with $\overline{c}=\Psi _{w}(\overline{x})$,
\end{itemize}
where $M$, $V$ and $V_1$ are, respectively, the mean value, variance and modified variance of $\Psi_{w}$.
\end{theorem}
\begin{proof}
The result is an immediate consequence of Proposition \ref{prop} and Theorems \ref{teo5} and \ref{teo3}.
\end{proof}

It is important to note that the result of 
Theorem \ref{mvaluevariance} holds under more general conditions. In fact, it is enough that $X$ is robust and $\Psi_{w}$ satisfies \textsc{\textbf{A2}} and  \textsc{\textbf{A3}}, for $w\in W^{\ast }$. Based on Theorem \ref{mvaluevariance},  we extend the algorithm proposed in \cite{hong1988} (originally to solve single objective problems) for obtaining an approximation of the weak Pareto front of the multiobjective problem \eqref{MOP}.

%----------------------------------------------------------
\section{The algorithm}
\label{sec:alg}
%----------------------------------------------------------
Now, inspired by \cite{hong1988},  we state an  algorithm based on the mean value of level sets for multiobjective problems and  we discuss   its global convergence regarding the scalarized problem \eqref{WCS}.

\medskip

%======================
\begin{algorithm}[H]
\label{MVLSM} 
  \SetAlgoLined
  \DontPrintSemicolon
  \SetNlSty{textbf}{}{.}
  \SetAlgoCaptionSeparator{.}
  \caption{Mean Value of Level Sets for Multiobjective Problems -- MVLSM}
\begin{tabbing}
\hspace*{7mm}\= \hspace*{7mm}\= \hspace*{7mm}\= \hspace*{7mm}\=
 \hspace*{7mm}\= \hspace*{7mm}\= \kill
Data:   
$\varepsilon \geq  0$, $\overline{w}\in\mathbb{R}^r_+$ with $\overline{w}_i\in (0,1)$ for all $i=1,\ldots,r$ and $\xi\in\mathbb{R}^{r}_{+}\backslash\set{0}$. \\
\textbf{Scalarization}\+ \\
Compute
$
\displaystyle{f^*_\ell=\min_{x\in X} \{f_\ell(x)\}}$, for each $\ell\in \{1,\ldots, r\}$ { and define $u^\ast=F^\ast-\xi$}. \\
Consider the weighting vector  $w\in\R^r$ such that $w_i=\dfrac{\overline{w}_i}{\|\overline{w}\|_1}$, for $i=1,\ldots,r$, and \\
the scalarized function $\Psi_w(x)= \displaystyle\max_{\ell}\{w_\ell(f_\ell(x)-{ u_\ell^*})\}$. \- \\
\textbf{Initialization}\+ \\
Take  $c_0\in \mathbb{R}$ such that $ H_{c_0}=\{x\in \mathbb{R}^n \mid \Psi_w(x)\leq c_0 \}\neq \emptyset.$ \\

%\textbf{Modified variance} \+ \\
Set $k:=0$. \- \\
\textbf{Iterations}\+ \\
{\sc repeat}  \+ \\
%\textbf{Mean Value}\+ \\
Let $H_{c_{k}}= \{x\in \mathbb{R}^n\mid \Psi_w(x)\leq c_{k}\}.$ \\
Compute
$VF=V_1(\Psi_w,c_{k},X)=\dfrac{1}{\mu(H_{c_k}\cap X)}\displaystyle\int_{H_{c_k}\cap X} (\Psi_w(x)-c_{k})^2\, d\mu.$ \\
Compute
$c_{k+1}=M(\Psi_w,c_{k},X)=\dfrac{1}{\mu(H_{c_{k}}\cap X)}\displaystyle\int_{H_{c_{k}}\cap X} \Psi_w(x)\, d\mu$. \\
$k: = k+1$  \- \\
{\sc until} $VF < \varepsilon$ \\
 $\overline{c}=c_{k}$ and  $\overline{H}=H_{\overline{c}}$
%\textbf{Stopping criterion}\+ \\
%{\sc if} $VF<\varepsilon$, then $\hat{c}=c_k$, $\hat{H}=H_k$, and {\sc stop}
\end{tabbing}
\end{algorithm}
%-------- end of algorithm

\medskip

The scalar $c_0$ can be chosen as any real such that the  set $H_{c_0}\cap X$ is nonempty. So, it can be set as a sufficiently large real or as $c_0=\Psi_w(x_0)$, for a given initial point $x_0\in X$. The stopping criterion of the algorithm is justified by Theorem \ref{mvaluevariance}, item ($v$). From now on, assume that  $\varepsilon=0$ and the algorithm generates an infinite sequence $\{ c_k\}$. Next theorem ensures that this sequence converges 
to the global minimum value  
of the scalarized function $\Psi_w$ over $X$.

\medskip

\begin{theorem}
Suppose that the problem \eqref{MOP} satisfies \textsc{\textbf{A1$^{\prime}$, A2}}$^{\prime }$ and \textsc{\textbf{A3}}$^{\prime }$.
Given a weighting vector  $w\in W^{*}$,
consider the sequence $\{ c_k\}$  generated by Algorithm \ref{MVLSM}. Then, this sequence is convergent and the limit  $\overline{c}=\displaystyle\lim_{k\to\infty} c_k$ is the global minimum value of $\Psi_w$ over $X$. Furthermore, $H_{\overline{c}}\cap X$ is the set of its global minimizers and consequently a subset of  weak Pareto optimal solutions of \eqref{MOP}.
\end{theorem}

\begin{proof}
Let $w\in W^*$,  $\hat{c}=\displaystyle\min_{x\in X} \Psi_w(x)$ and  the sequence $\{c_k\}$ generated by the algorithm from $c_0$
such that $ H_{c_0}=\{x\in \mathbb{R}^n \mid \Psi_w(x)\leq c_0 \}\neq \emptyset.$ If $c_0=\hat{c}$, then $VF=V_1(\Psi_w,c_0,X)=0$ and the algorithm stops. Now, consider $c_0>\hat{c}$.
In this case, for all $x\in H_{c_0}\cap X$, $\hat{c}\leq \Psi_{w}(x)\leq c_0$. Integrating this expression and using the definition of $c_1$ and the fact that, by \cite[Lemma 5.1]{zheng1991}, $\mu(H_{c_0}\cap X)>0$, we have
%{\blue 
%\begin{eqnarray*}	c_{1}=M(\Psi_{w},c_0,X)&=&\dfrac{1}{\mu(H_{c_0}\cap X)} \int_{H_{c_0}\cap X} \Psi_w(x)\, d\mu \\ &\leq& \dfrac{1}{\mu(H_{c_0}\cap X)}\int_{H_{c_0}\cap X} c_0\, d\mu\\ &=& \dfrac{1}{\mu(H_{c_0}\cap X)} \,c_0\, \mu(H_{c_0}\cap X) = c_0. \end{eqnarray*}
%\begin{eqnarray*}	\hat{c}\leq \dfrac{1}{\mu(H_{c_0}\cap X)} \int_{H_{c_0}\cap X} %\Psi_w(x)\, d\mu \leq \dfrac{1}{\mu(H_{c_0}\cap X)}\int_{H_{c_0}\cap X} c_0\, d\mu
%= c_0.
%\end{eqnarray*}
%}
 $\hat{c}\leq c_1\leq c_0$. Following a similar reasoning we can conclude that $\hat{c}\leq c_{k+1}\leq c_k$ for all $k\geq 0$.
% By the construction of the algorithm,  for all $k\geq 0$, $$c_{k+1}=M(\Psi_w,c_k,X),$$ which satisfies 
%In fact, we have that $\Psi_{w}(x)\leq c_k$ for all $x\in H_{c_k}\cap X$, thus
%\begin{eqnarray*}
%c_{k+1}=M(\Psi_{w},c_k,X)&=&\dfrac{1}{\mu(H_{c_k}\cap X)} \int_{H_{c_k}\cap X} \Psi_w(x)\, d\mu \\
%&\leq& \dfrac{1}{\mu(H_{c_k}\cap X)}\int_{H_{c_k}\cap X} c_k\, d\mu\\
%&=& \dfrac{1}{\mu(H_{c_k}\cap X)} c_k \mu(H_{c_k}\cap X) = c_k.
%\end{eqnarray*}

If there exists $k_0\in \mathbb{N}$ such that $c_{k_0}=M(\Psi_w,c_{k_0},X)$, then, by Theorem \ref{mvaluevariance}, $V_1(\Psi,c_{k_0},X)=0$ and the algorithm stops with $\hat{c}=c_{k_0}$. In this case,  $H_{c_{k_0}}\cap X$ is the set of global minimizers of $\Psi_w$ and, by Theorem \ref{teo3}, it is a subset of Pareto optimal solutions of (MOP). 

Otherwise, 
%we get a 
%as $c_{k+1}$ is the mean value of $\Psi_w$ in the level set $H_{c_k}\cap X$, it follows that, for all $k\geq 0$, $\hat{c} \leq c_{k+1}< c_k$. So, the decrease 
the sequence $\{ c_k\}$ of mean values is decreasing and bounded below, and consequently convergent, say to $\overline{c} \geq \hat{c}$. Thus, by the convergence of the sequence $\{c_k\}$ and the continuity of the function $M$ with respect to the second argument (more details about the continuity of function $M$ can be found in \cite[Prop. 1.3]{hong1988}), we have $$\overline{c}=M(\Psi_w,\overline{c},X).$$ Applying Theorem \ref{mvaluevariance}, we conclude that $\overline{c}=\hat{c}$ is the global minimum value of $\Psi_w$ over $X$. On the other hand, as the sequence $\{c_k\}$ is decreasing and bounded below by $\overline{c}$, it follows that the
%monotone 
sequence $\{ H_{c_k}\}$ of level sets satisfies the following condition
$$ 	(H_{c_0}\cap X)\supset (H_{c_1}\cap X)\supset\ldots\supset (H_{c_k}\cap X)\supset (H_{c_{k+1}}\cap X)\supset \ldots\supset (H_{\overline{c}}\cap X).$$
%	$$
%	H_{c_0}\supset H_{c_1}\supset\ldots\supset H_{c_k}\supset H_{c_{k+1}}\supset \ldots\supset H_{\overline{c}}.
%	$$
This fact implies
	%Using this and the  convergence of the sequence $\{c_k\}$ to $\overline{c}$, we have
	%Moreover, we also have
	$$
	%\displaystyle\lim_{k \to \infty} (H_{c_k}\cap X)=
	\bigcap_{k=1}^{\infty} (H_{c_k}\cap X)=H_{\overline{c}}\cap X=
	%\overline{H}=
	\{x\in X \mid \Psi_w(x)=\overline{c}\},$$
	which proves $H_{\overline{c}}\cap X$ is 
	the set of  global minimizers 
	of $\Psi_w$ over $X$. Furthermore, 
	by Theorem \ref{mvaluevariance}, this set is a subset of weak Pareto optimal solutions of \eqref{MOP}, completing the proof.
\end{proof}

\medskip

For each vector $w\in\R^r$, we  obtain a subset of weak Pareto optimal solutions of \eqref{MOP}. So, to determine an approximation of the weak Pareto front of the multiobjective problem \eqref{MOP}, the algorithm should be run several times using different weights.

%-----------------------------------------------------------
\section{Numerical experiments}
\label{sec:numexp}

In this section, we describe numerical experiments to illustrate the computational performance of Algorithm \ref{MVLSM}. The tests were performed in a high performance workstation MARKOV: 2*CPU: Intel{\small \textregistered}\ Xeon{\small\textregistered}\ Processor E5-2650 v3 (10 Cores 25M Cache, 2.30 GHz), 160GB RAM\@ 2,133GHz, using {\tt Matlab} 2018b. The set of test problems consists of  {\bf all} $26$ unconstrained and box-constrained multiobjective problems with continuous variable of dimension at most $4$ presented in \cite{chankong1983,deb2001,deb2005,veld1999}. 

In Algorithm \ref{MVLSM}, the random weighting vector $\bar{w}\in\R^r$ has been computed by the {\sc rand Matlab} routine,  the initial mean value has been set as $c_0=10^{8}$, the stopping tolerance as  $\varepsilon=10^{-8}$ and $\xi_\ell=10^{-4}$ for all $\ell=1,\ldots,r$. The multiple integrals in the modified variance $VF$ and in the mean value $c_k$  were computed by nested commands of the {\sc trapz Matlab} routine. The domain of integration was discretized in $10000$ points uniformly distributed.  

Initially, we run 3000 times  Algorithm \ref{MVLSM} considering different weighting random vectors for solving each problem. Tables \ref{table1} - \ref{table7} show the results where the first column displays the data of the problems such as references, dimension $n$,  number $r$ of objectives and some results as the average $T$ of the CPU time  and the average $\bar{k}$ of the number of iterations  among  the $3000$ runs for each problem. As the dimension of the problems presented in Tables \ref{table1} - \ref{table5} is less than $3$, we show,  in the second column, the graph of the objective functions. The last column of all tables presents the approximation of the weak Pareto front generated  from the total of runs of the algorithm and the exact Pareto front is shown whenever its analytical expression is available. These figures illustrate the good performance of Algorithm \ref{MVLSM} that found a good approximation of the weak Pareto front for all $26$ problems spending in average $T=0.0518$ sec and $25$ iterations. The longest CPU time was $0.3112$ sec and the largest number of iterations was, $129$ spent for solving \cite[Example 9]{chankong1983} and \cite[Problem 4.7]{chankong1983}, respectively.

Secondly, we compared the performance of MVLSM (Mean Value of Level Sets for Multiobjective Problems) as described in Algorithm \ref{MVLSM} for solving the 26 problems with two solvers from the literature, namely: 
\begin{itemize}
%\item MVLSM (Mean Value of Level Sets for Multiobjective Problems), Algorithm \ref{MVLSM};
\item DMS (Direct Multisearch) proposed by Cust\'odio, Madeira, Vaz, and Vicente in  \cite{custodio2011} and freely available at \emph{http://www.mat.uc.pt/dms};
\item MOIF (Multiobjecticve Implicit Filtering) proposed by Cocchi, Liuzzi, Papini, and Sciandrone in  \cite{cocchi2018} and freely available at \emph{http://www.dis.uniroma1.it/lucidi/DFL}.
\end{itemize}
The solvers DMS and MOIF have been tested using their default parameters. In these experiments, we fixed the maximum function evaluations as $20000$ for each algorithm for solving each problem. Figure \ref{purity} shows the performance profile \cite{dolan2002} using the purity metric \cite{pal2004} which compares the quality of Pareto fronts obtained by different solvers. Although the solver DMS is slightly more efficient, the three algorithms are competitive.

% purity 
\begin{figure}[h!]
\centering
\subfigure[MVLSM versus DMS]{\includegraphics[scale=0.3]{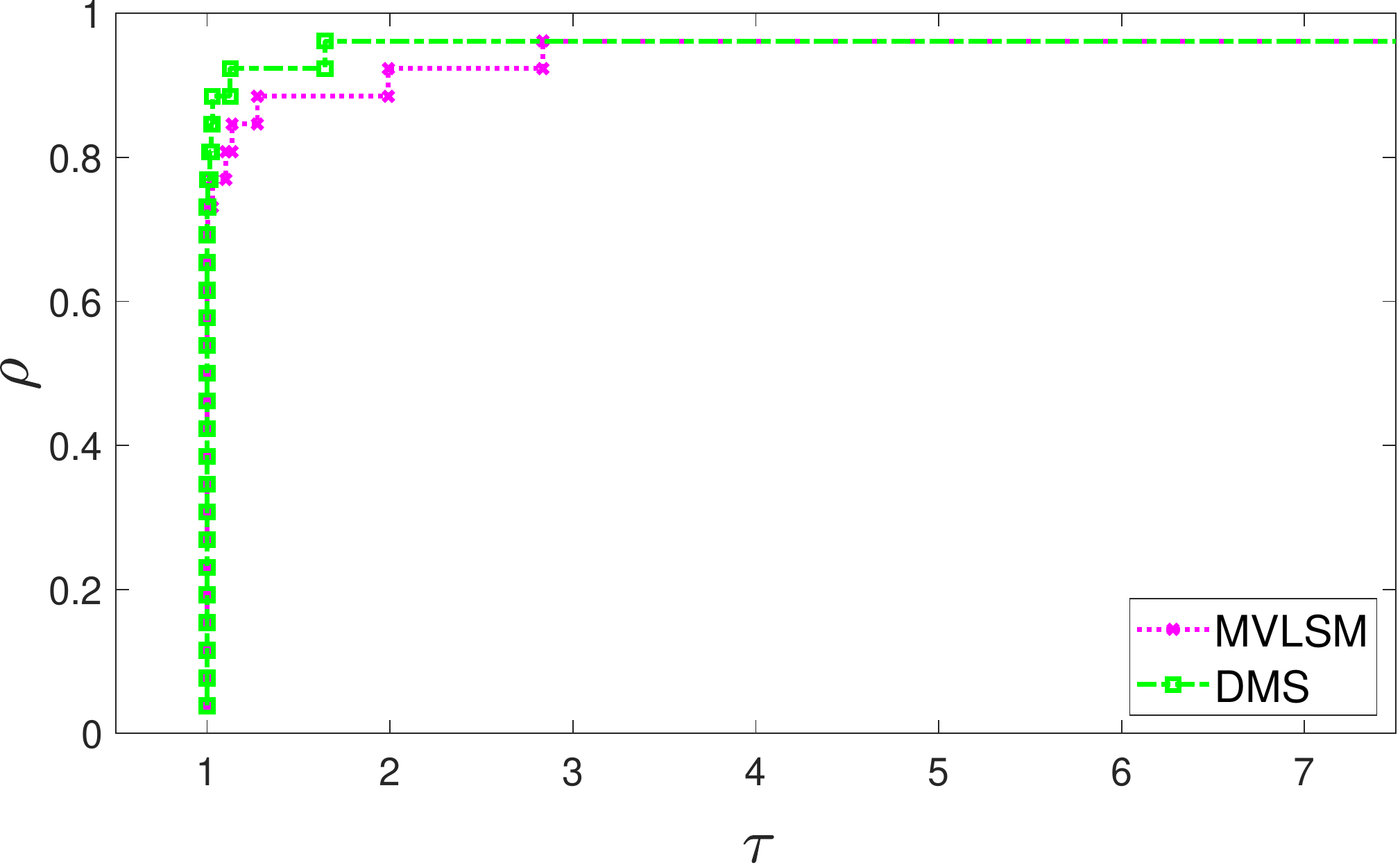}}
\subfigure[MVLSM versus MOIF]{\includegraphics[scale=0.3]{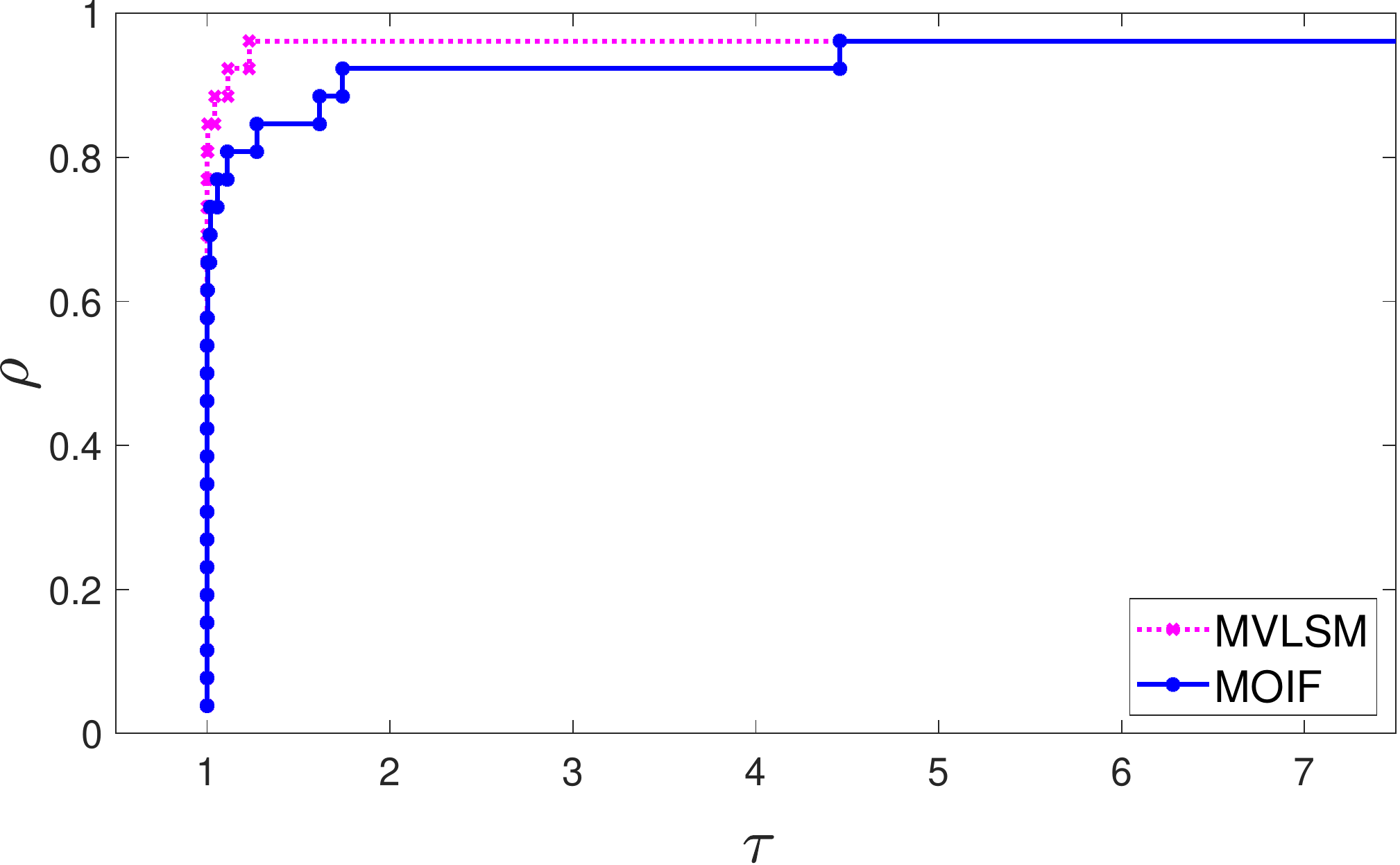}}
\subfigure[All]{\includegraphics[scale=0.3]{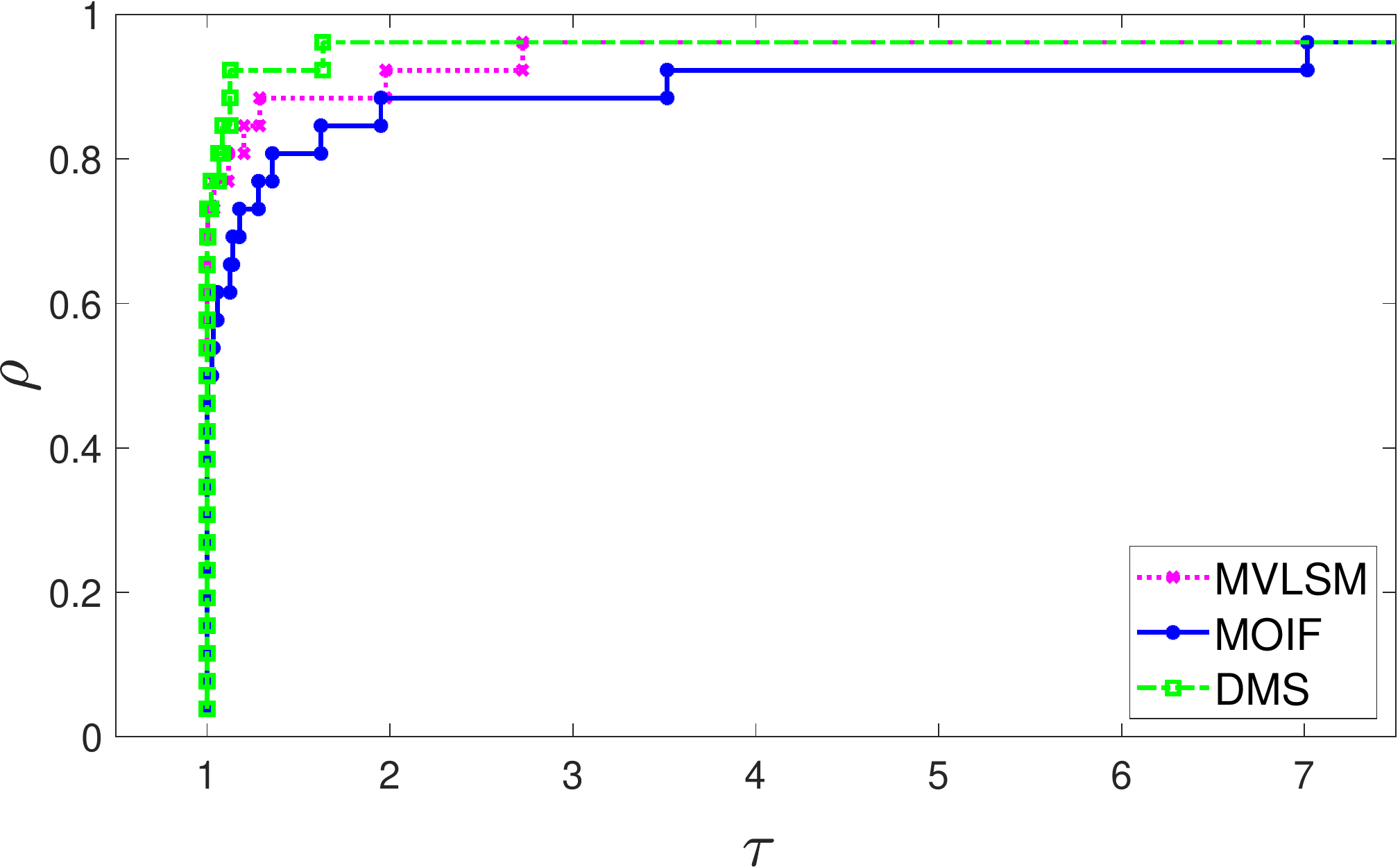}}
\caption{Comparing the MVLSM with DMS and MOIF based on purity performance profiles for multiobjective problems.}
\label{purity}
\end{figure}

Figure \ref{hypervolume} shows the performance profile using the hypervolume indicator   which represents the volume in the objective space dominated by a Pareto front approximation $Y_N$ and delimited above by an objective vector $v\in\mathbb{R}^r$ such that for all $y\in Y_N$, we have that $y < v$, as explained in \cite{zitzler1998}. According to this figure, DMS is more efficient than the other solvers. However, MVLSM is the most robust ones.

% Hypervolume
\begin{figure}[h!]
\centering
\subfigure[MVLSM versus DMS.]{\includegraphics[scale=0.3]{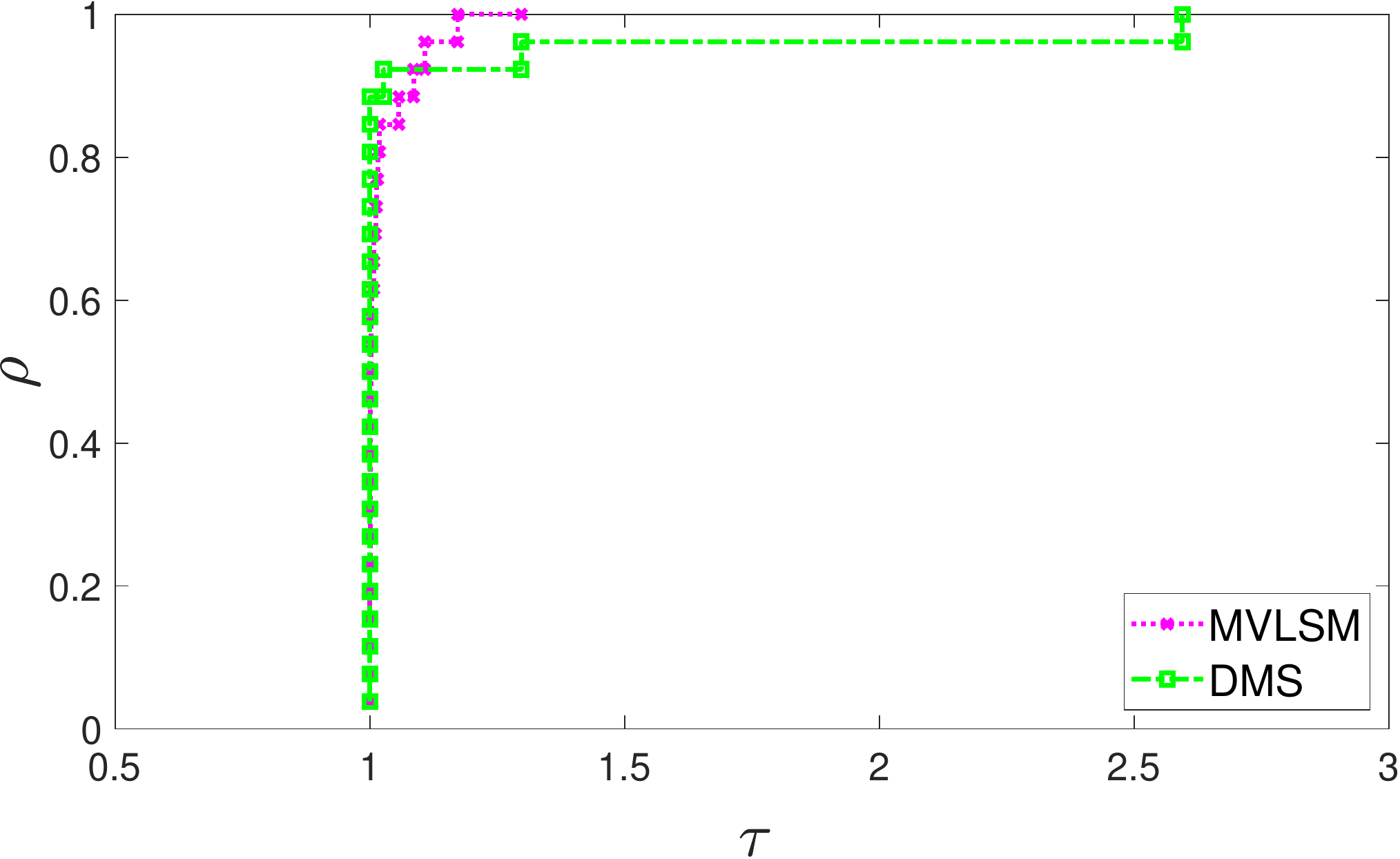}}
\subfigure[MVLSM versus MOIF.]{\includegraphics[scale=0.3]{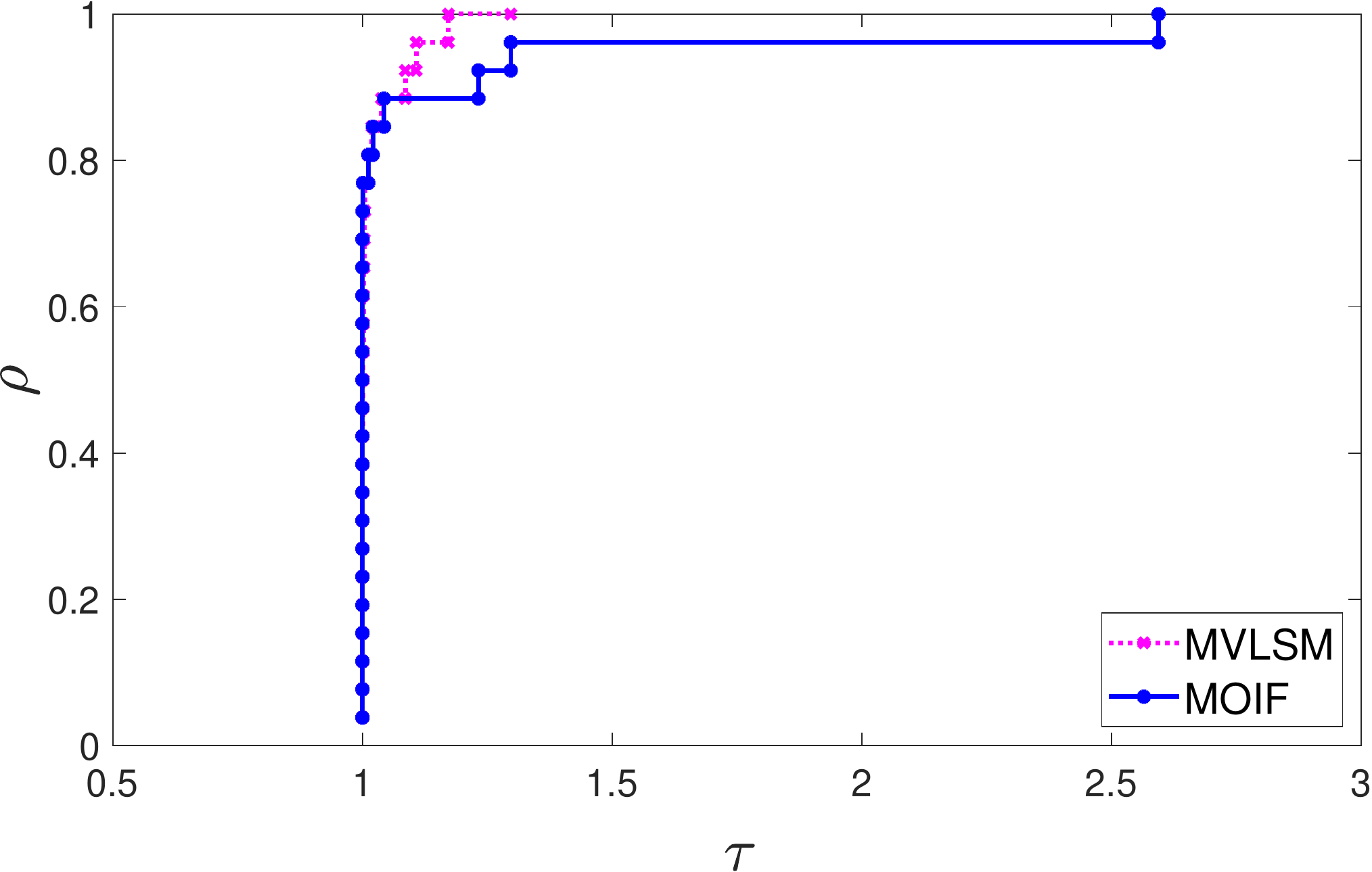}}
\subfigure[All.]{\includegraphics[scale=0.3]{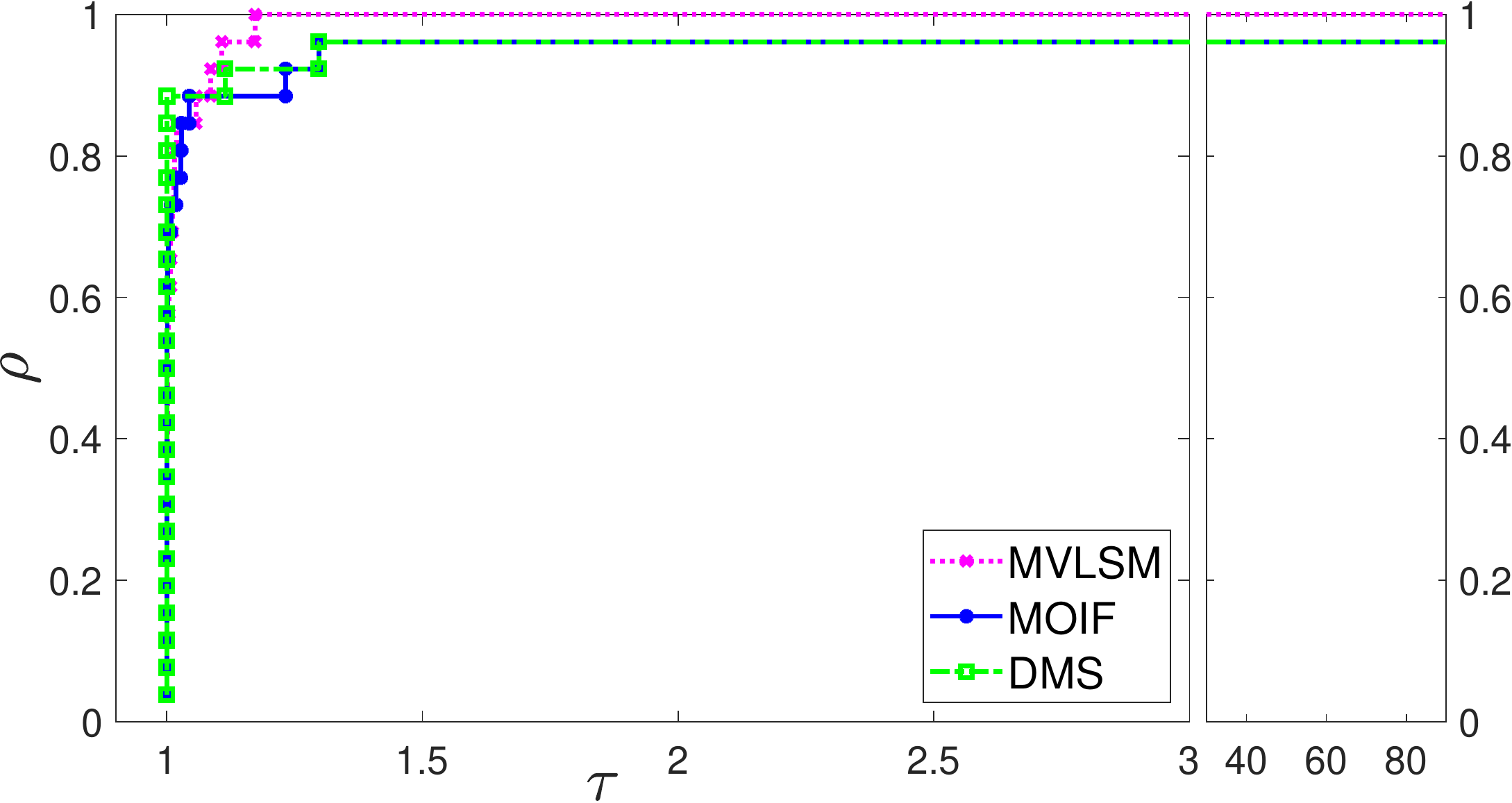}}
\caption{Comparing the MVLSM with DMS and MOIF based on hypervolume performance profiles for multiobjective problems.}
\label{hypervolume}
\end{figure}

%As we know, the more MVLSM rounds we make the more more points in the Pareto front are generated. Nevertheless, we can get reasonable  good results in purity and Hypervolume allowing the algorithm one budget computational of $20.000$ function evaluations.

%------------------------
% Tables
% Table 1 - 1-dimensional
\begin{table}[h!]
\centering
\begin{tabular}{|l|c|c|}\hline
 Problem & Objective functions & Pareto front\\
\hline
 &\multirow{3}{*}{\includegraphics[scale=0.47]{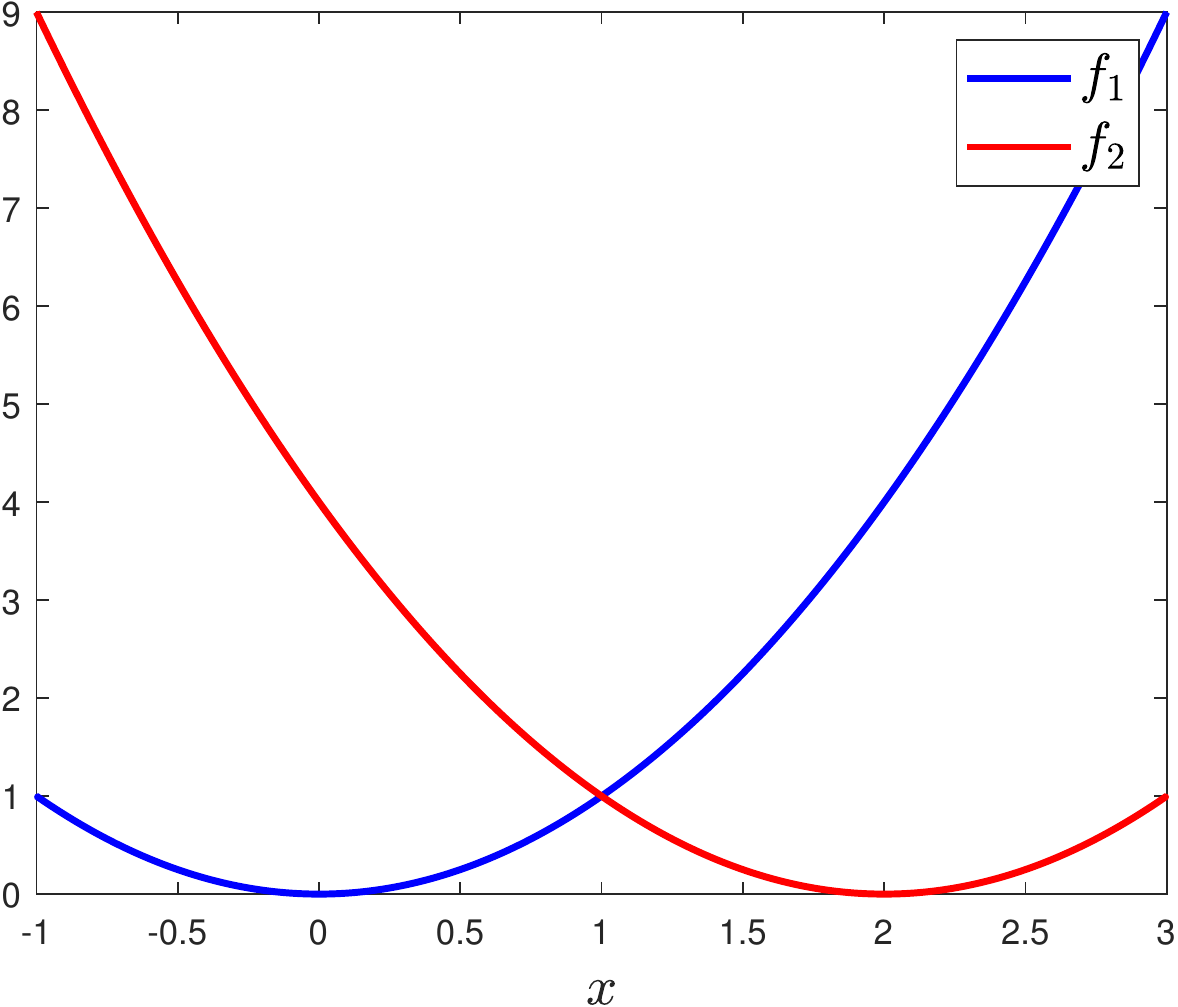}}
 &\multirow{3}{*}{\includegraphics[scale=0.47]{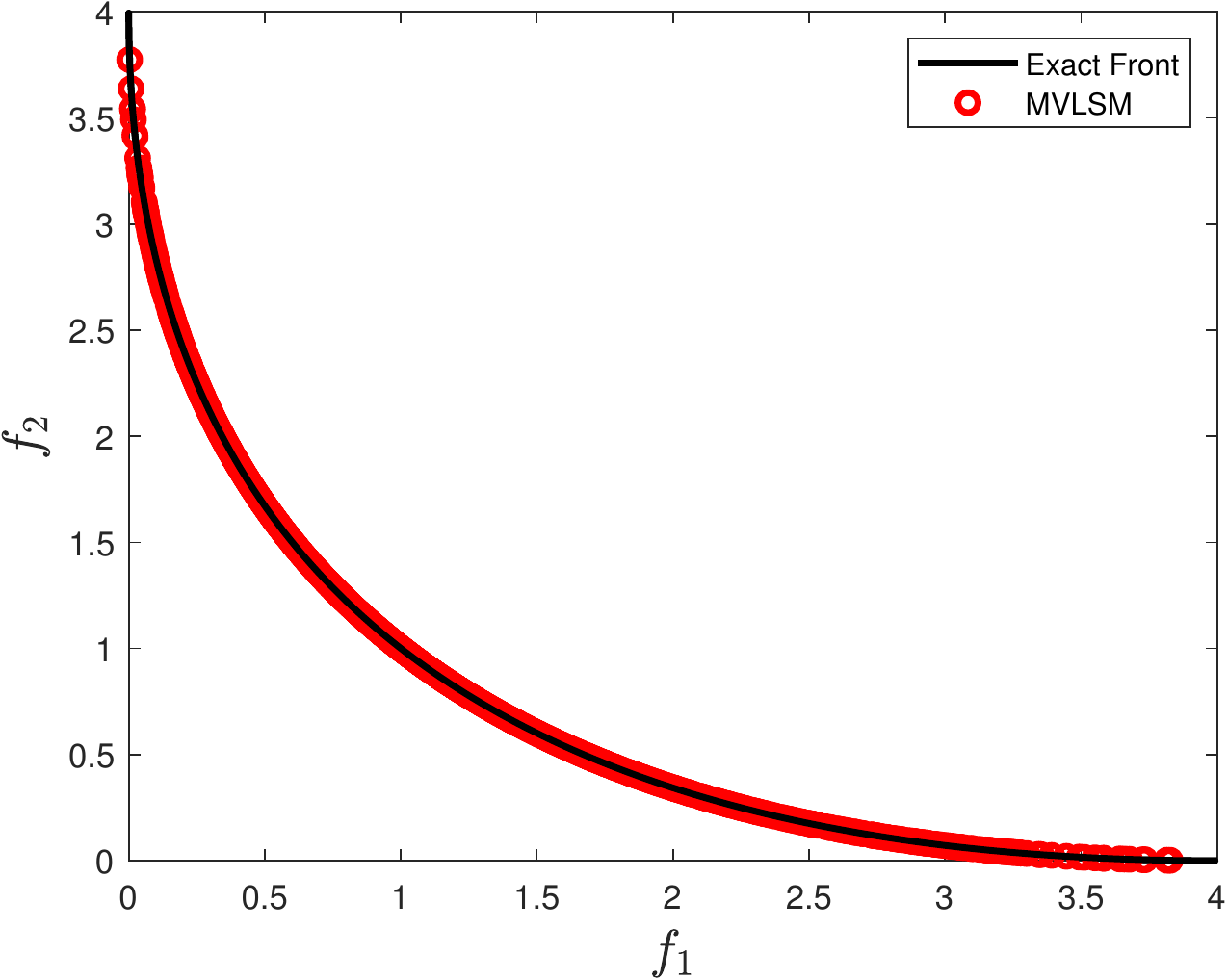}}\\
\cite[MOP 13]{veld1999}& &\\
\cite[SCH1]{deb2005}& &\\
& &\\
$n=1$ & & \\
$r=2$ & &\\
& &\\
& &\\
$T=0.0973$ sec & &\\
$\overline{k}=50.4$ & &\\
& &\\
\hline
&\multirow{3}{*}{\includegraphics[scale=0.47]{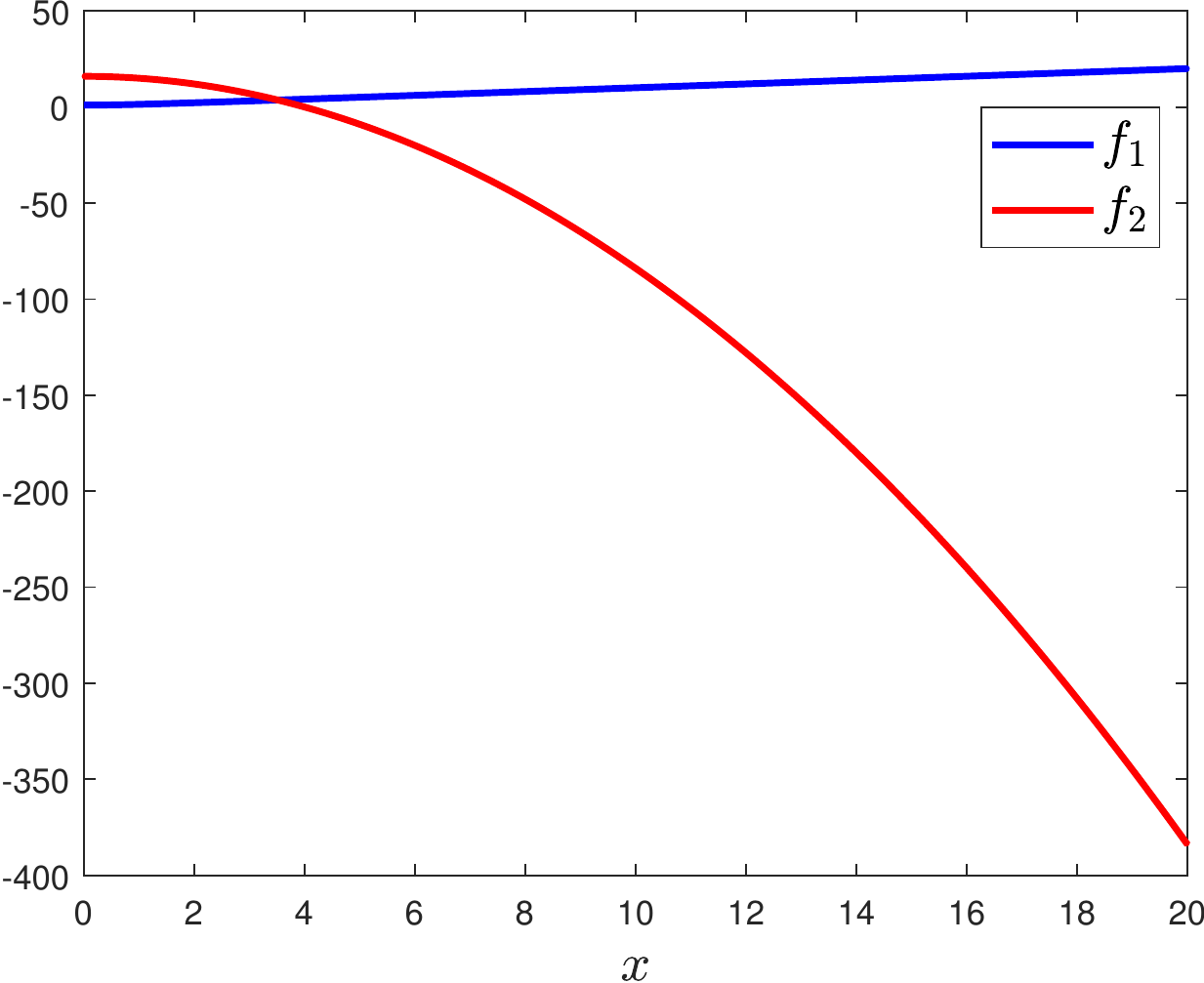}}
 &\multirow{3}{*}{\includegraphics[scale=0.47]{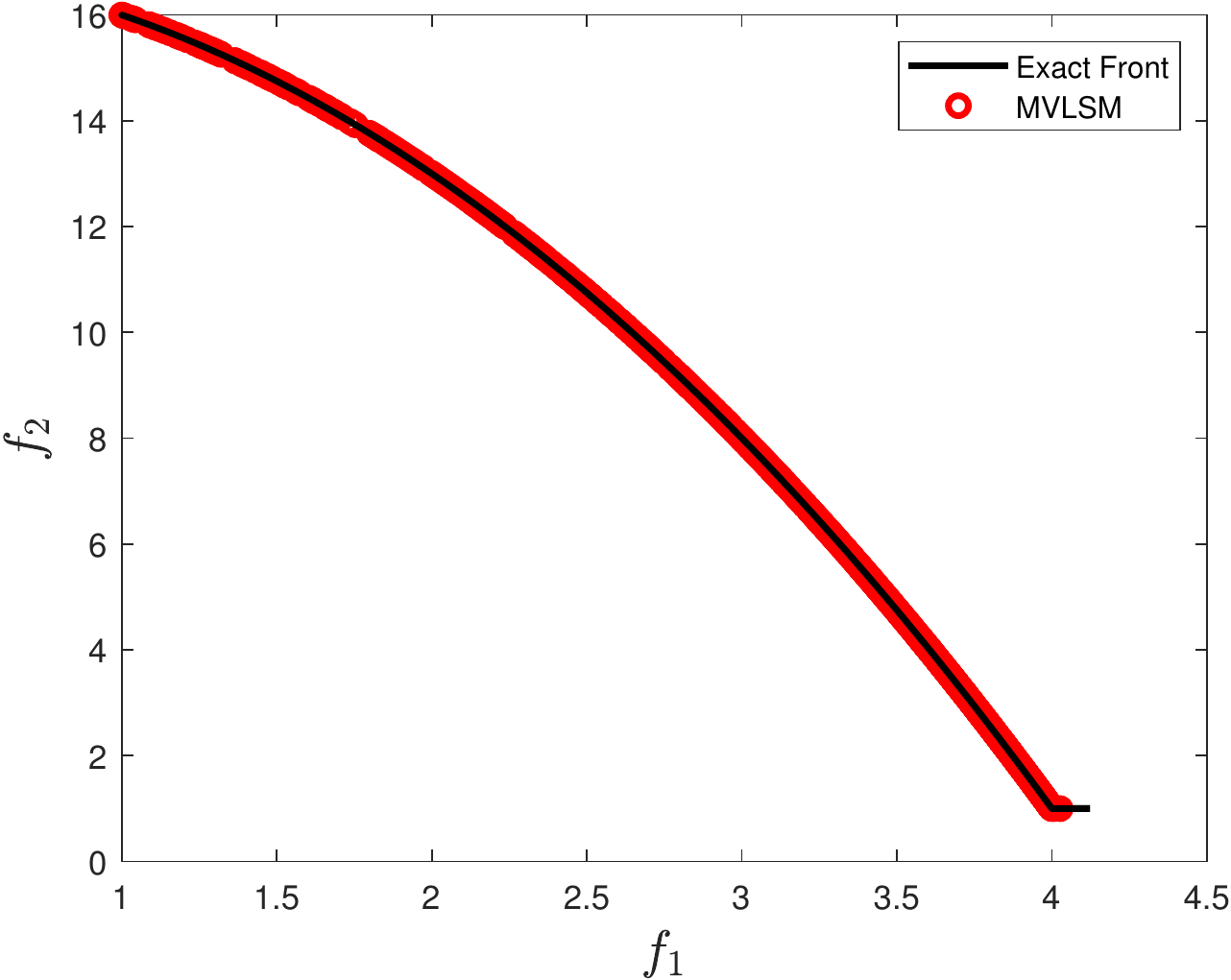}}\\
\cite[Example 9]{chankong1983}& &\\
& &\\
& &\\
$n=1$ & & \\
$r=2$ & &\\
& &\\
& &\\
$T=0.1972$ sec & &\\
$\overline{k}=97.5$ 
& &\\
& &\\
\hline
&\multirow{3}{*}{\includegraphics[scale=0.47]{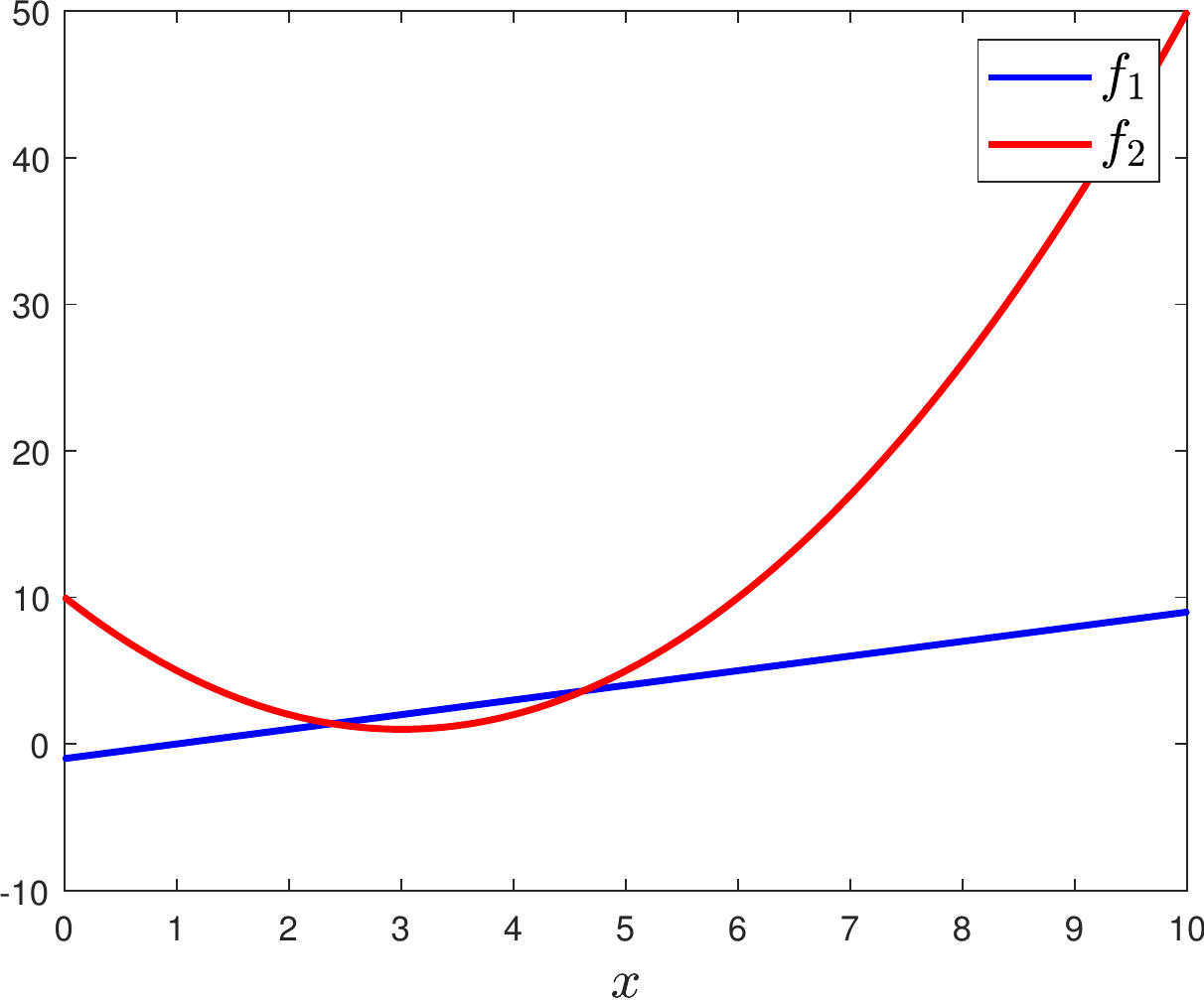}}
 &\multirow{3}{*}{\includegraphics[scale=0.47]{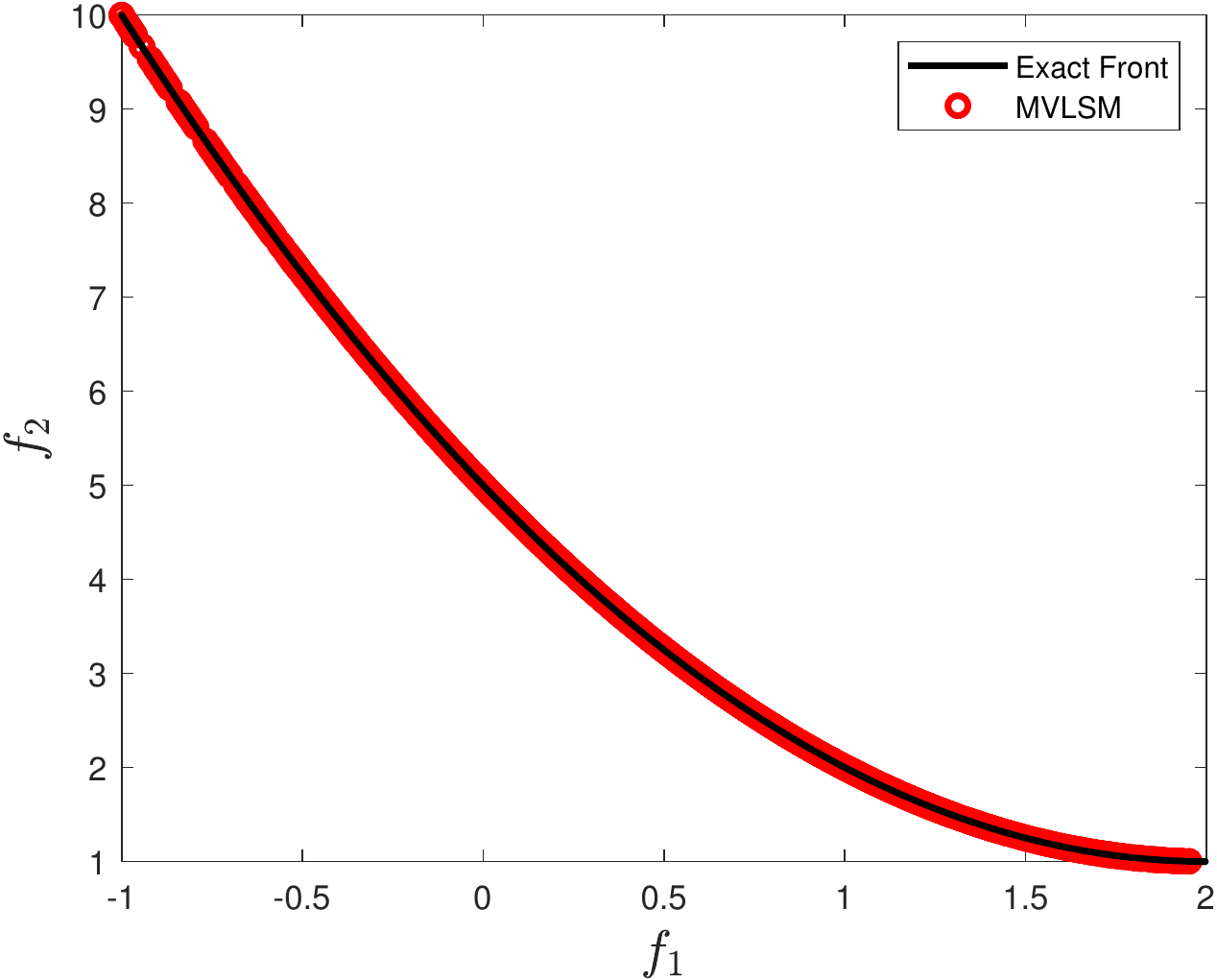}}\\
\cite[Problem 4.7]{chankong1983}& &\\
& &\\
& &\\
$n=1$ & & \\
$ r=2$ & &\\
& &\\
& &\\
$T=0.1653$ sec & &\\
$\overline{k}=85.0$ 
& &\\
& &\\
\hline
&\multirow{3}{*}{\includegraphics[scale=0.47]{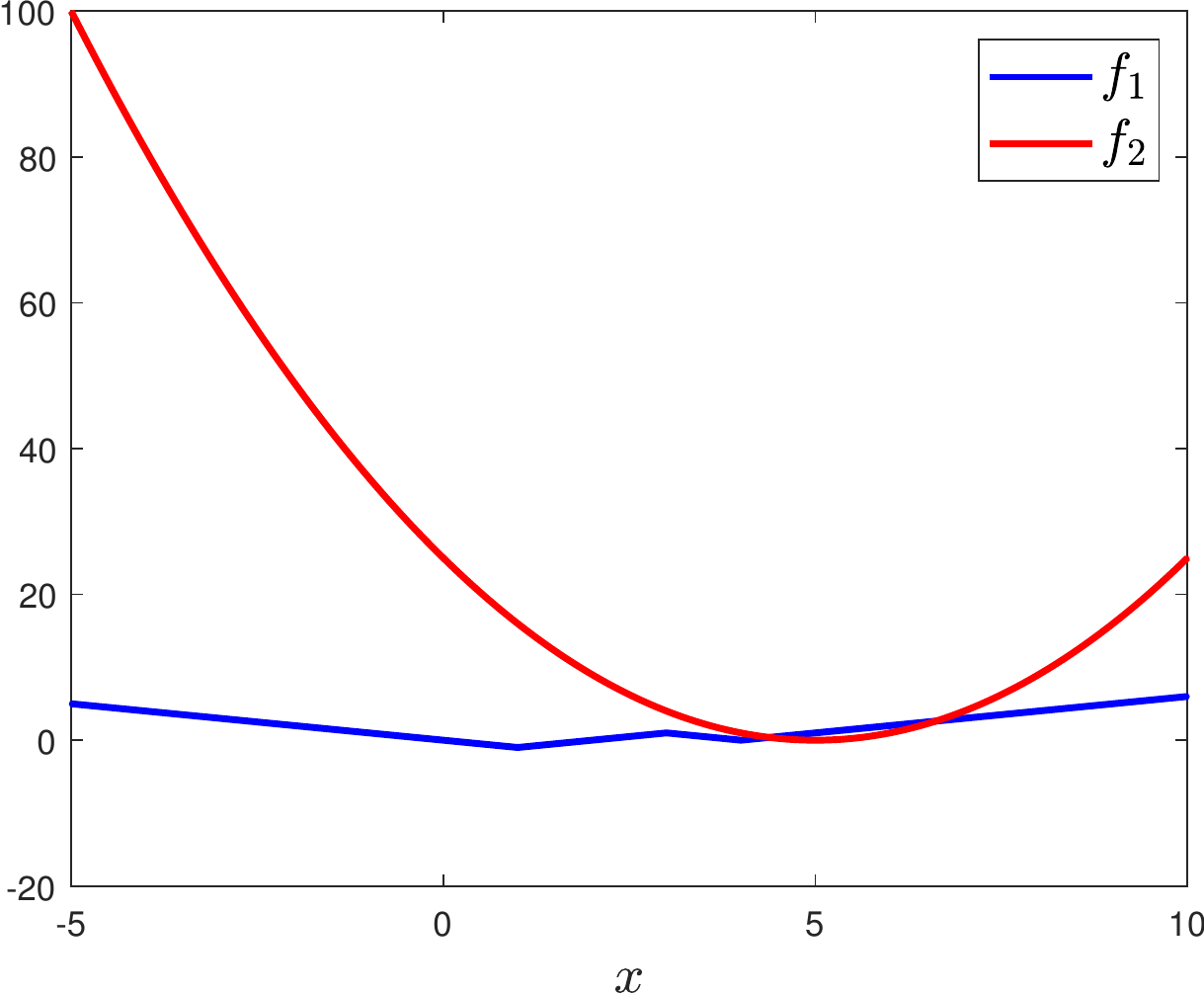}}
 &\multirow{3}{*}{\includegraphics[scale=0.47]{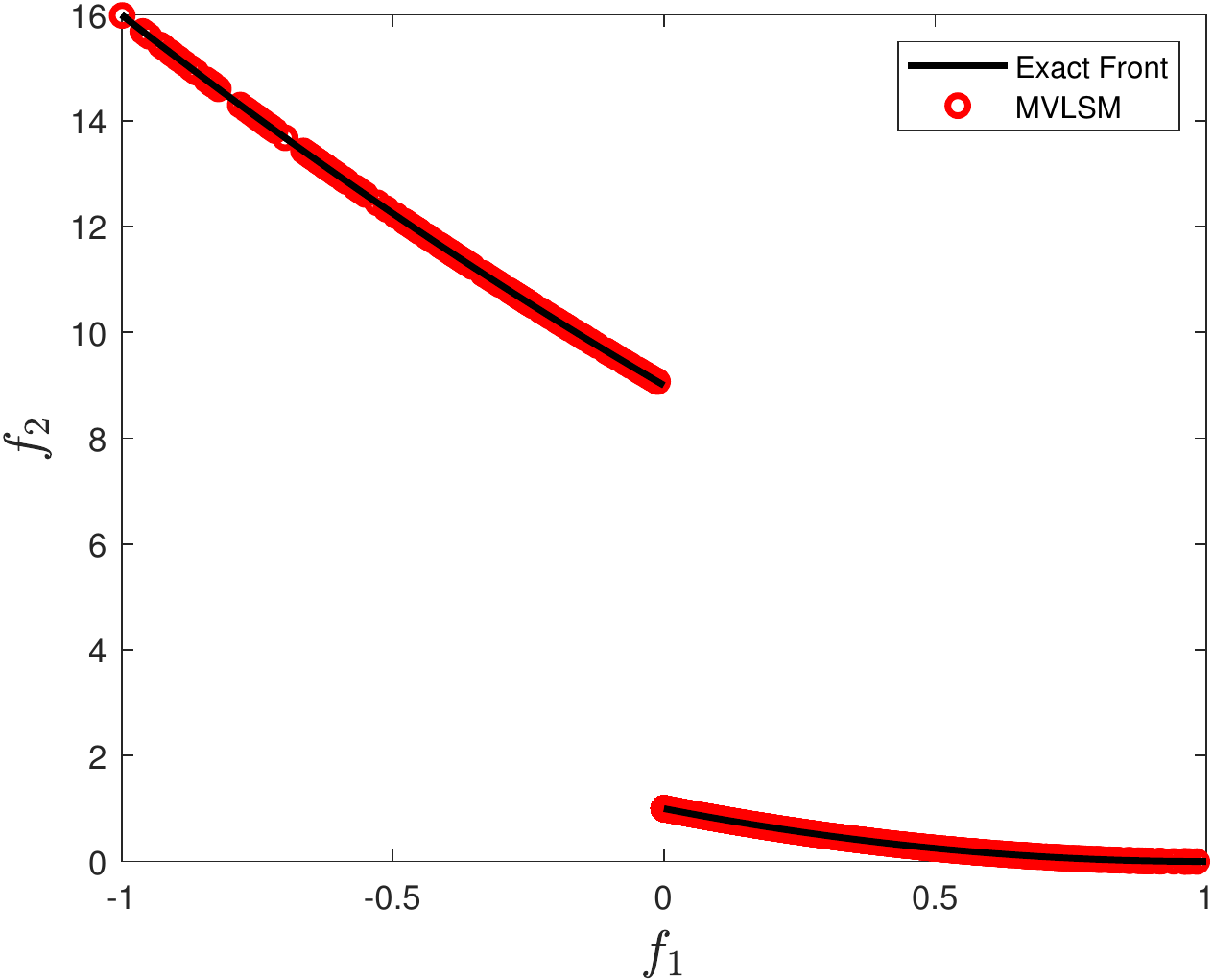}}\\
\cite[MOP 14]{veld1999}& &\\
& &\\
& &\\
$n=1$ & & \\
$r=2$ & &\\
& &\\
& &\\
$T=0.0867$ sec & &\\
$\overline{k}=45.7$ 
& &\\
& &\\
\hline
\end{tabular}
\caption{Results for problems with dimension $n=1$.}
\label{table1}
\end{table}

% 2-Dimensional -Table 2-------
\begin{table}[h!]
\centering
\begin{tabular}{|l|c|c|}\hline
problems & Objective functions & Pareto front\\
\hline
&\multirow{3}{*}{\includegraphics[scale=0.47]{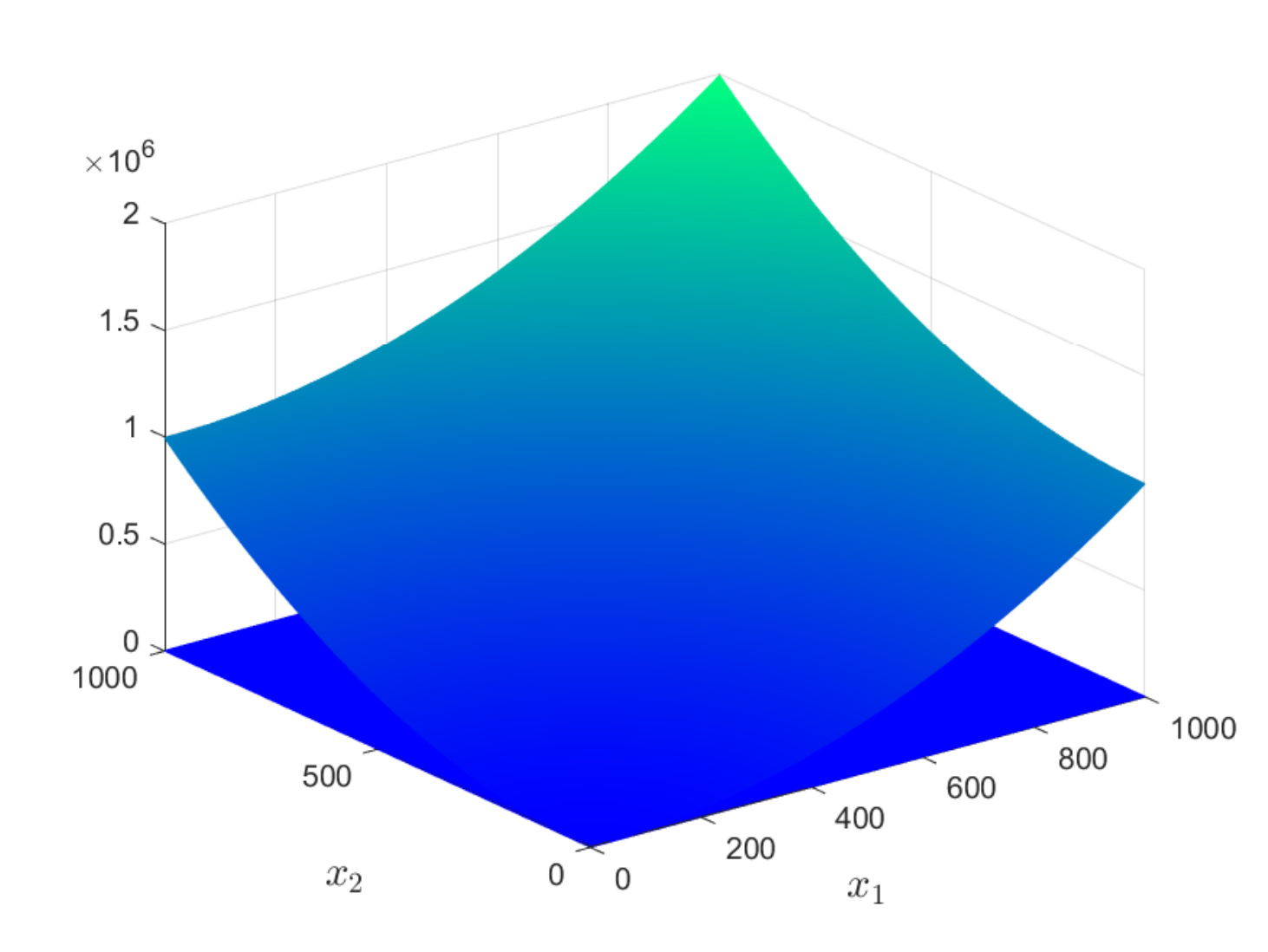} }
 &\multirow{3}{*}{\includegraphics[scale=0.47]{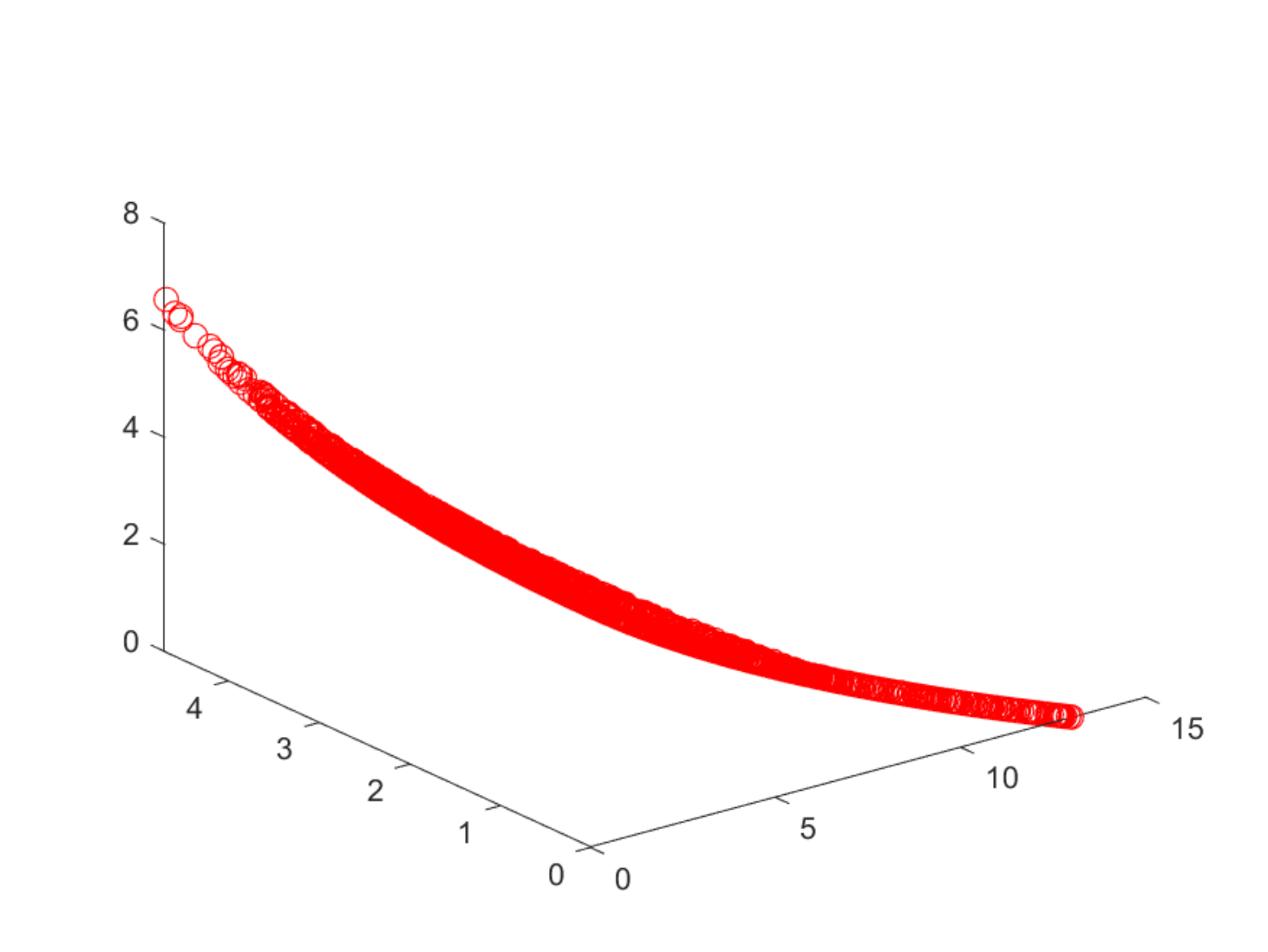}}\\
\cite[Example 4.3.6]{chankong1983} & &\\
& &\\
& &\\
$n=2$ & & \\
$r=3$ & &\\
& &\\
& &\\
$T= 0.0610$ sec & &\\
$\overline{k}=32.1$ & &\\
& &\\
\hline
&\multirow{3}{*}{\includegraphics[scale=0.47]{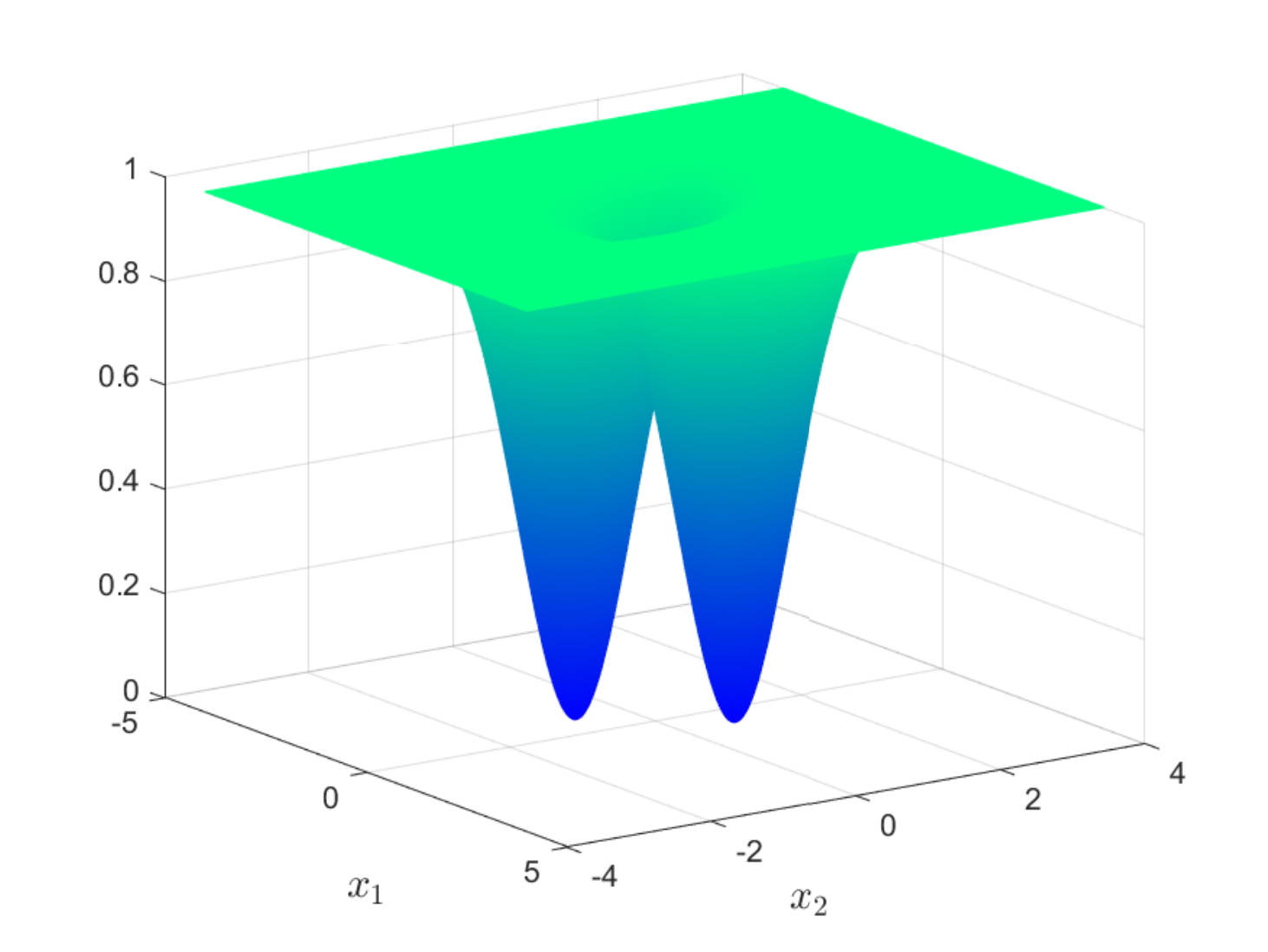}}
 &\multirow{3}{*}{\includegraphics[scale=0.47]{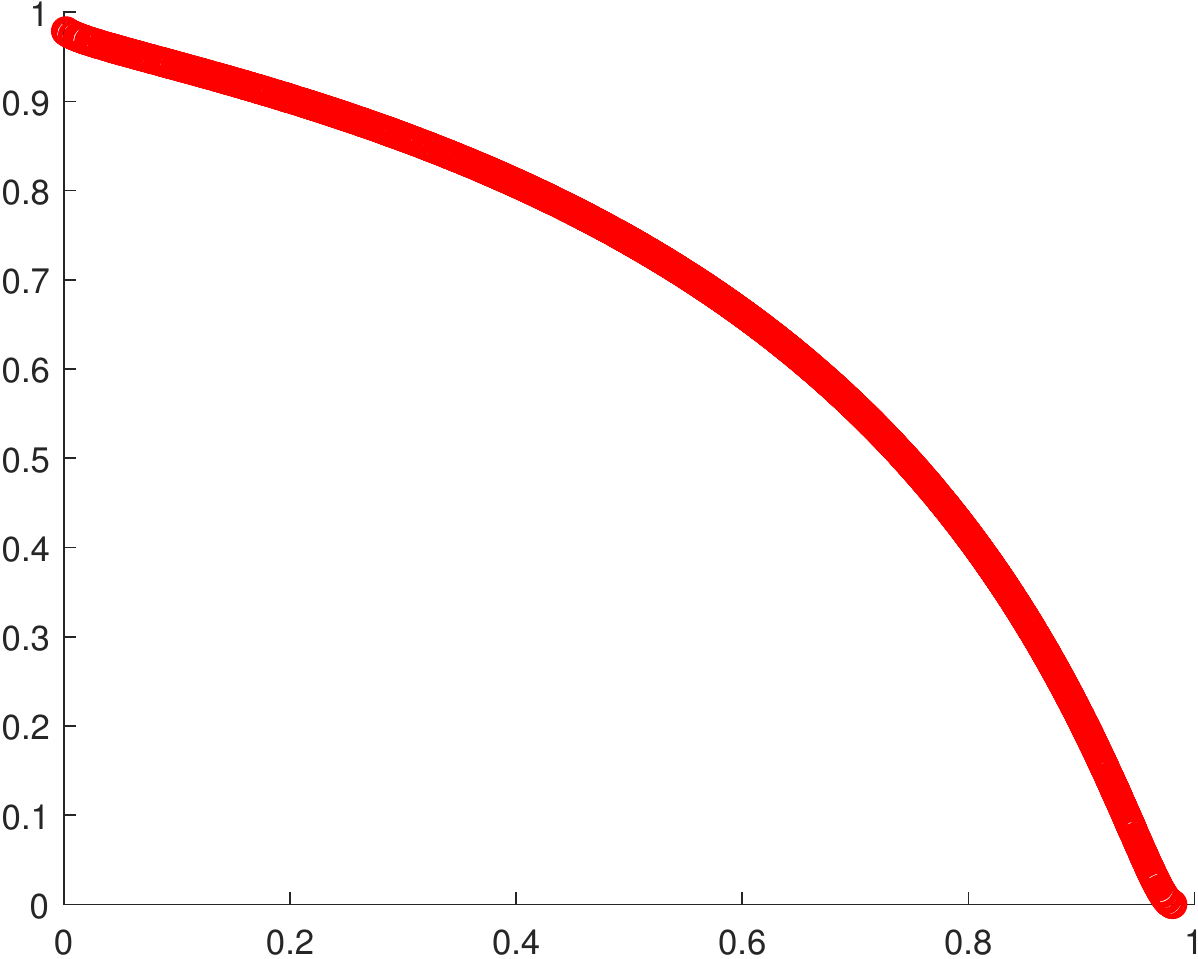}}\\ 
\cite[(6.1)]{deb2005},& &\\
\cite[MOP 4]{veld1999}& &\\
& &\\
$n=2$ & & \\
$r=2$ & &\\
& &\\
 & &\\
$T=0.0391$ sec & &\\
$\overline{k}=20.4$ & &\\
& &\\
\hline
 &\multirow{3}{*}{\includegraphics[scale=0.47]{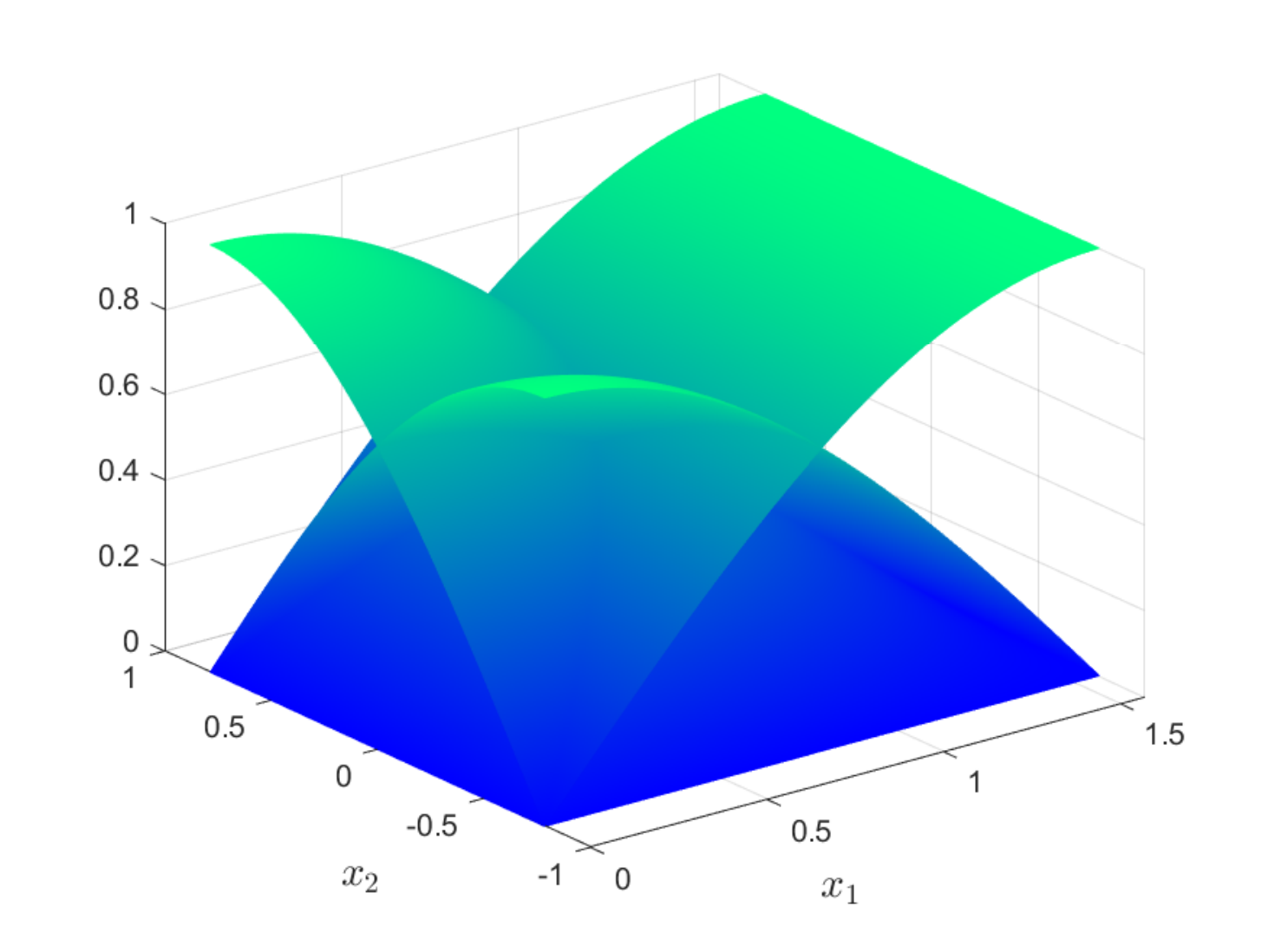}}
 &\multirow{3}{*}{\includegraphics[scale=0.47]{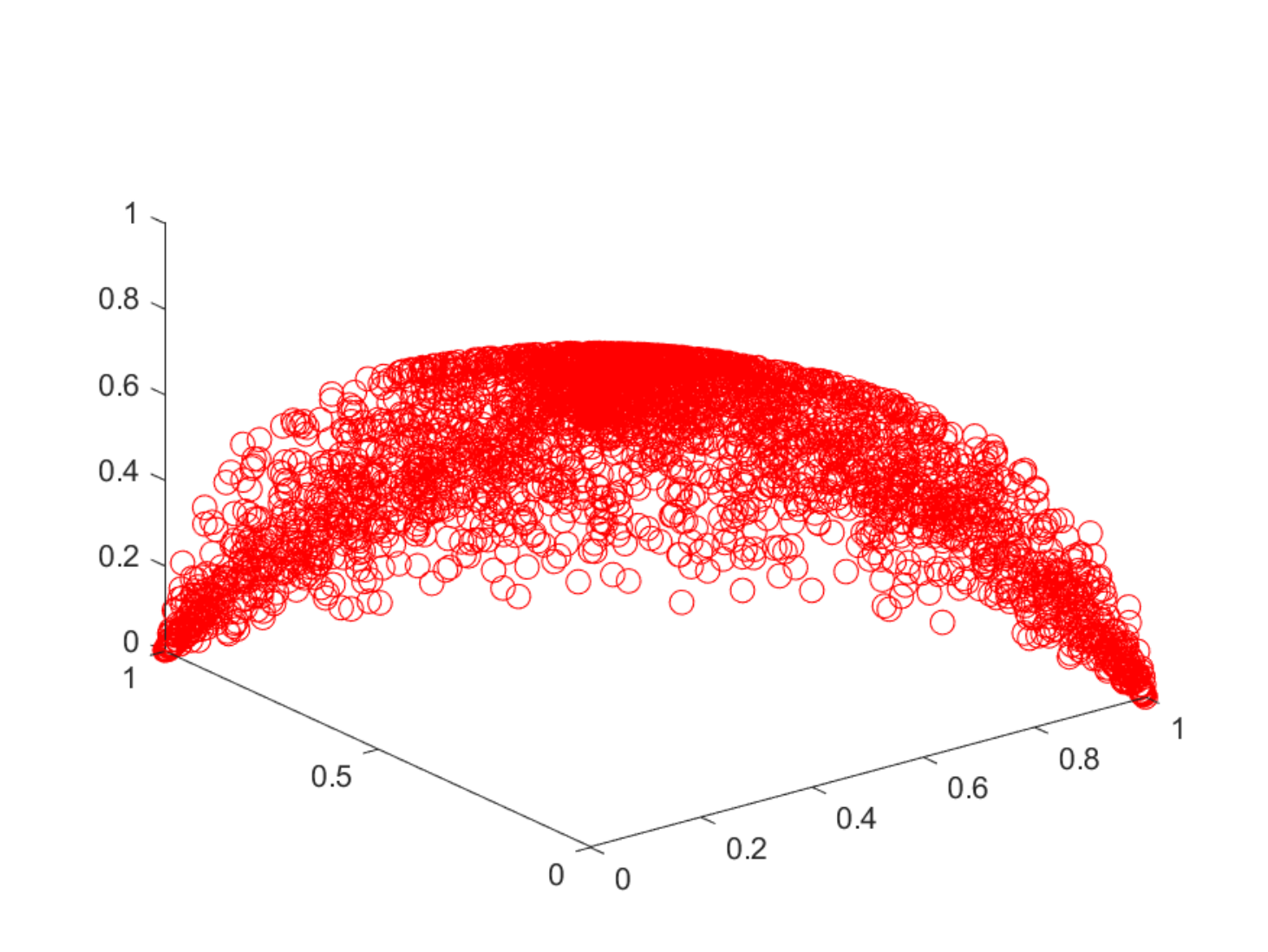}}\\
\cite[(6.2)]{deb2005} & &\\
& &\\
& &\\
$n=2$ & & \\
$r=3$ & &\\
& &\\
 & &\\
$T=0.0427$ sec & &\\
$\overline{k}=20.4$ & &\\
& &\\
\hline
 &\multirow{3}{*}{\includegraphics[scale=0.47]{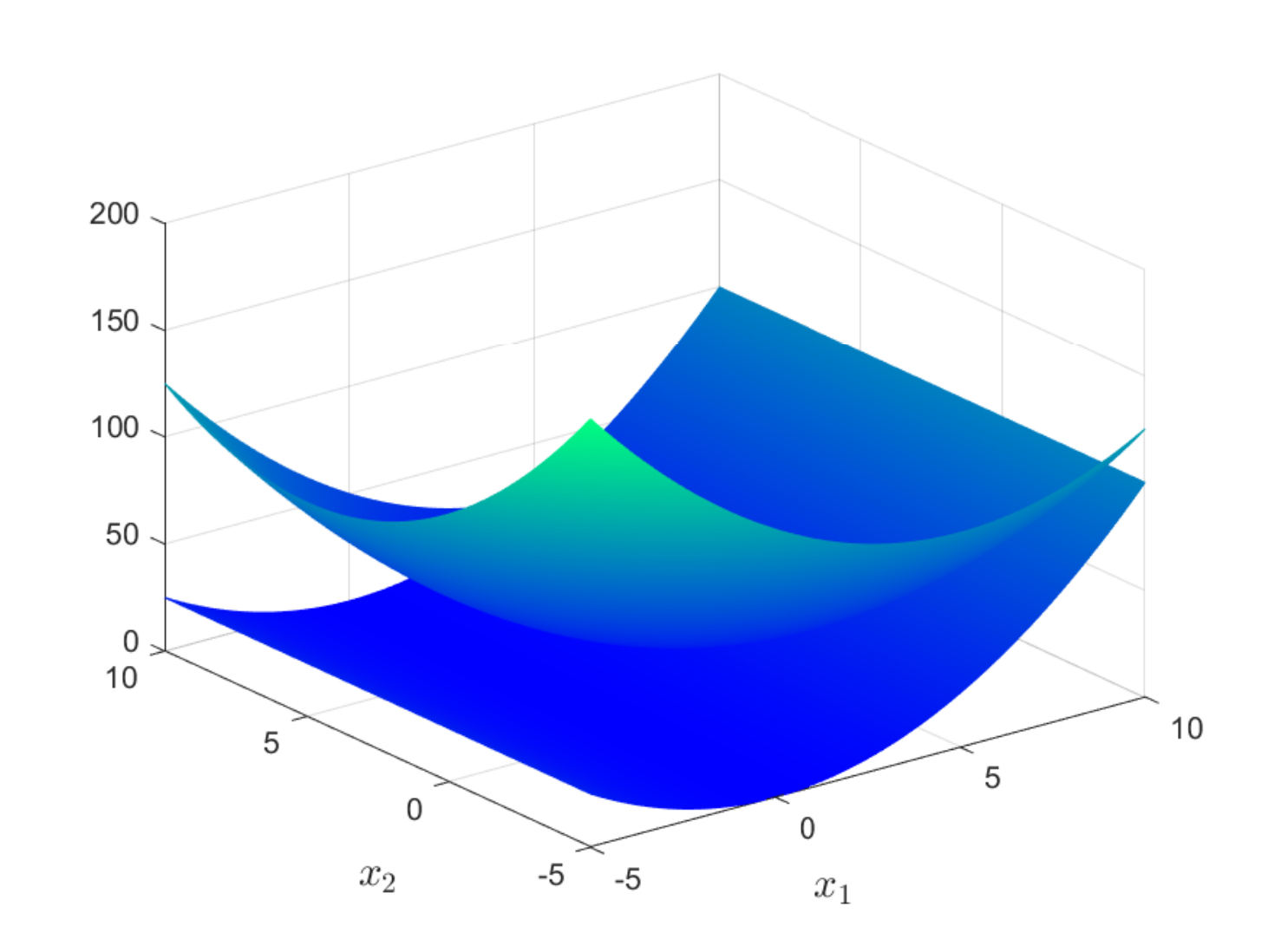}}
 &\multirow{3}{*}{\includegraphics[scale=0.47]{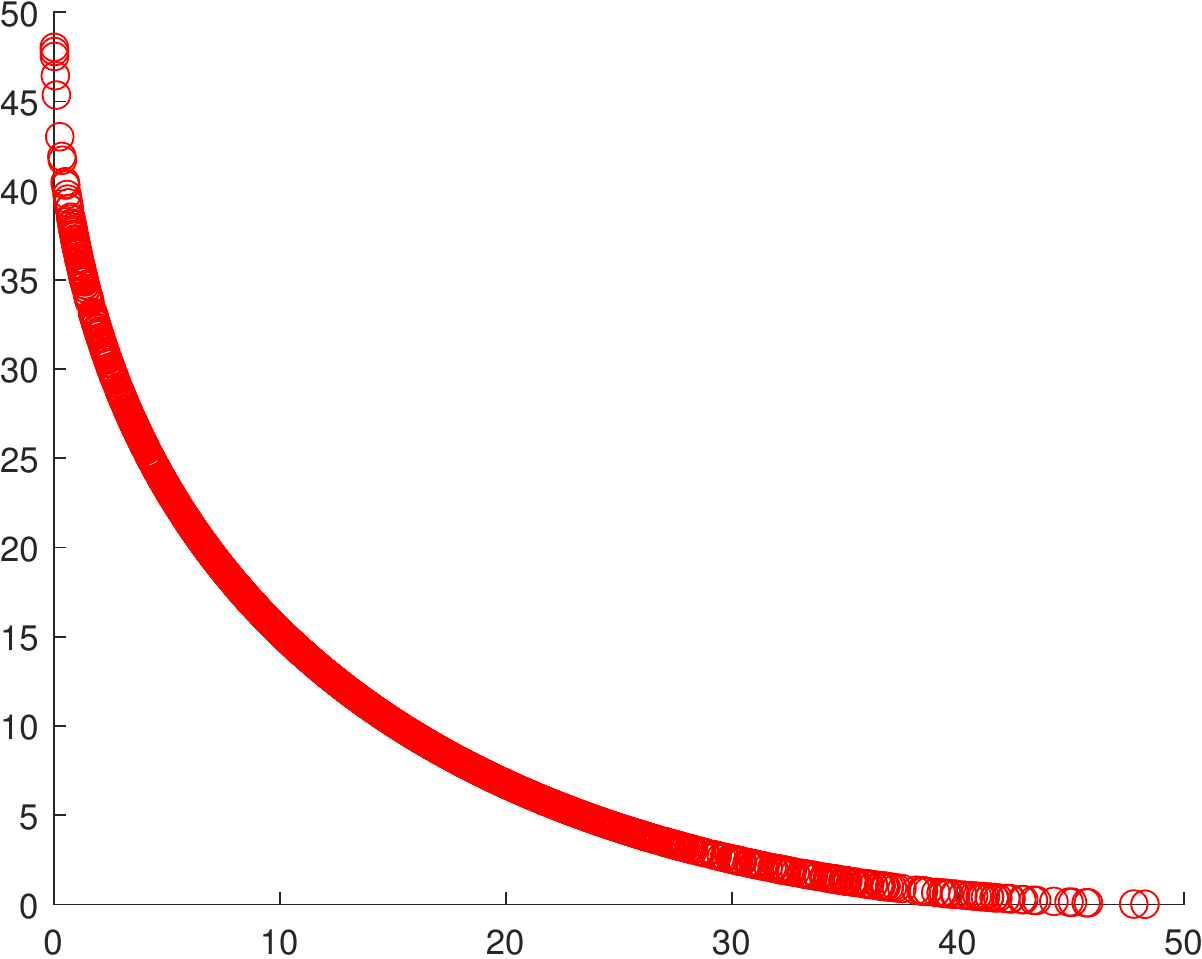}}\\
\cite[MOP 1]{veld1999} & &\\
& &\\
& &\\
$n=2$ & & \\
$r=2$ & &\\
& &\\
& &\\
$T=0.0427$ sec & &\\
$\overline{k}=23.1$ & &\\
& &\\
\hline
\end{tabular}
\caption{Results for problems with dimension $n=2$.}
\label{table2}
\end{table}

% 2-Dimensional and bibliography veld1999 - Table 3
\begin{table}[h!]
\centering
\begin{tabular}{|l|c|c|}\hline
Problem & Objective functions & Pareto front\\
\hline
&\multirow{3}{*}{\includegraphics[scale=0.47]{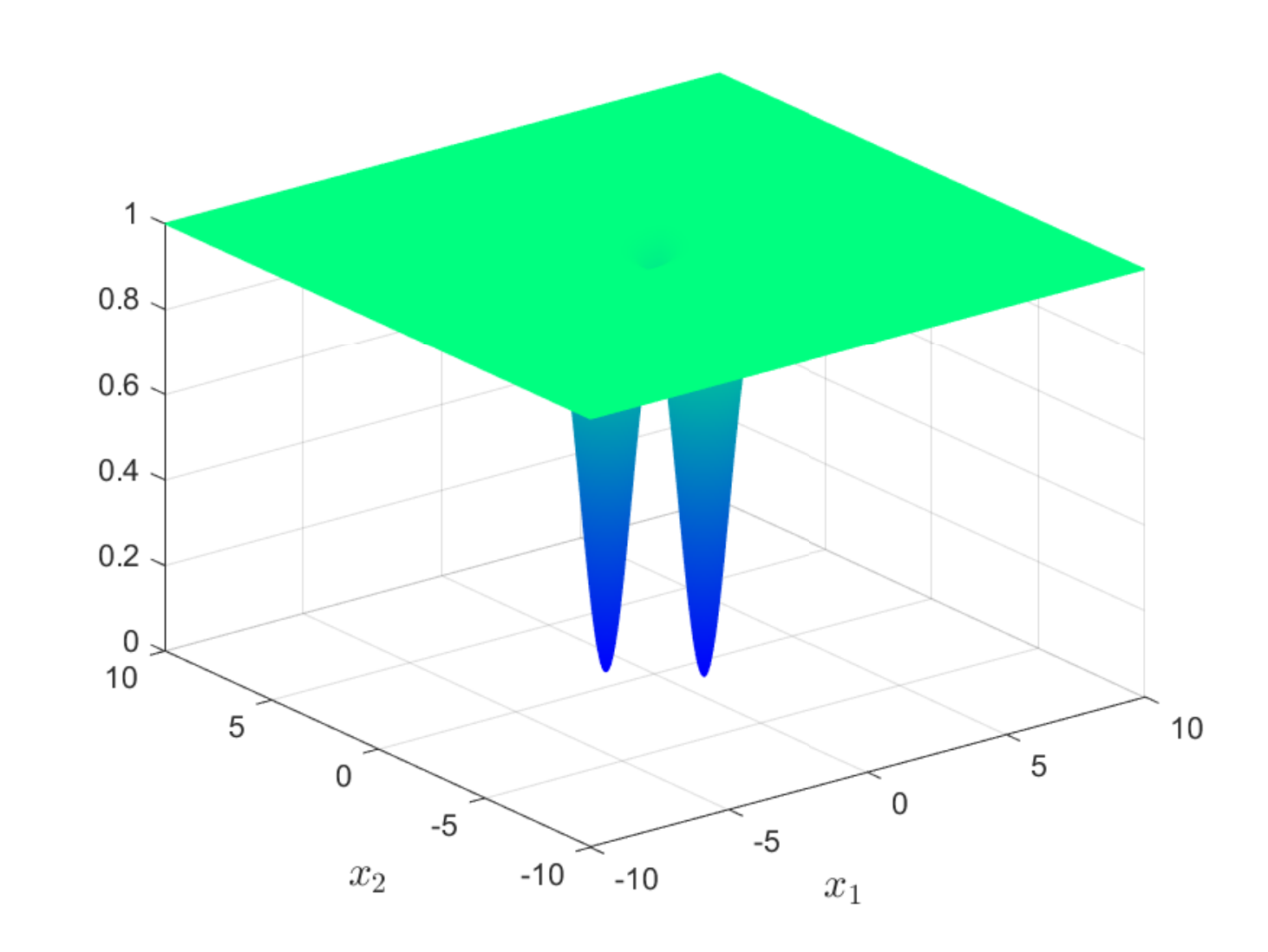}}
 &\multirow{3}{*}{\includegraphics[scale=0.47]{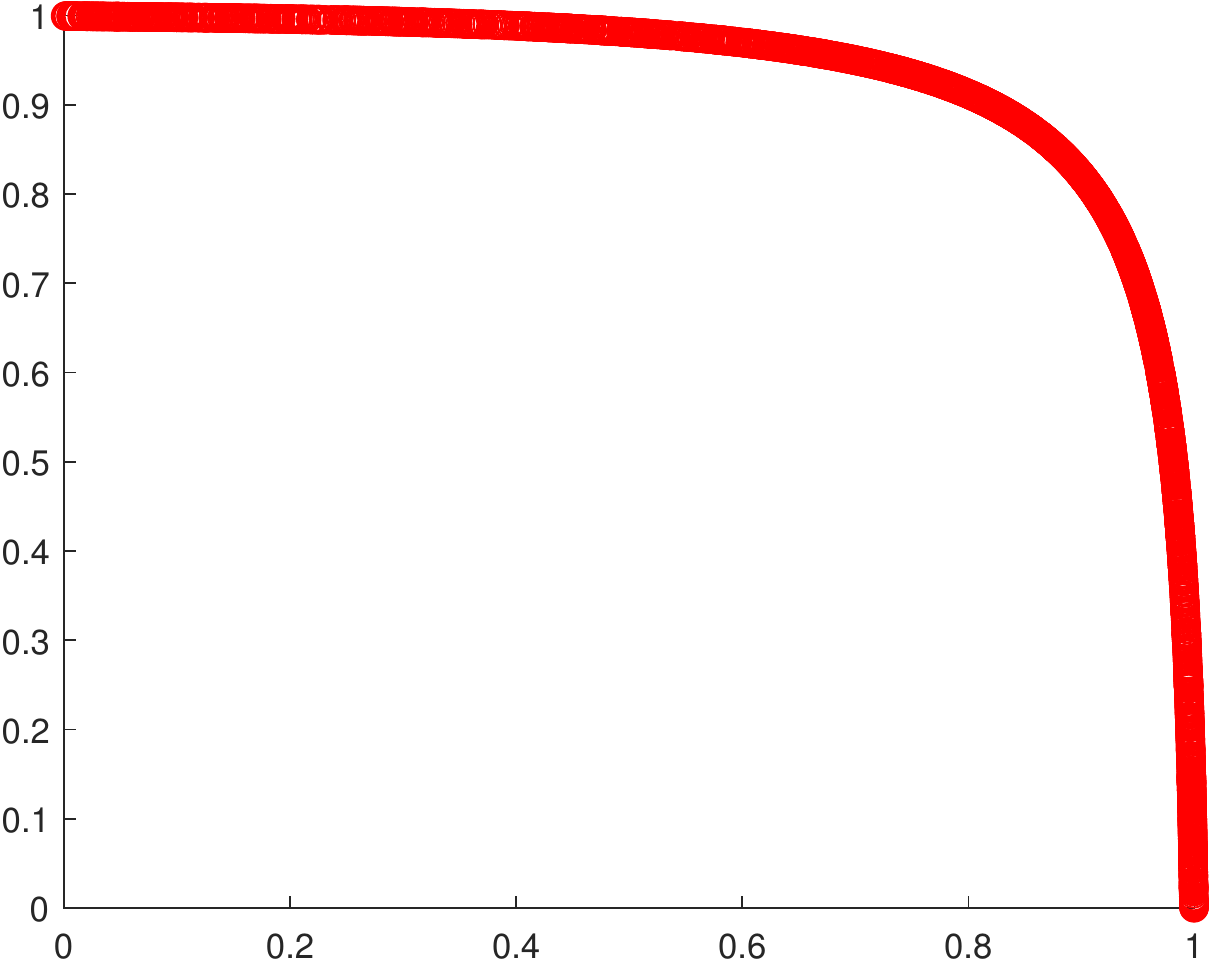}}\\
\cite[MOP 3]{veld1999} & &\\
& &\\
& &\\
$n=2$ & & \\
$r=2$ & &\\
 & &\\
& &\\
$T=0.0283$ sec & &\\
$\overline{k}=14.0$ & &\\
& &\\
\hline
 &\multirow{3}{*}{\includegraphics[scale=0.47]{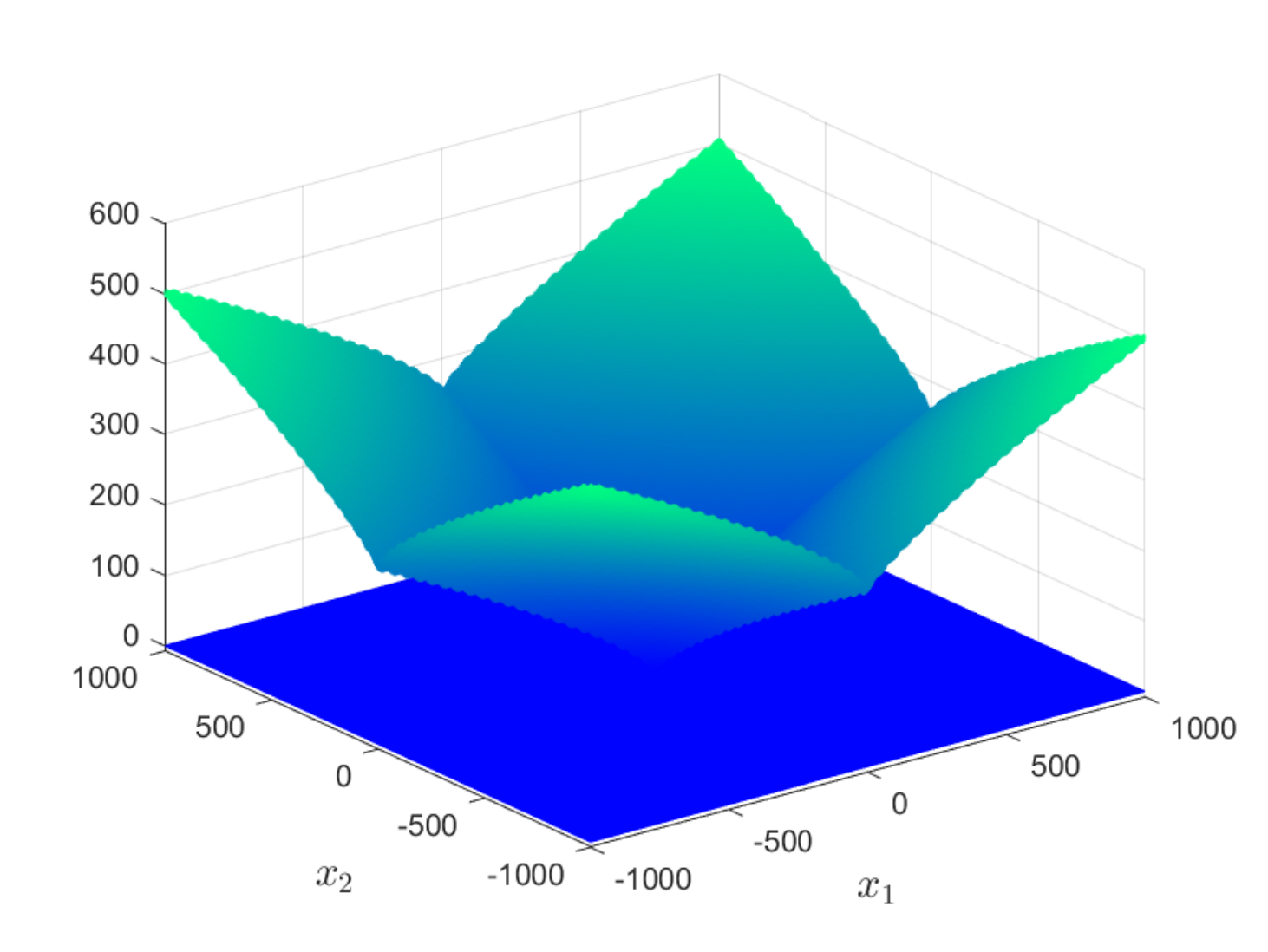}}
 &\multirow{3}{*}{\includegraphics[scale=0.47]{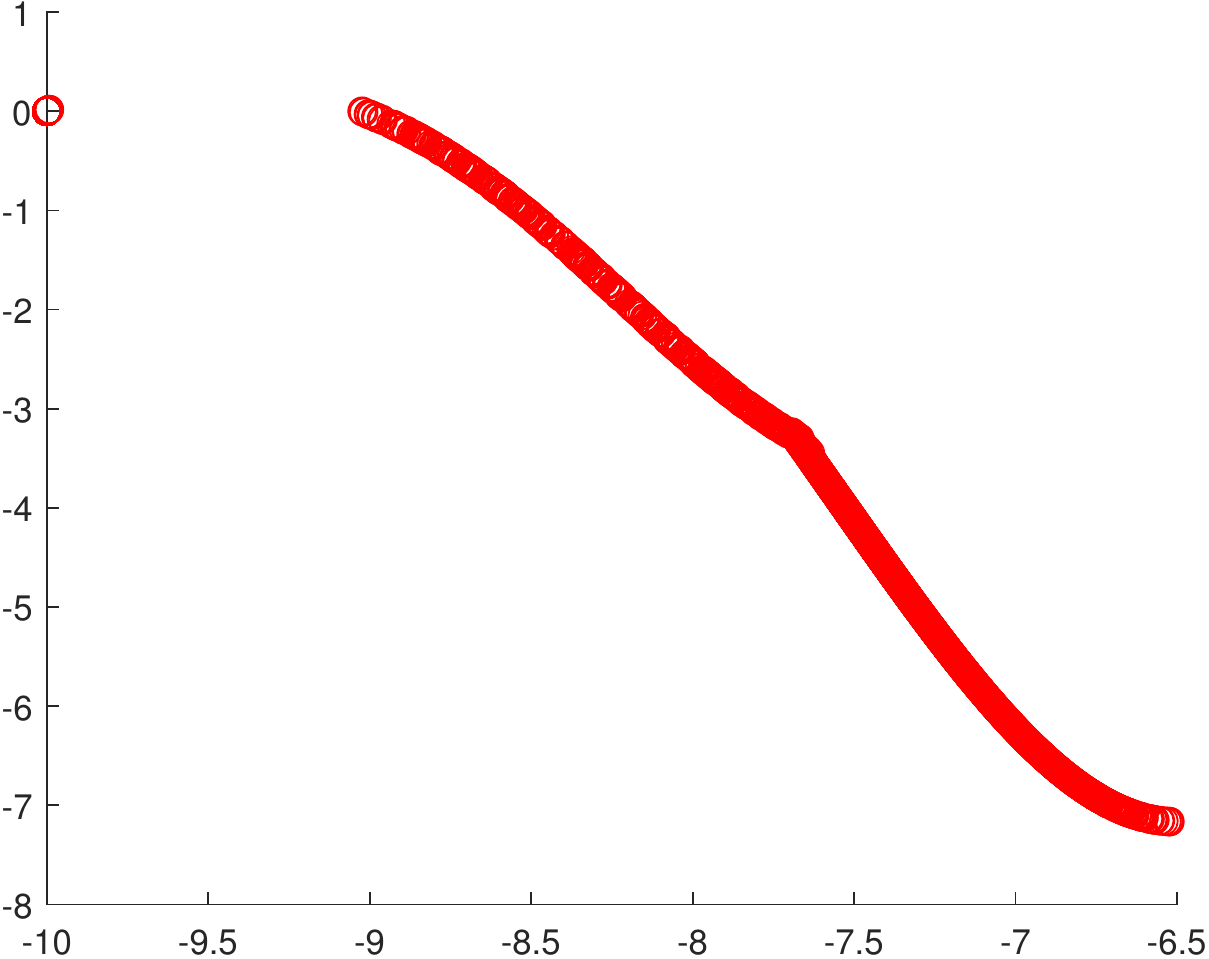}}\\
\cite[MOP 5]{veld1999} & &\\
& &\\
& &\\
$n=2$ & & \\
$r=2$ & &\\
& &\\
 & &\\
$T=0.0829$ sec & &\\
$\overline{k}=32.6$ & &\\
& &\\
\hline
%------PM10
 &\multirow{3}{*}{\includegraphics[scale=0.47]{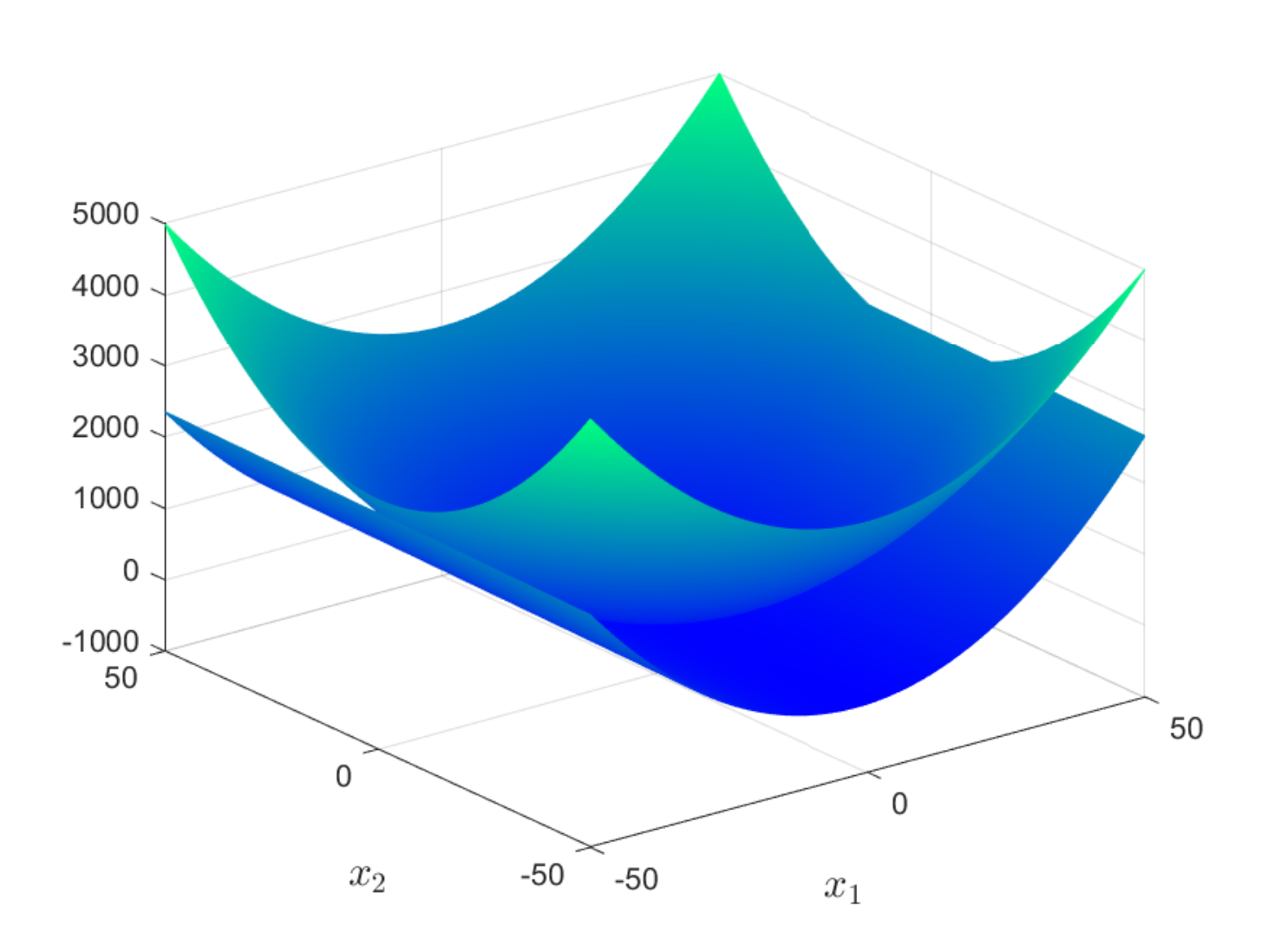}}
 &\multirow{3}{*}{\includegraphics[scale=0.47]{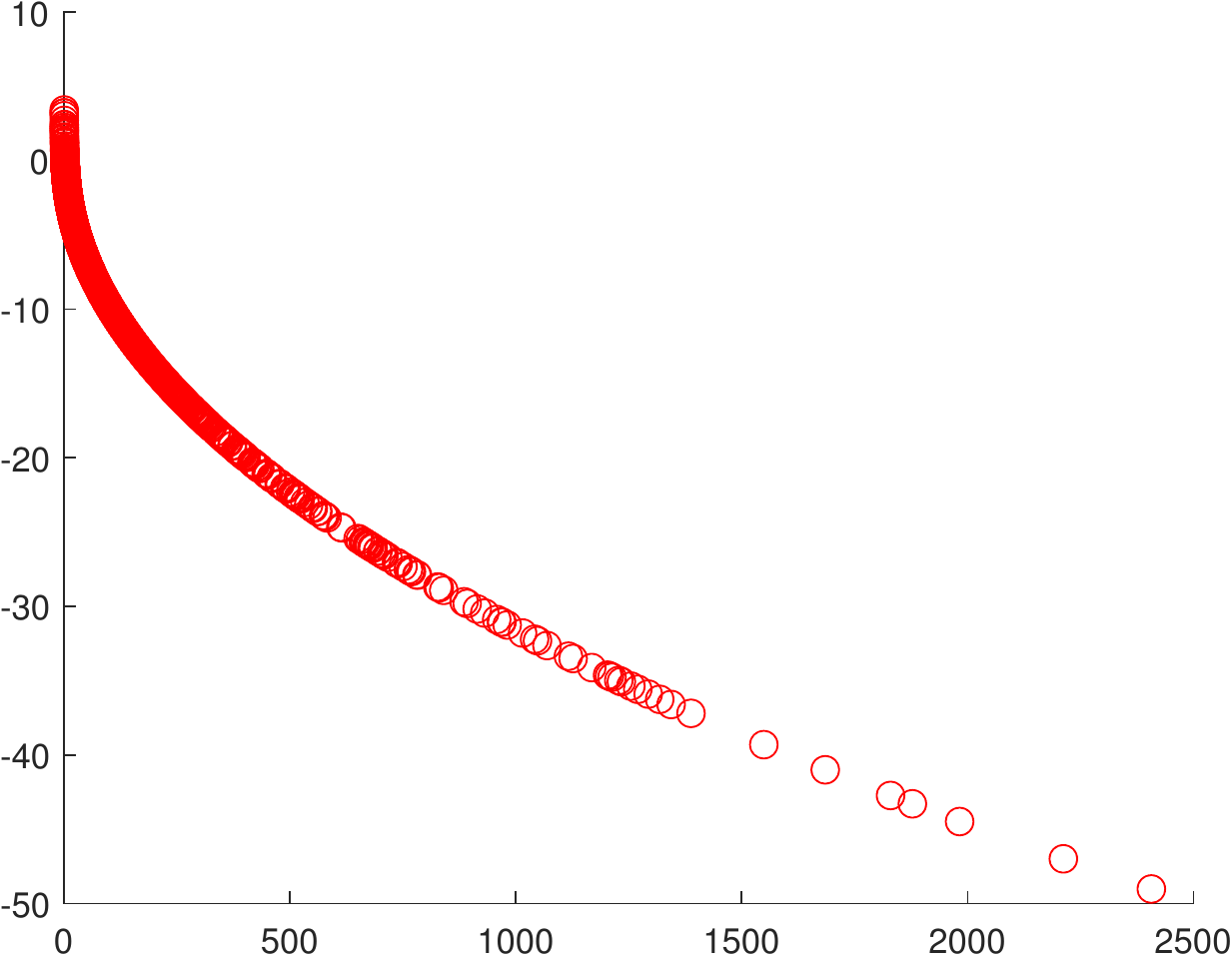}}\\
\cite[MOP 6]{veld1999}& &\\
& &\\
& &\\
$n=2$ & & \\
$r=2$ & &\\
& &\\
 & &\\
$T=0.0511$ sec & &\\
$\overline{k}=28.0$ & &\\
& &\\
\hline
%---- PROBLEM PM11
 &\multirow{3}{*}{\includegraphics[scale=0.47]{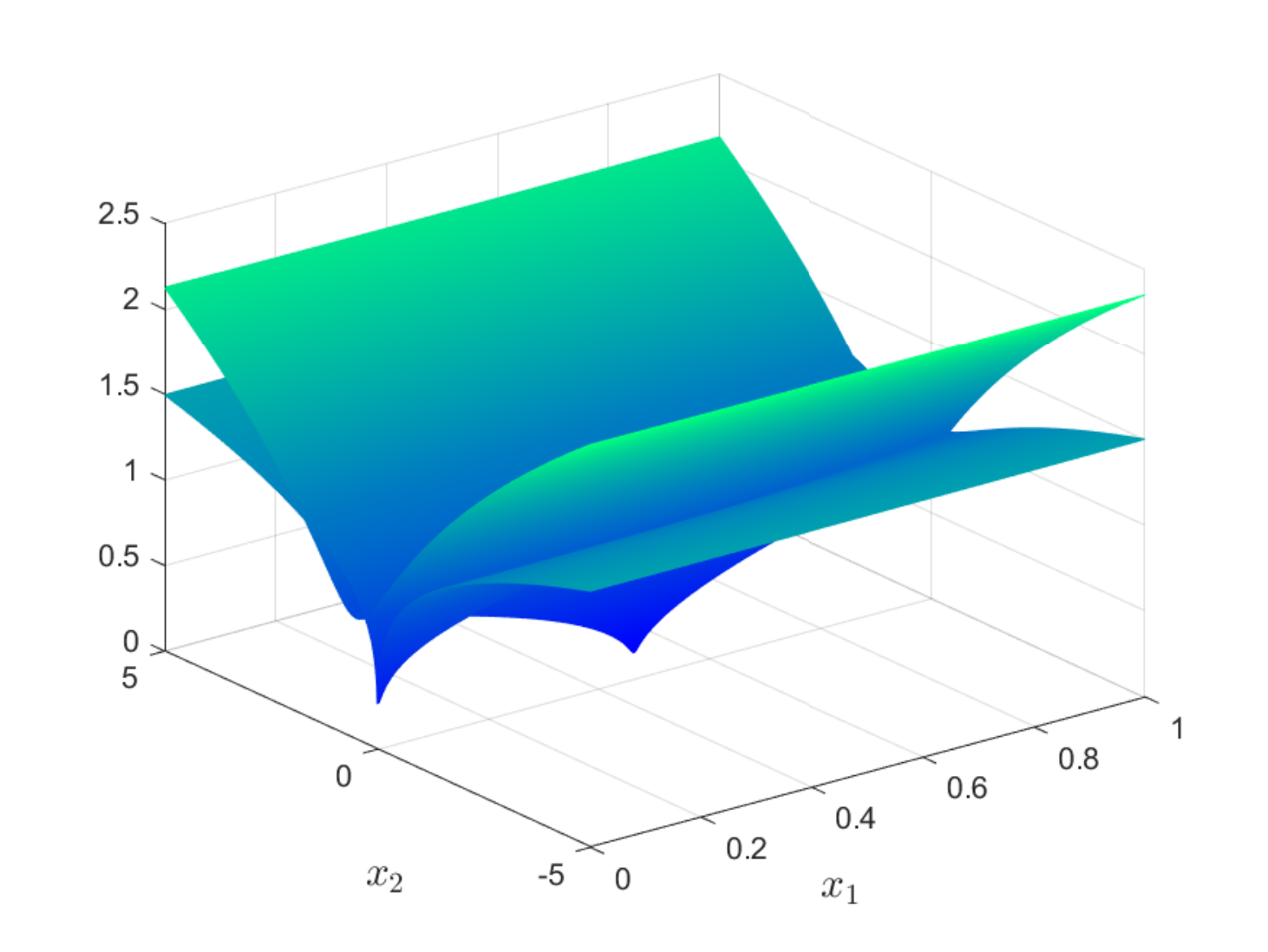}}
 &\multirow{3}{*}{\includegraphics[scale=0.47]{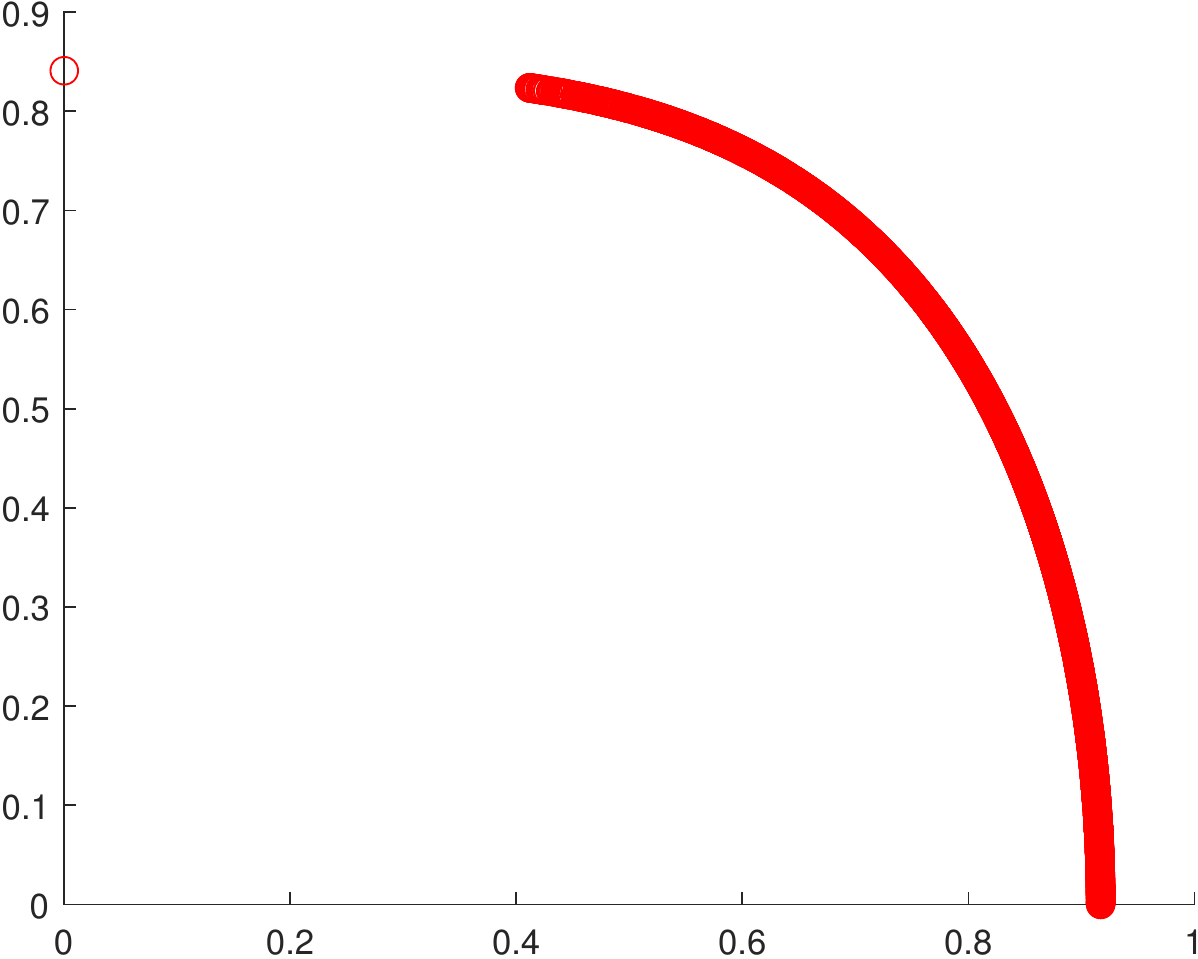}}\\
\cite[MOP 7]{veld1999} & &\\
& &\\
& &\\
$n=2$ & & \\
$r=2$ & &\\
& &\\
& &\\
$T=0.0503$ sec & &\\
$\overline{k}=21.1$ & &\\
& &\\
\hline
\end{tabular}
\caption{Results for problems with dimension $n=2$ of reference \cite{veld1999}.}
\label{table3}
\end{table}

% 2-Dimensional and bibliography veld1999 - Table 4
\begin{table}[h!]
\centering
\begin{tabular}{|l|c|c|}\hline
Problem & Objective functions & Pareto front\\
\hline
%---- PROBLEM PM12
 &\multirow{3}{*}{\includegraphics[scale=0.47]{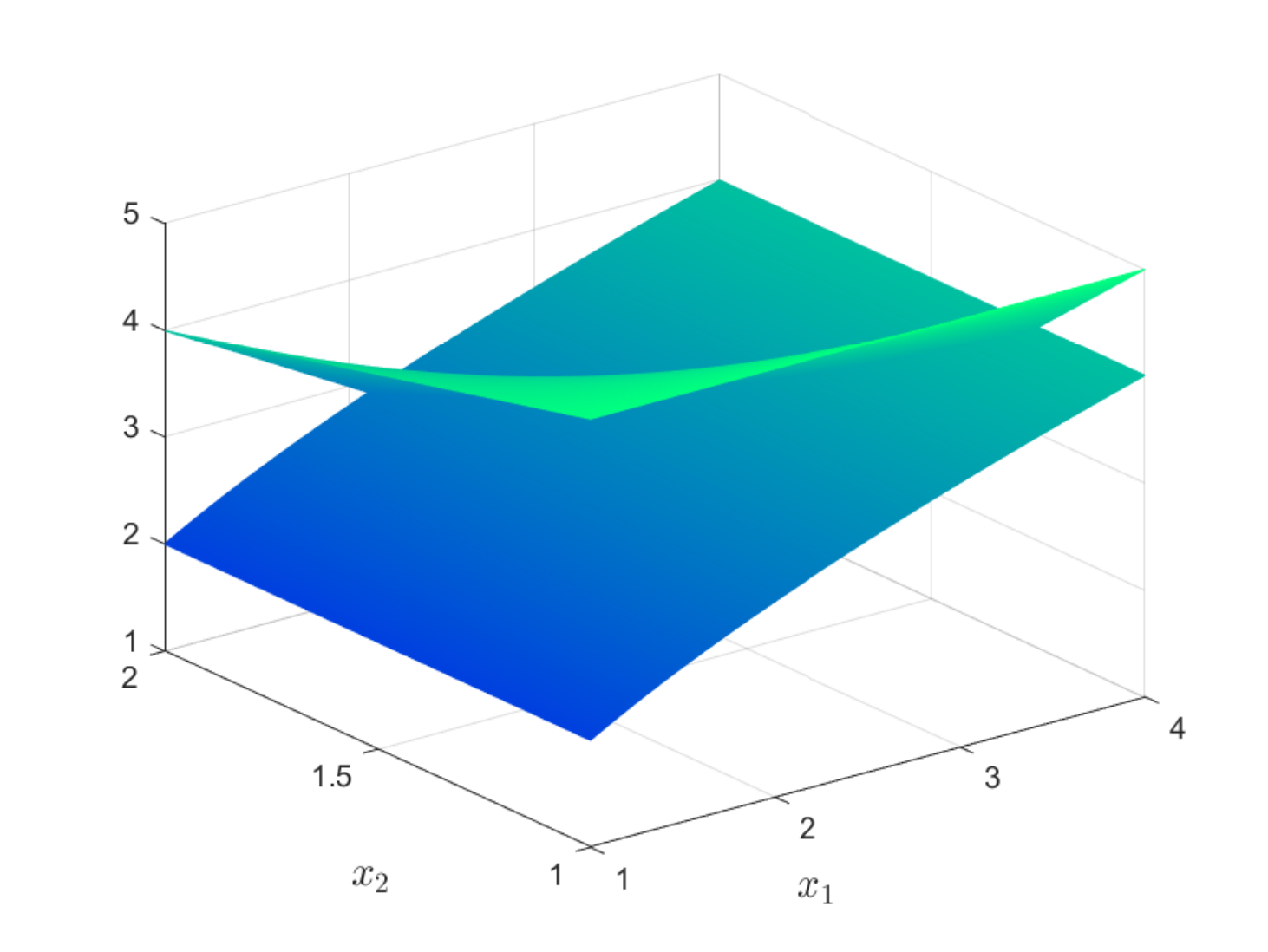}}
 &\multirow{3}{*}{\includegraphics[scale=0.47]{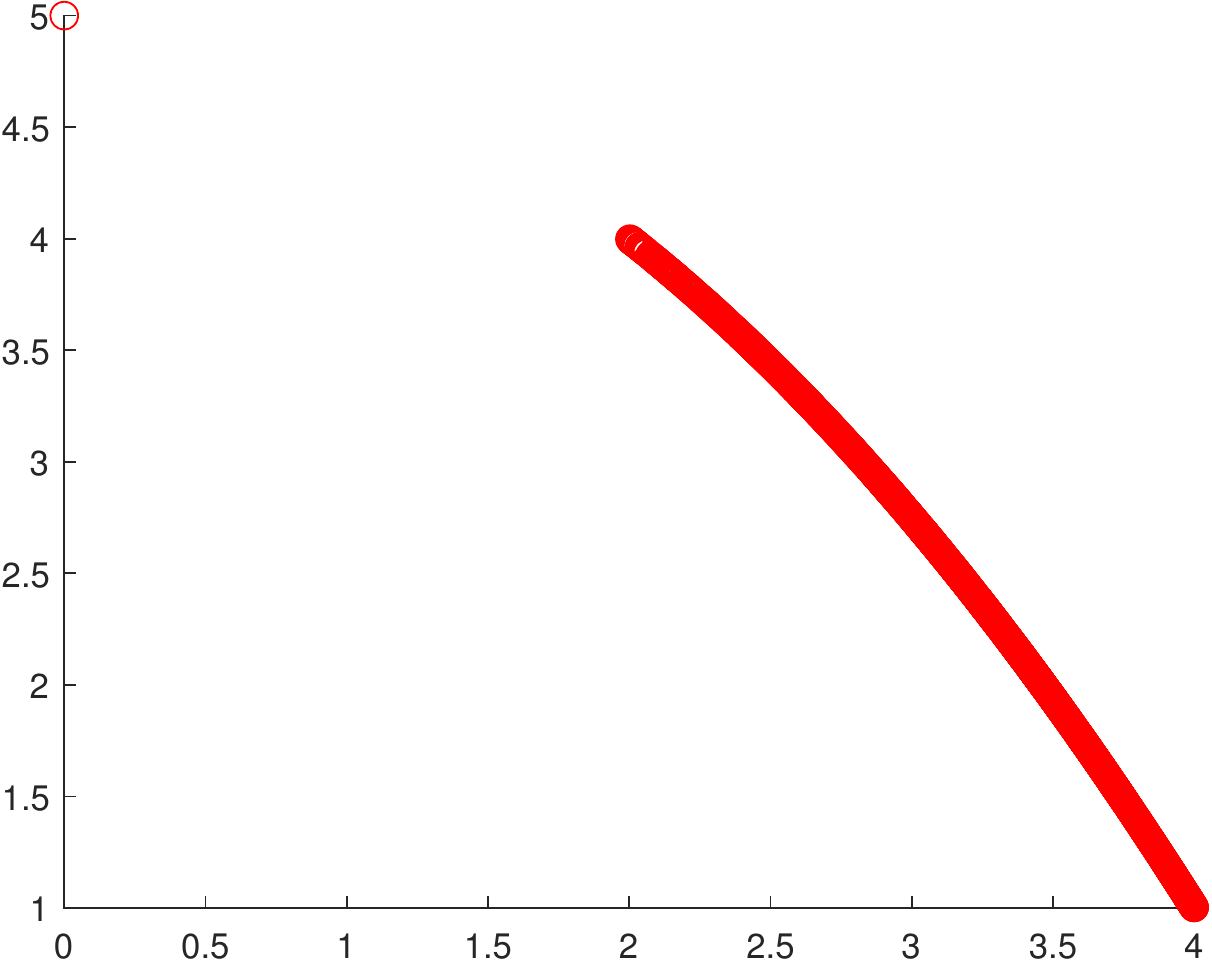}}\\
\cite[MOP 8]{veld1999} & &\\
& &\\
& &\\
$n=2$ & & \\
$r=2$ & &\\
 & &\\
& &\\
$T=0.0196$ sec & &\\
$\overline{k}=8.7$ & &\\
& &\\
\hline
%---- PROBLEM PM15
Problem &\multirow{3}{*}{\includegraphics[scale=0.47]{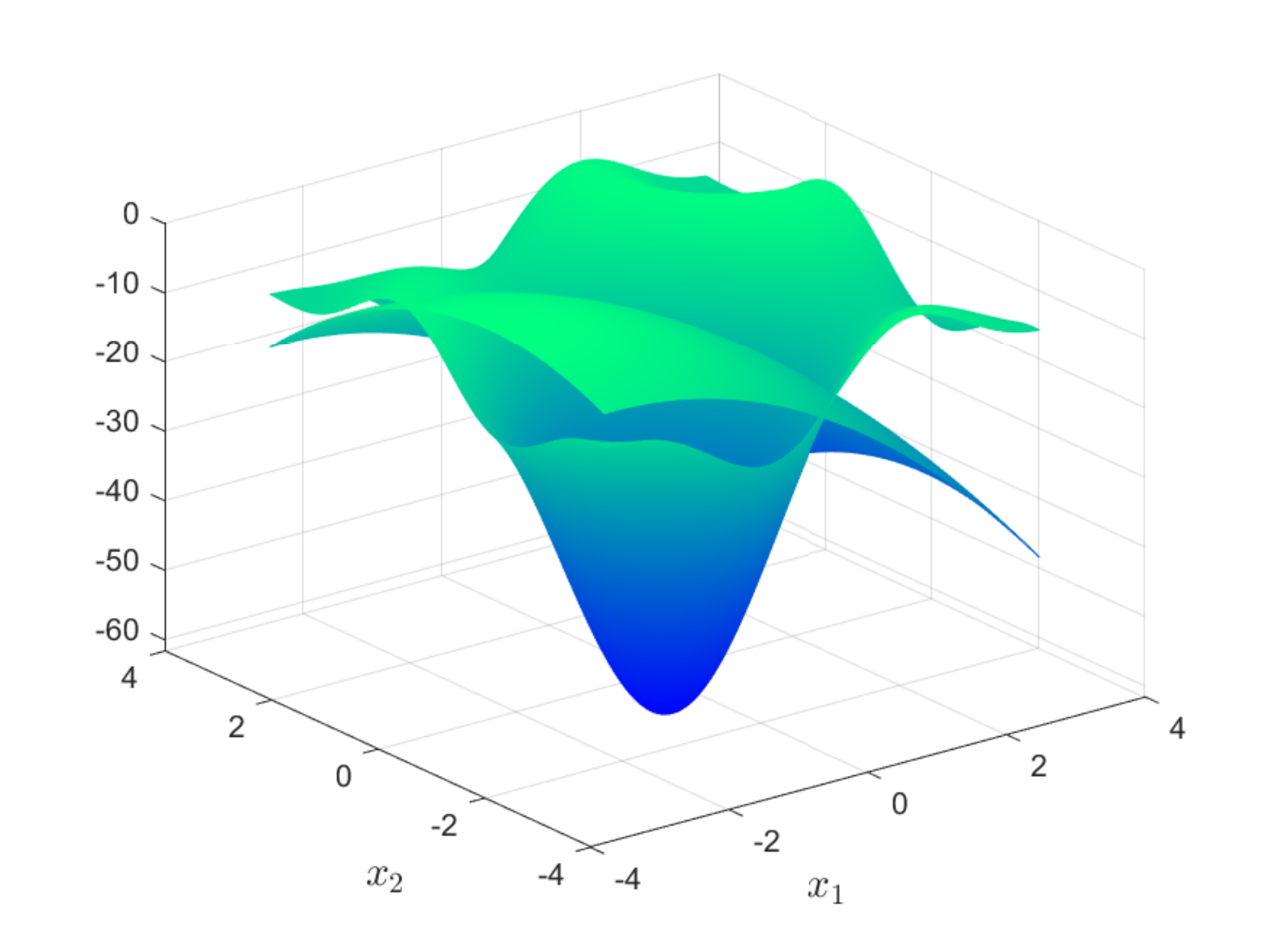}}
 &\multirow{3}{*}{\includegraphics[scale=0.47]{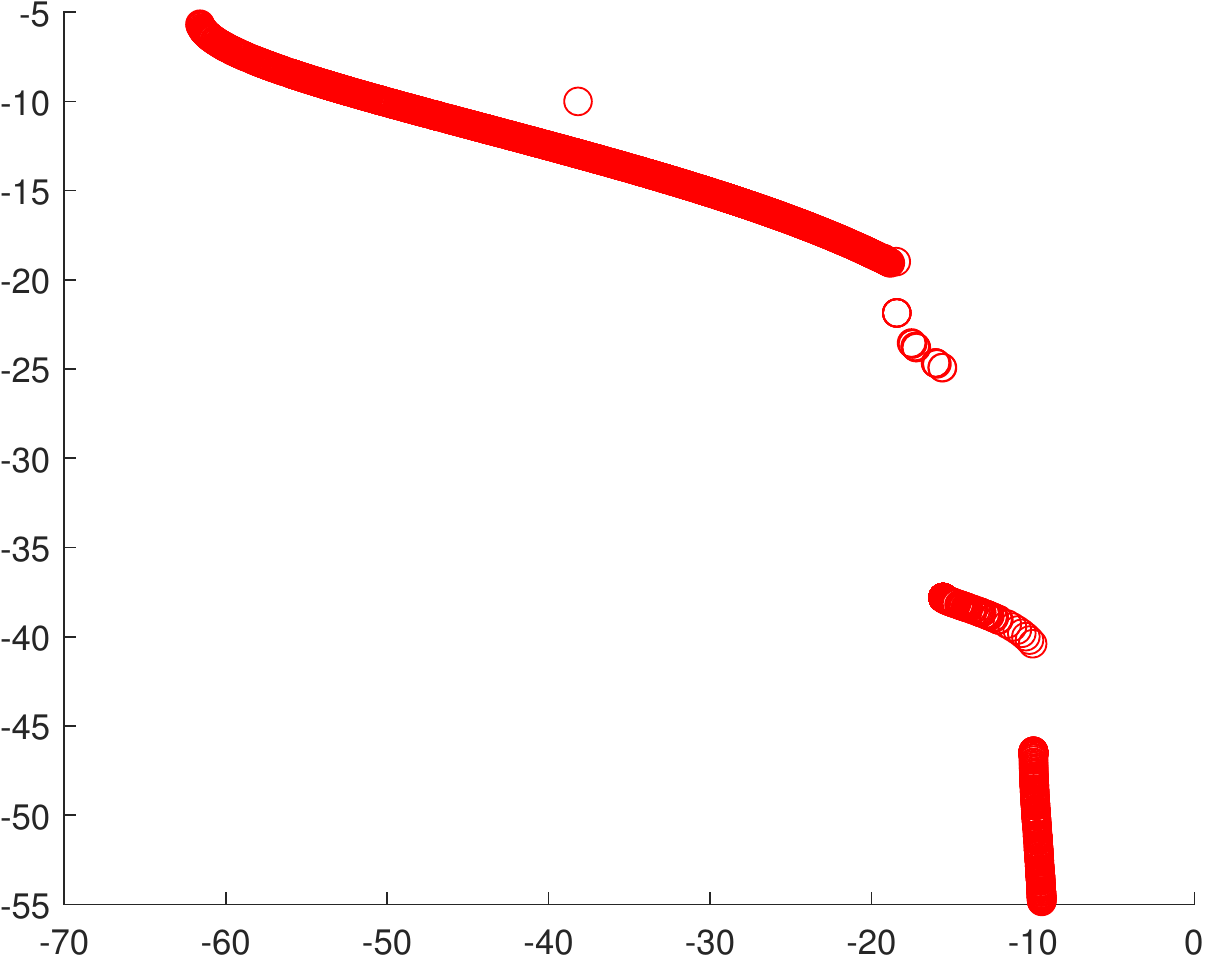}}\\
\cite[MOP 9]{veld1999} & &\\
& &\\
& &\\
$n=2$ & & \\
$r=2$ & &\\
& &\\
& &\\
$T=0.0339$ sec & &\\
$\overline{k}=14.3$ & &\\
& &\\
\hline
%---- PROBLEM PM16
Problem &\multirow{3}{*}{\includegraphics[scale=0.47]{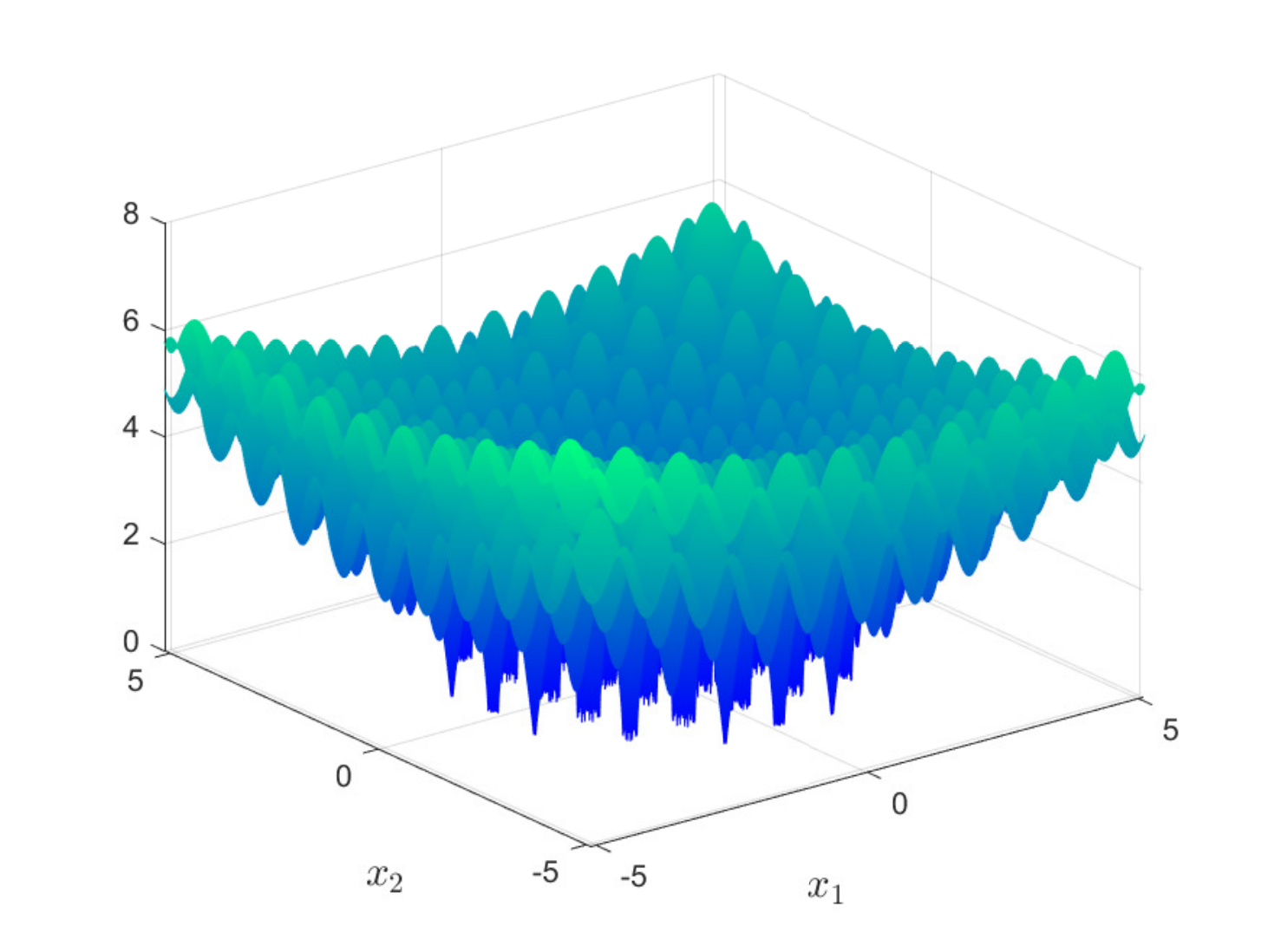}}
 &\multirow{3}{*}{\includegraphics[scale=0.47]{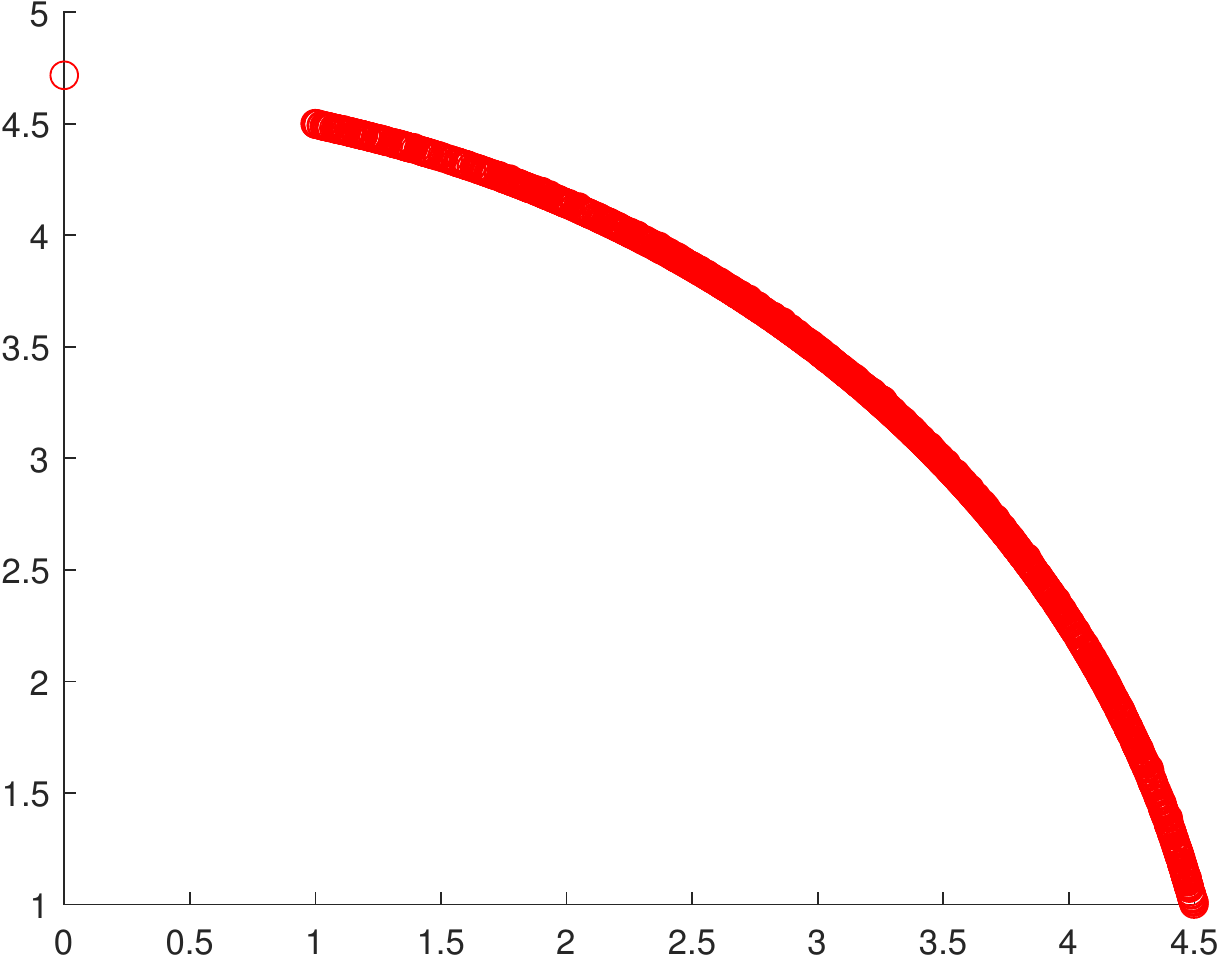}}\\
\cite[MOP 10]{veld1999} & &\\
& &\\
& &\\
$n=2$ & & \\
$r=2$ & &\\
& &\\
& &\\
$T=0.0323$ sec & &\\
$\overline{k}=13.2$ & &\\
& &\\
\hline
%------PM18
 &\multirow{3}{*}{\includegraphics[scale=0.47]{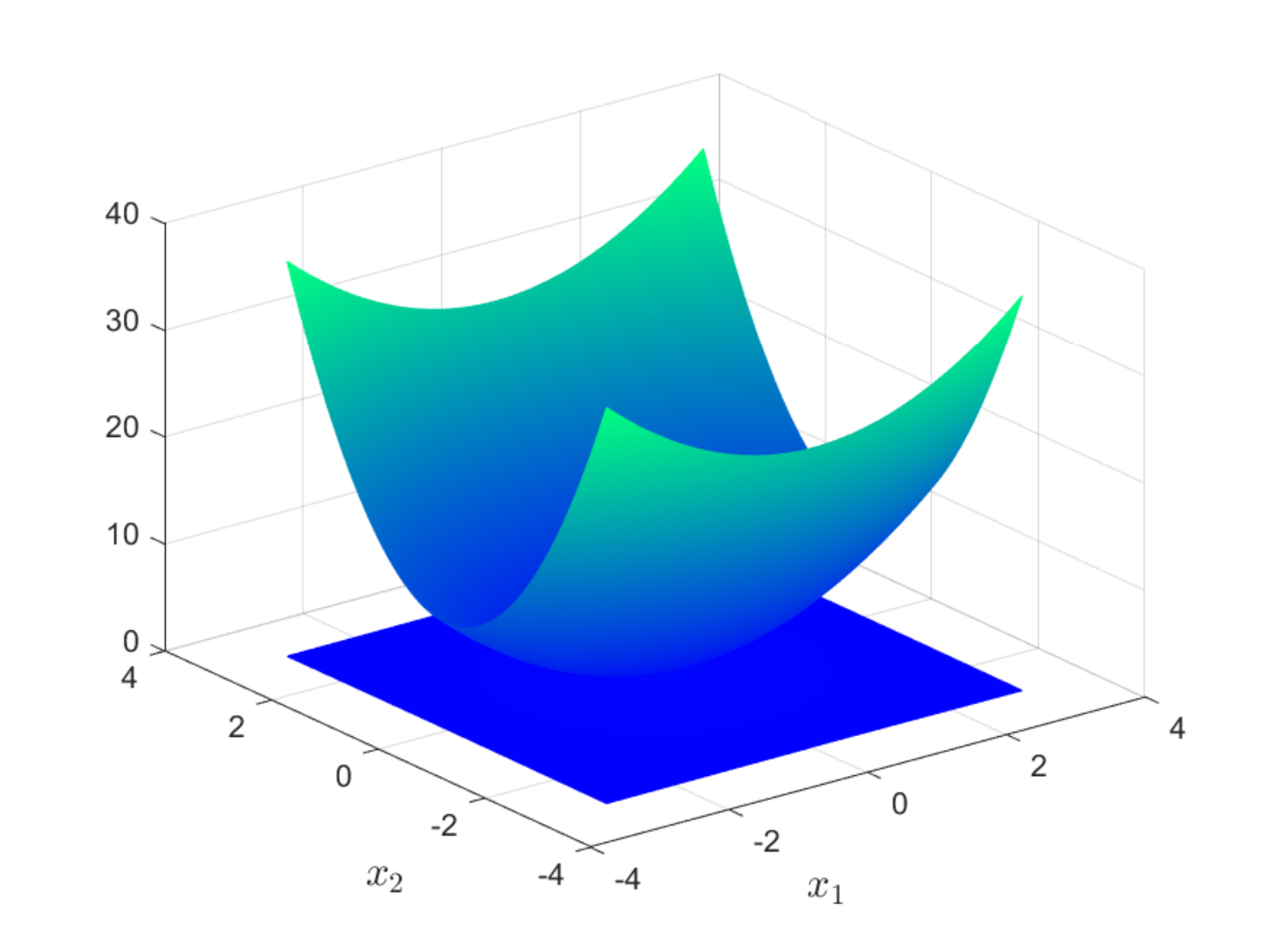}}
 &\multirow{3}{*}{\includegraphics[scale=0.47]{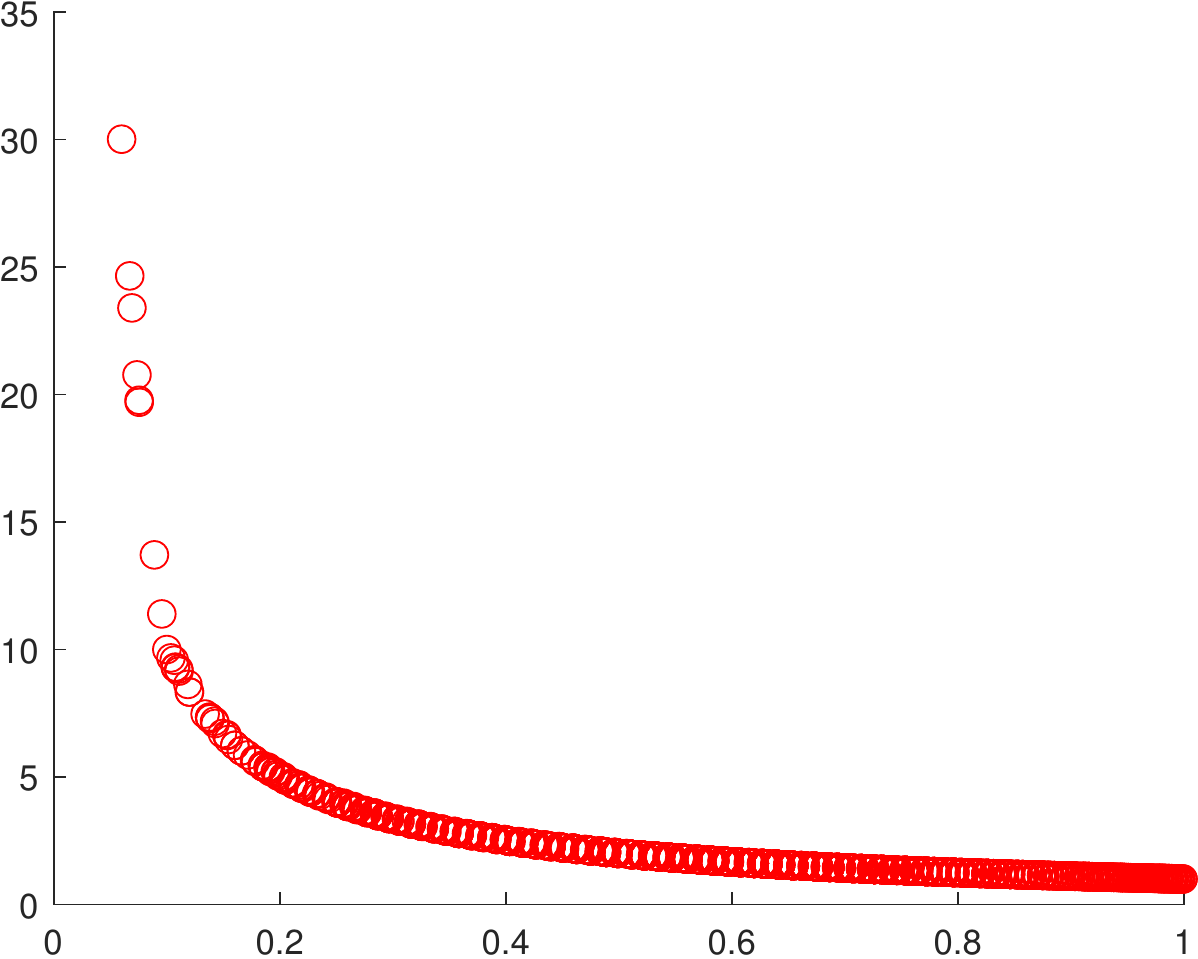}}\\
\cite[MOP 11]{veld1999}& &\\
& &\\
& &\\
$n=2$ & & \\
$r=2$ & &\\
& &\\
& &\\
$T=0.0242$ sec & &\\
$\overline{k}=12.1$ & &\\
& &\\
\hline
\end{tabular}
\caption{Results for problems with dimension $n=2$ of reference \cite{veld1999}.}
\label{table4}
\end{table}

% 2-dimensional and bibliography veld1999 - table 5
\begin{table}[h!]
\centering
\begin{tabular}{|l|c|c|}\hline
 Problem & Objective functions & Pareto front\\
\hline
%---- PROBLEM PM19
 &\multirow{3}{*}{\includegraphics[scale=0.47]{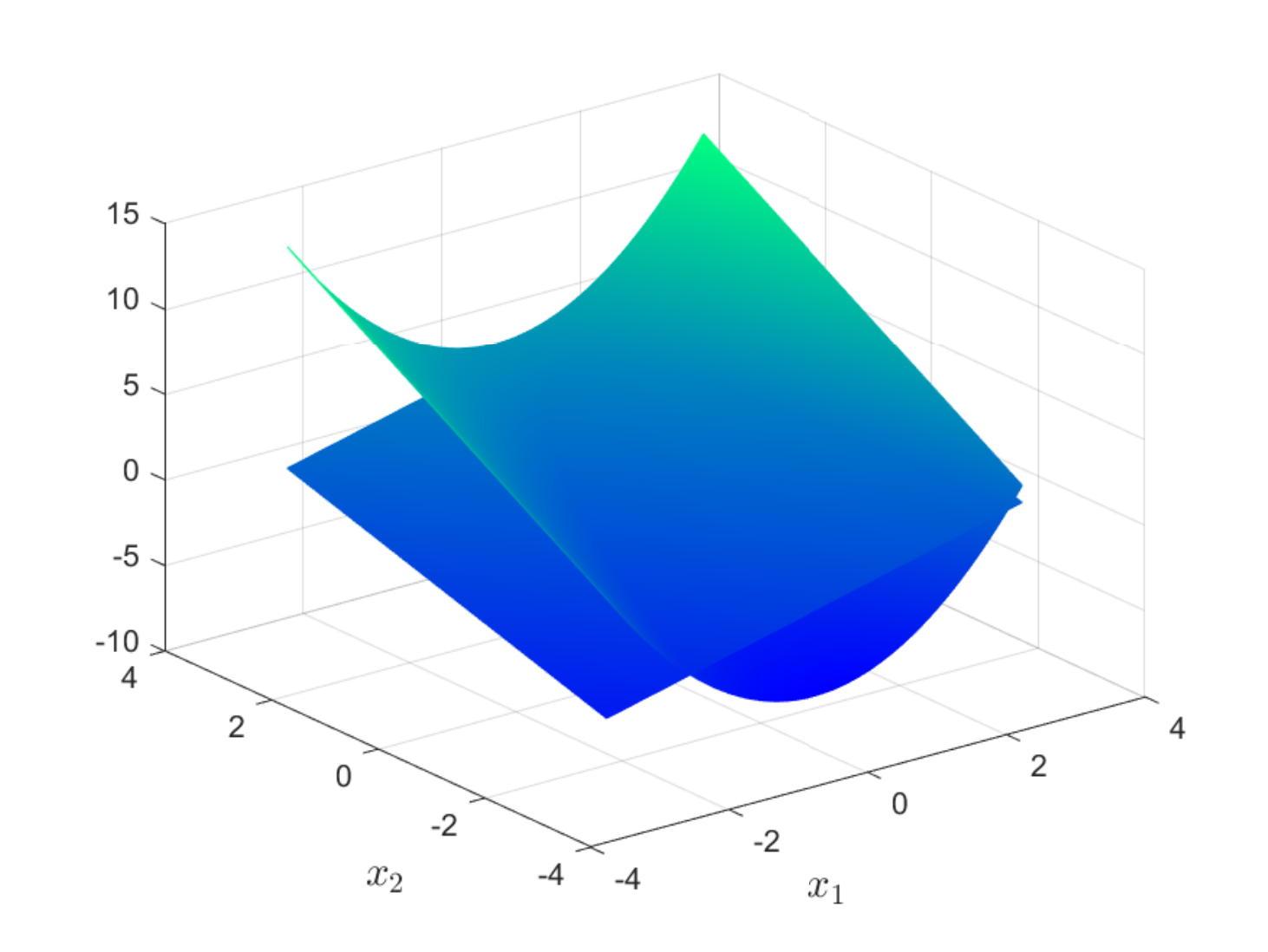}}
 &\multirow{3}{*}{\includegraphics[scale=0.47]{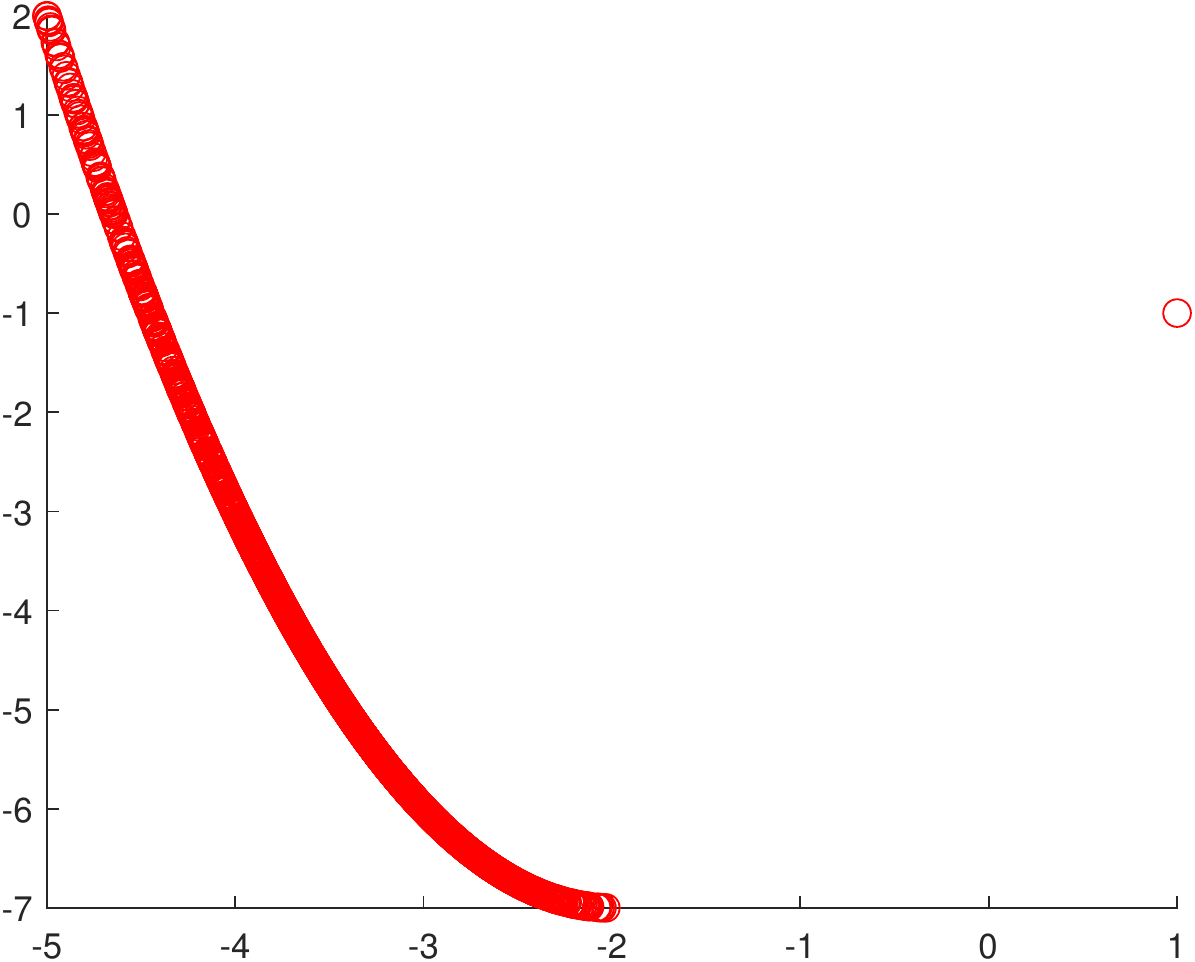}}\\
\cite[MOP 12]{veld1999} & &\\
& &\\
& &\\
$n=2$ & & \\
$r=2$ & &\\
 & &\\
& &\\
$T=0.0232$ sec & &\\
$\overline{k}=12.1$ & &\\
& &\\
\hline
 &\multirow{3}{*}{\includegraphics[scale=0.47]{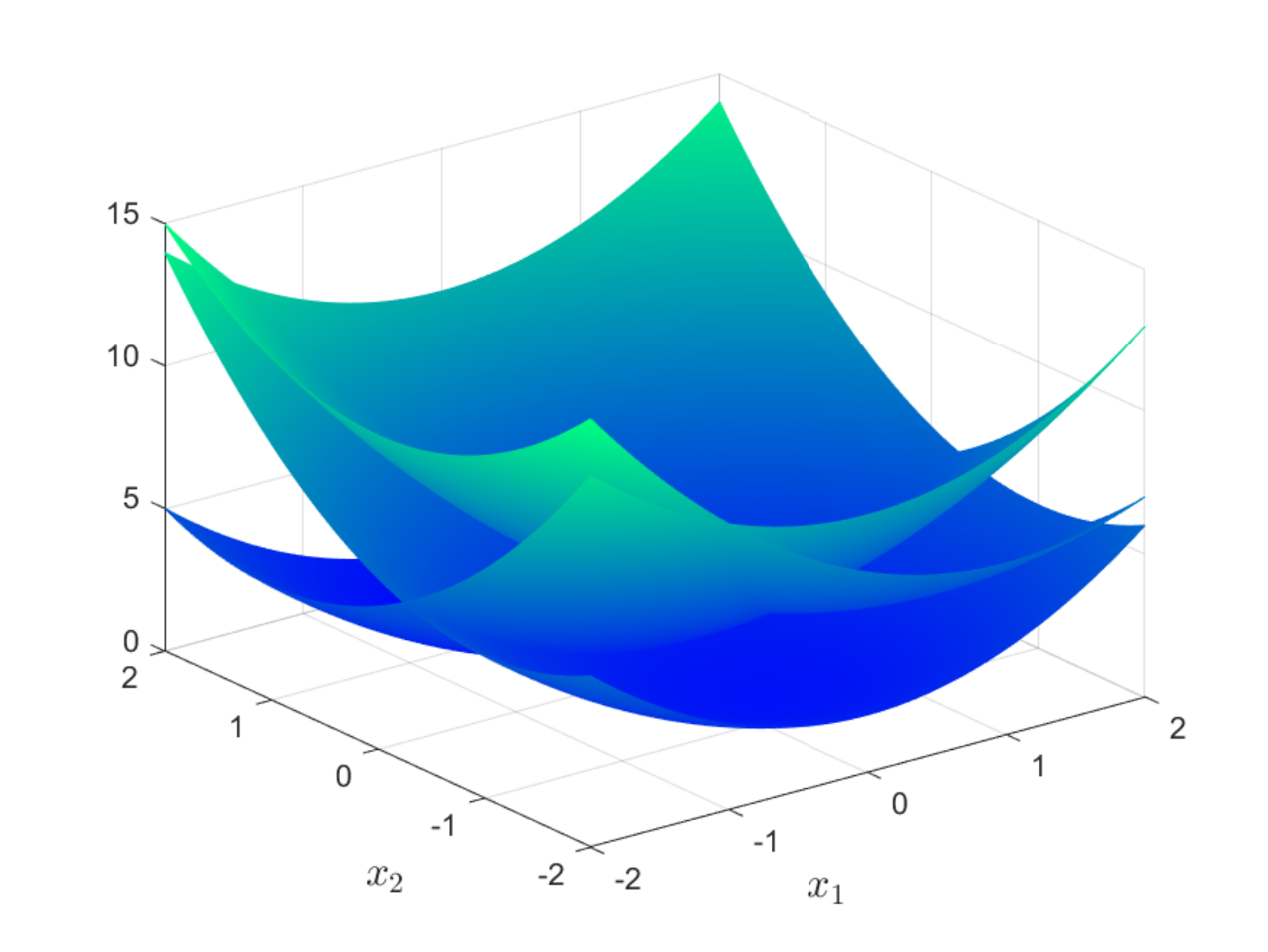}}
 &\multirow{3}{*}{\includegraphics[scale=0.47]{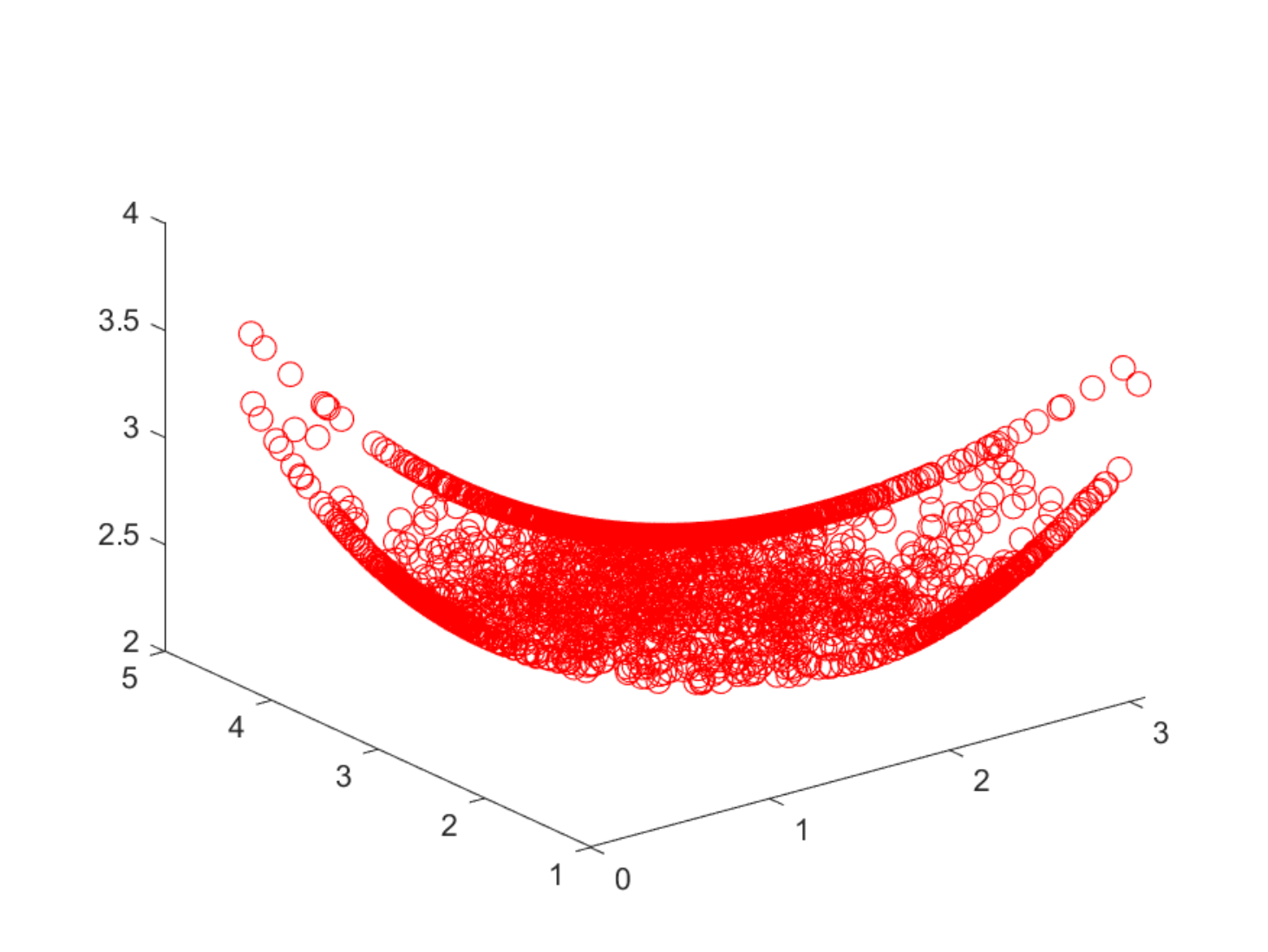}}\\
\cite[MOP 16]{veld1999}& &\\
& &\\
& &\\
$n=2$ & & \\
$r=3$ & &\\
 & &\\
& &\\
$T=0.0581$ sec & &\\
$\overline{k}=29.7$ & &\\
& &\\
\hline
%------PM22
 &\multirow{3}{*}{\includegraphics[scale=0.47]{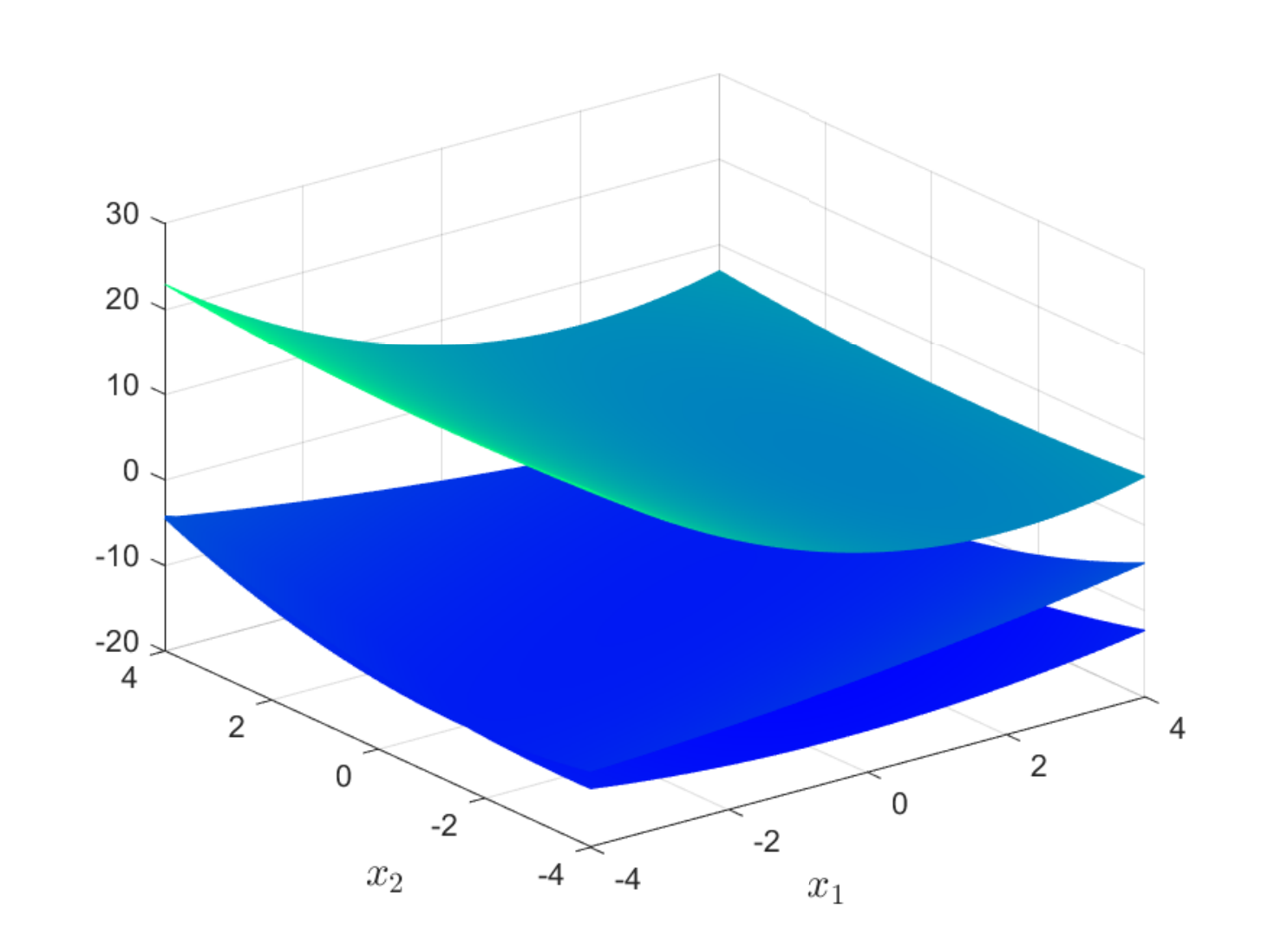}}
 &\multirow{3}{*}{\includegraphics[scale=0.47]{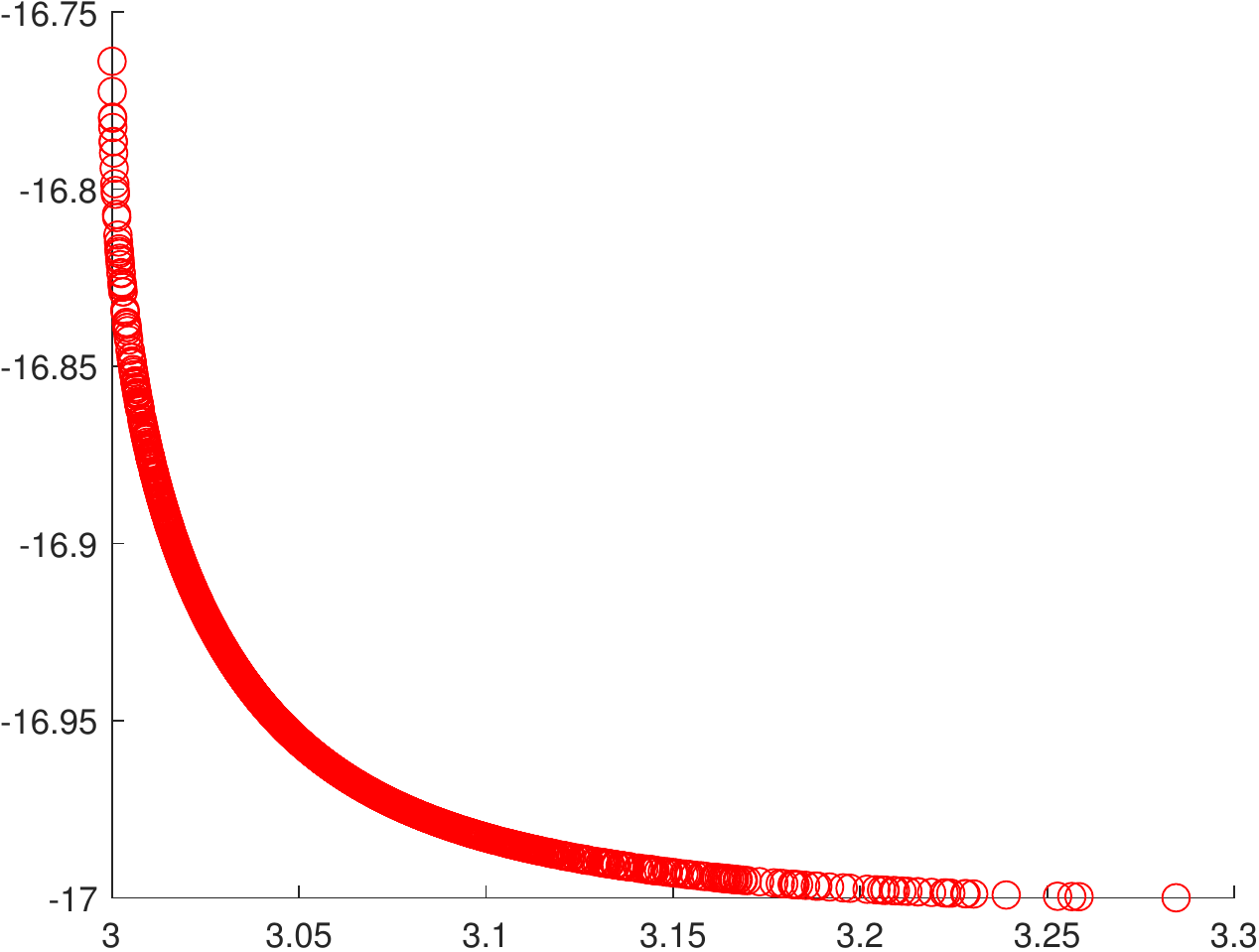}}\\
\cite[MOP 17]{veld1999}& &\\
& &\\
& &\\
$n=2$ & & \\
$r=3$ & &\\
 & &\\
& &\\
$T=0.0538$ sec & &\\
$\overline{k}=26.3$ & &\\
& &\\
\hline
&\multirow{3}{*}{\includegraphics[scale=0.47]{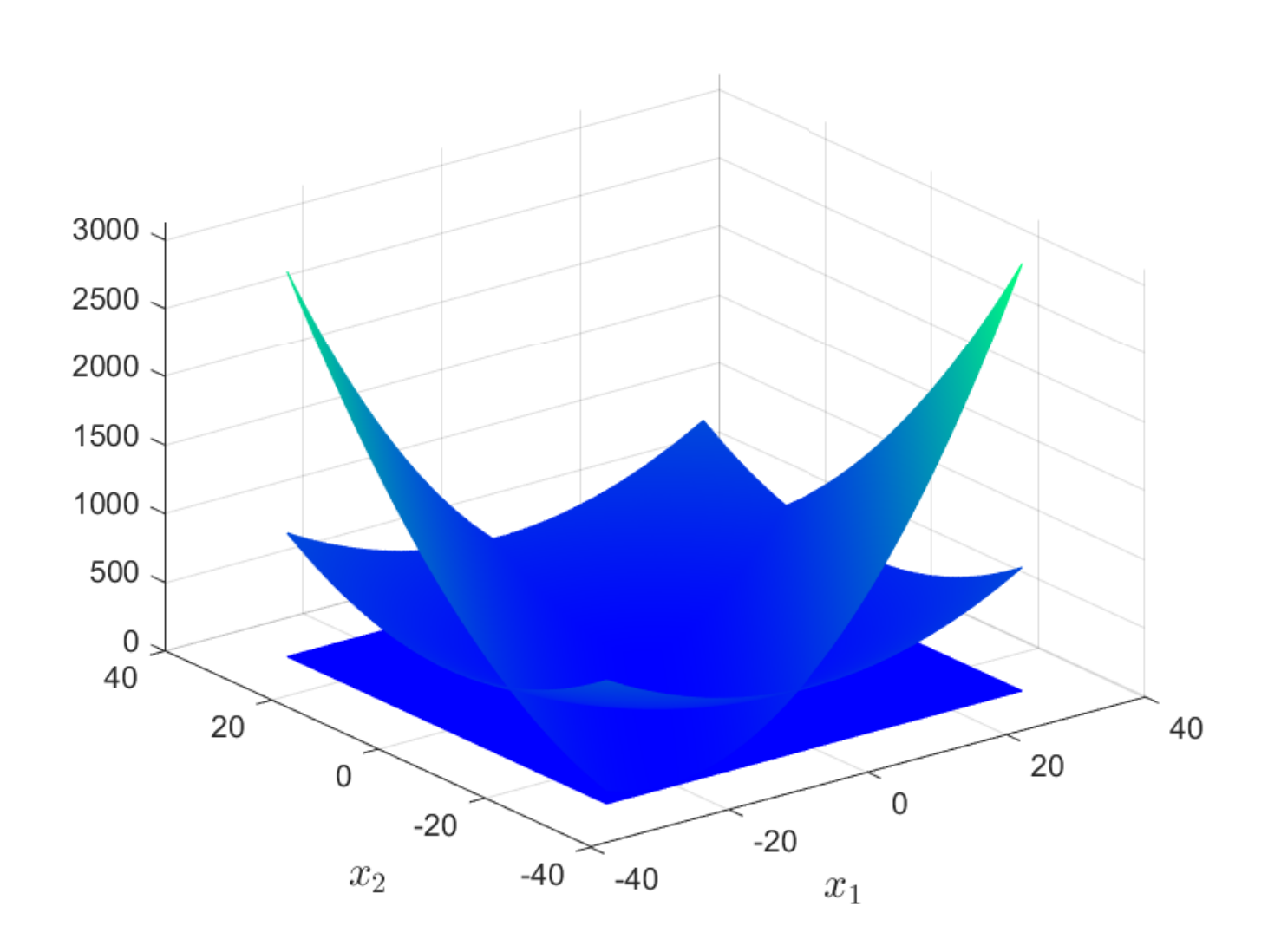}}
 &\multirow{3}{*}{\includegraphics[scale=0.47]{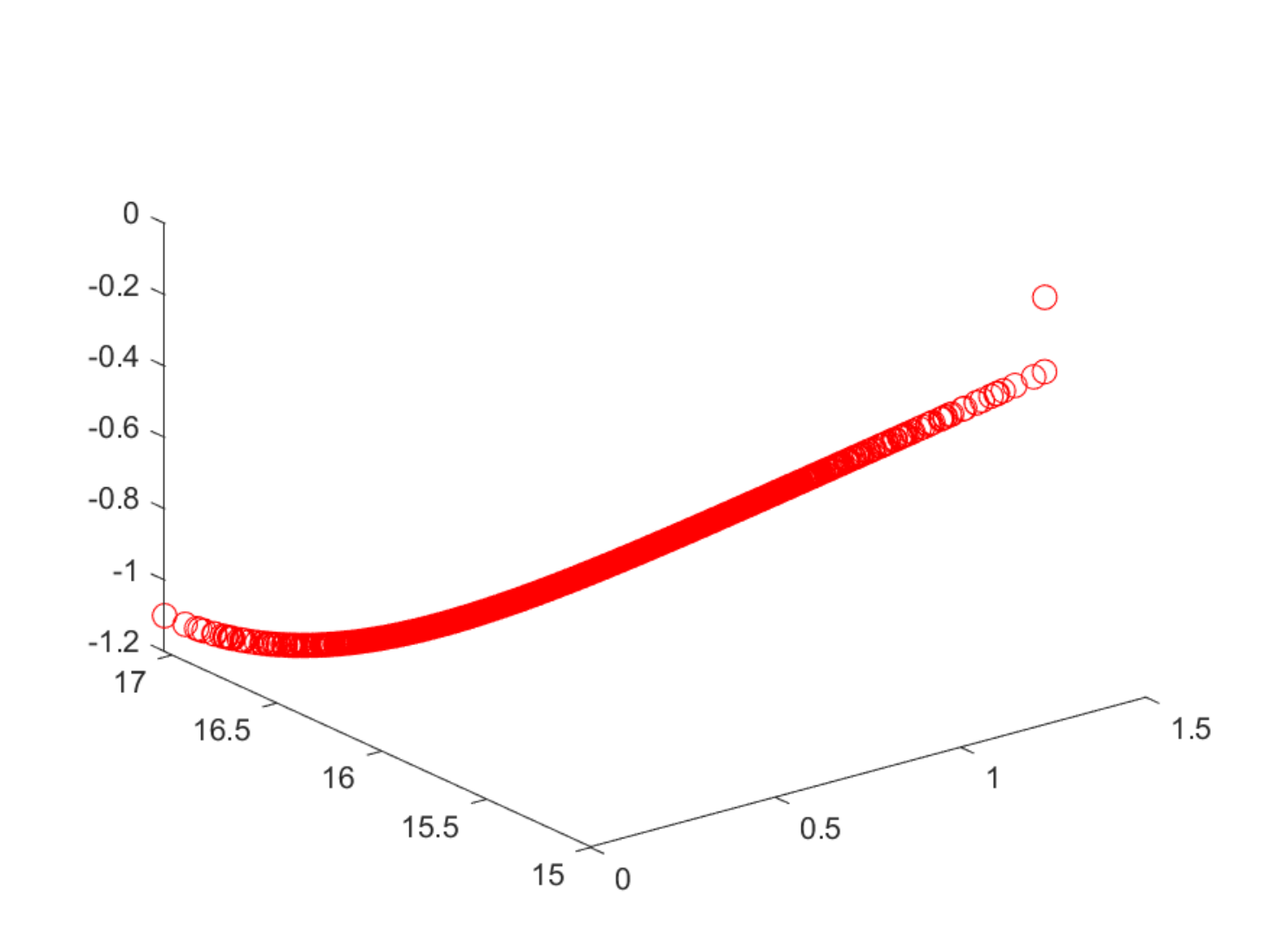}}\\
\cite[MOP 18]{veld1999} & &\\
& &\\
& &\\
$n=2$ & & \\
$r=3$ & &\\
 & &\\
& &\\
$T=0.0606$ sec & &\\
$\overline{k}=30.0$ & &\\
& &\\
\hline
\end{tabular}
\caption{Results for problems with dimension $n=2$ of reference \cite{veld1999}.}
\label{table5}
\end{table}

% Dimension bigger than 2
\begin{table}[h!]
\centering
\begin{tabular}{|l|c|}\hline
 Problem & Pareto front\\
\hline
&\multirow{2}{*}{\includegraphics[scale=0.47]{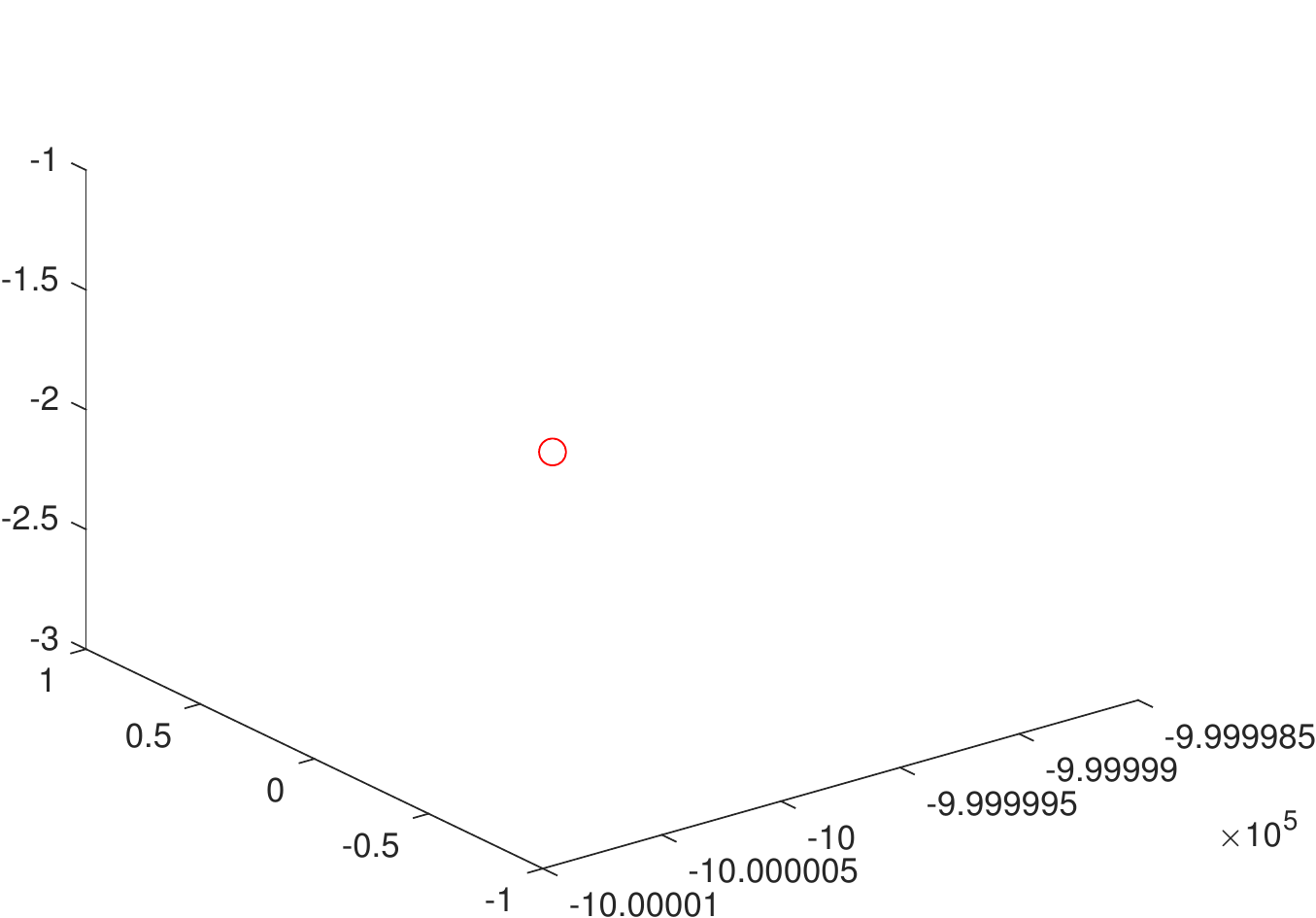}}\\
\cite[MOP 2]{veld1999} &\\
&\\
&\\
$n=2$ & \\
$r=3$ &\\
&\\
&\\
$T=0.0032$ sec &\\
$\overline{k}=1$ 
&\\
& \\
\hline
%------PM14
&\multirow{2}{*}{\includegraphics[scale=0.47]{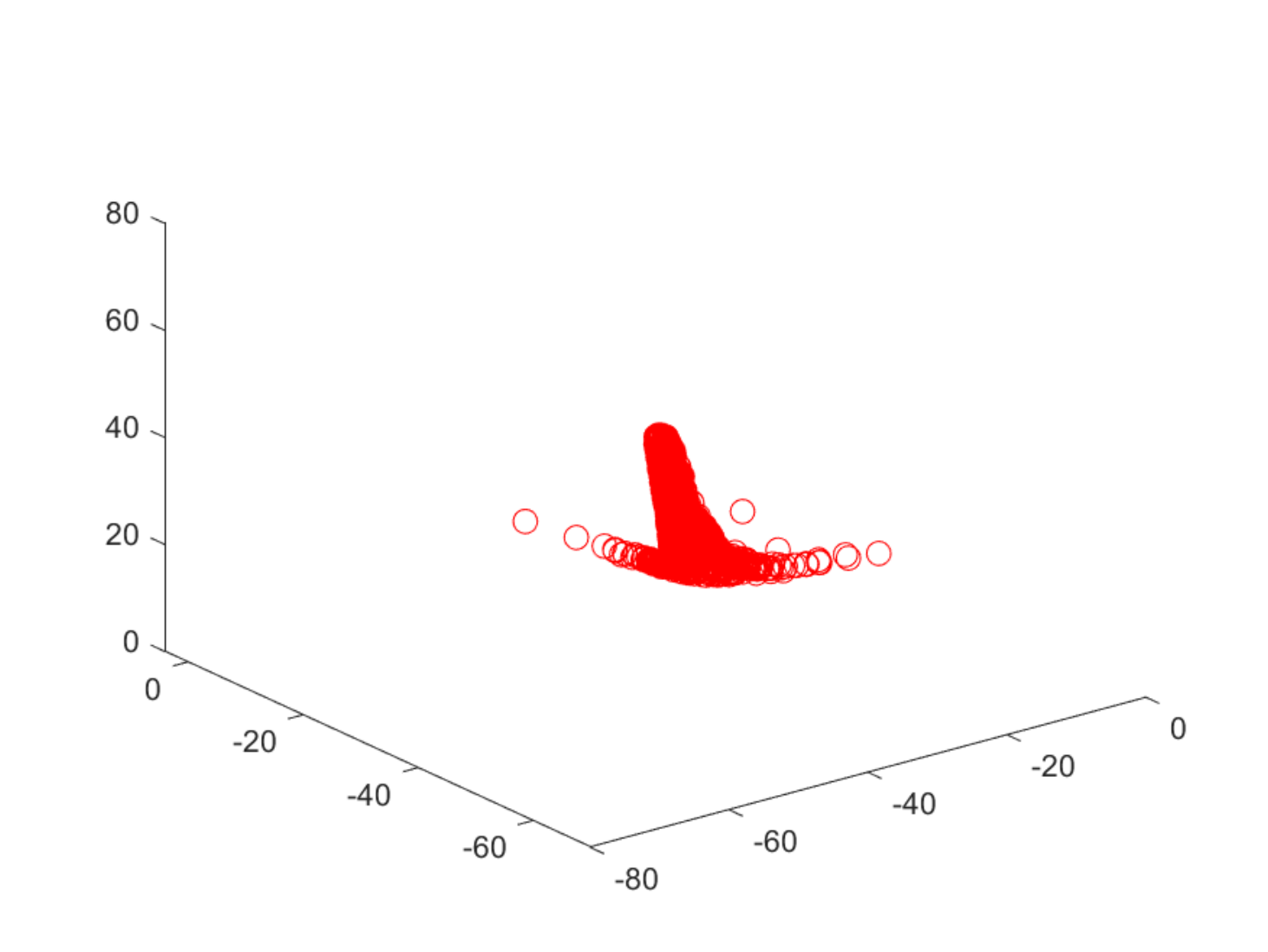}}\\
\cite[(6.28)]{deb2005} &\\
&\\
&\\
$n=3$ & \\
$r=3$ &\\
&\\
&\\
$T=0.0200$ sec &\\
$\overline{k}=6.3$ 
&\\
& \\
\hline
 &\multirow{2}{*}{\includegraphics[scale=0.47]{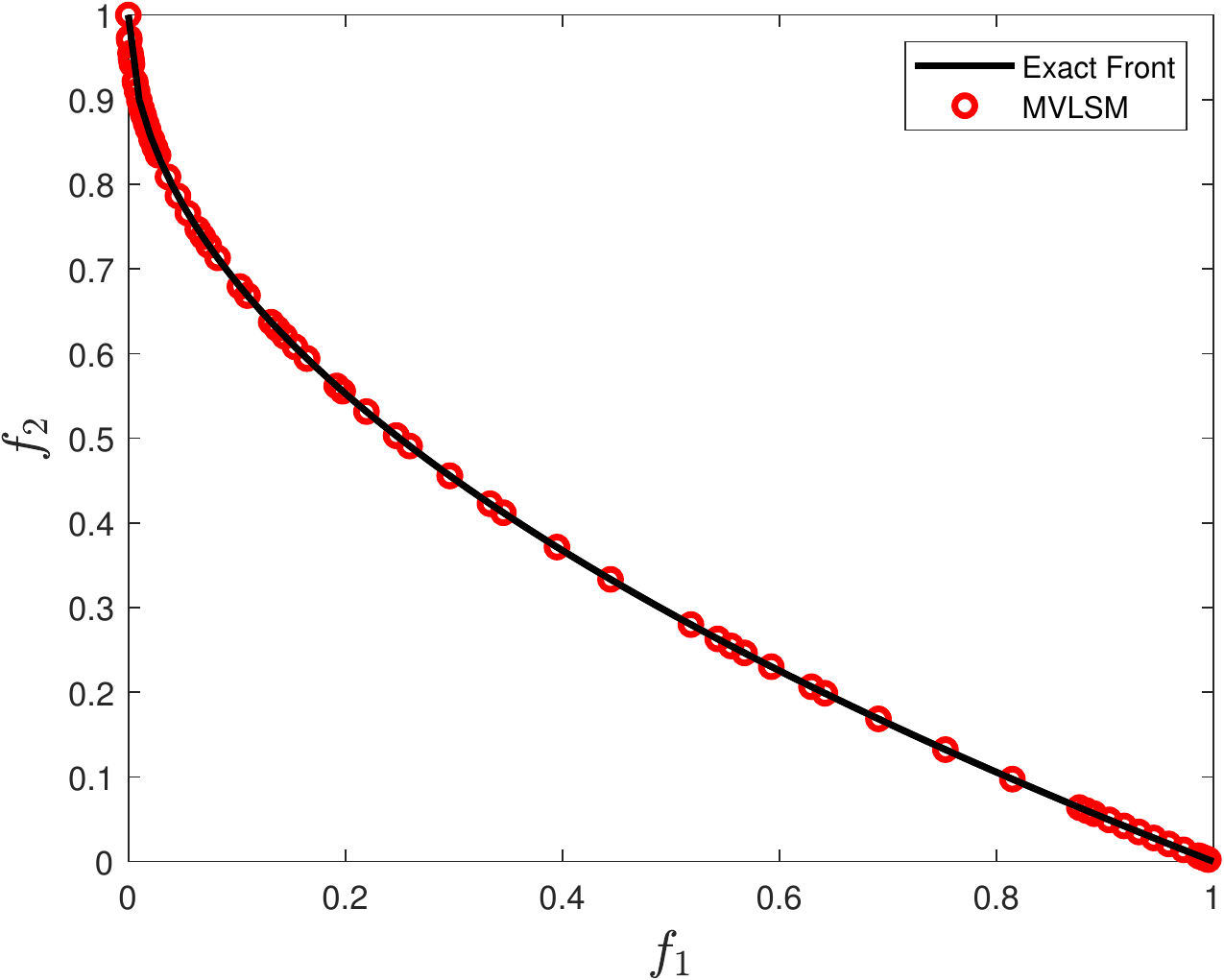}}\\
\cite[ZDT 1]{deb2001} &\\
& \\
&\\
$n=4$ &  \\
$r=2$ &\\
&\\
&\\
$T=0.0116$ sec & \\
$\overline{k}=3.5$ & \\
 &\\
 &\\
\hline
&\multirow{2}{*}{\includegraphics[scale=0.47]{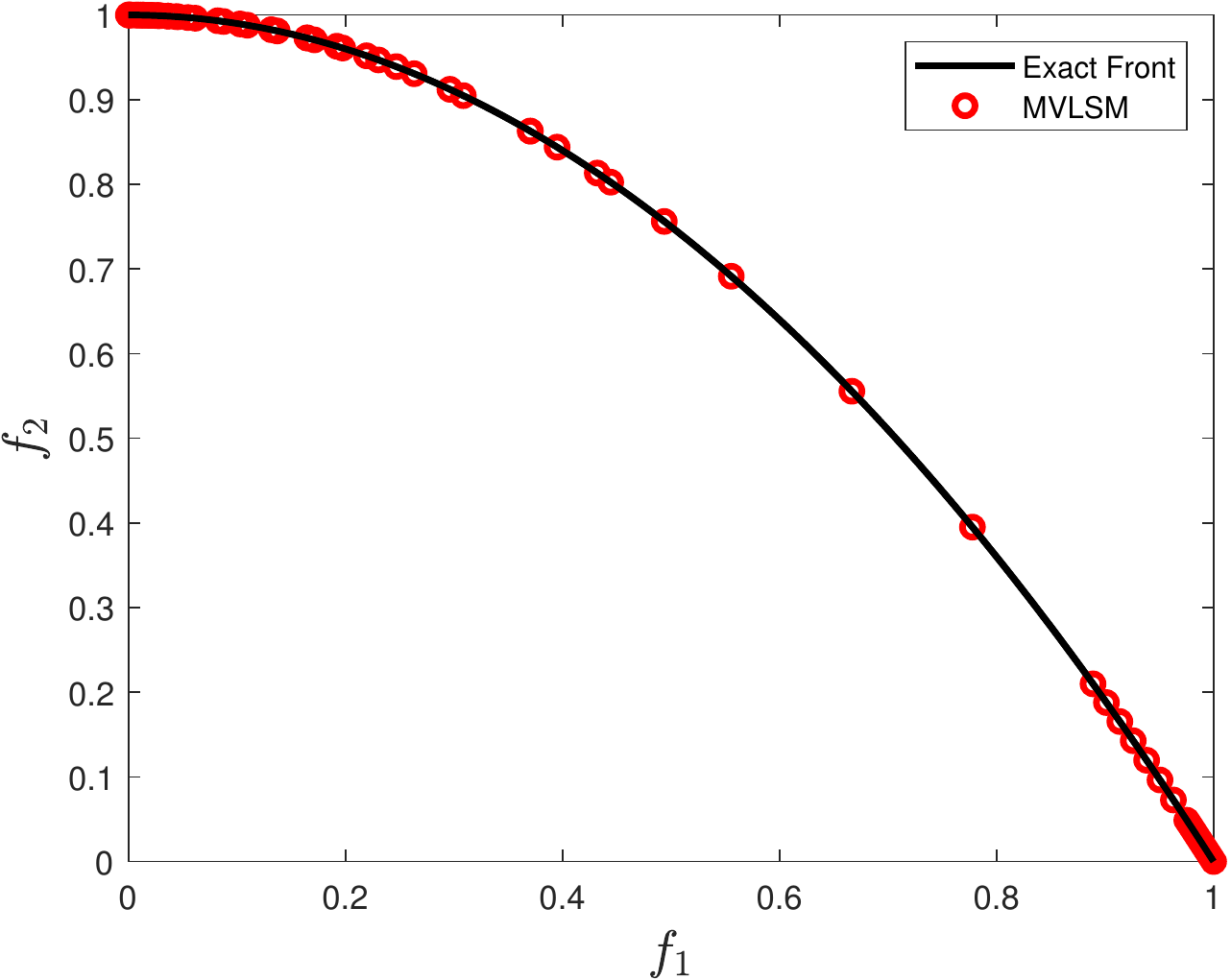}}\\
\cite[ZDT 2]{deb2001}&\\
&\\
&\\
$n=4$ &\\
$r=2$ &\\
&\\
&\\
$T=0.0110$ sec &\\
$\overline{k}=3.3$ & \\
&\\
&\\
\hline
\end{tabular}
\caption{Results for problems with dimension $n\geq 3$ or $r=3$.}
\label{table6}
\end{table}

% Dimension bigger than 2
\begin{table}[h!]
\centering
\begin{tabular}{|l|c|}\hline
 Problem & Pareto front\\
\hline
%------PM24
 &\multirow{2}{*}{\includegraphics[scale=0.47]{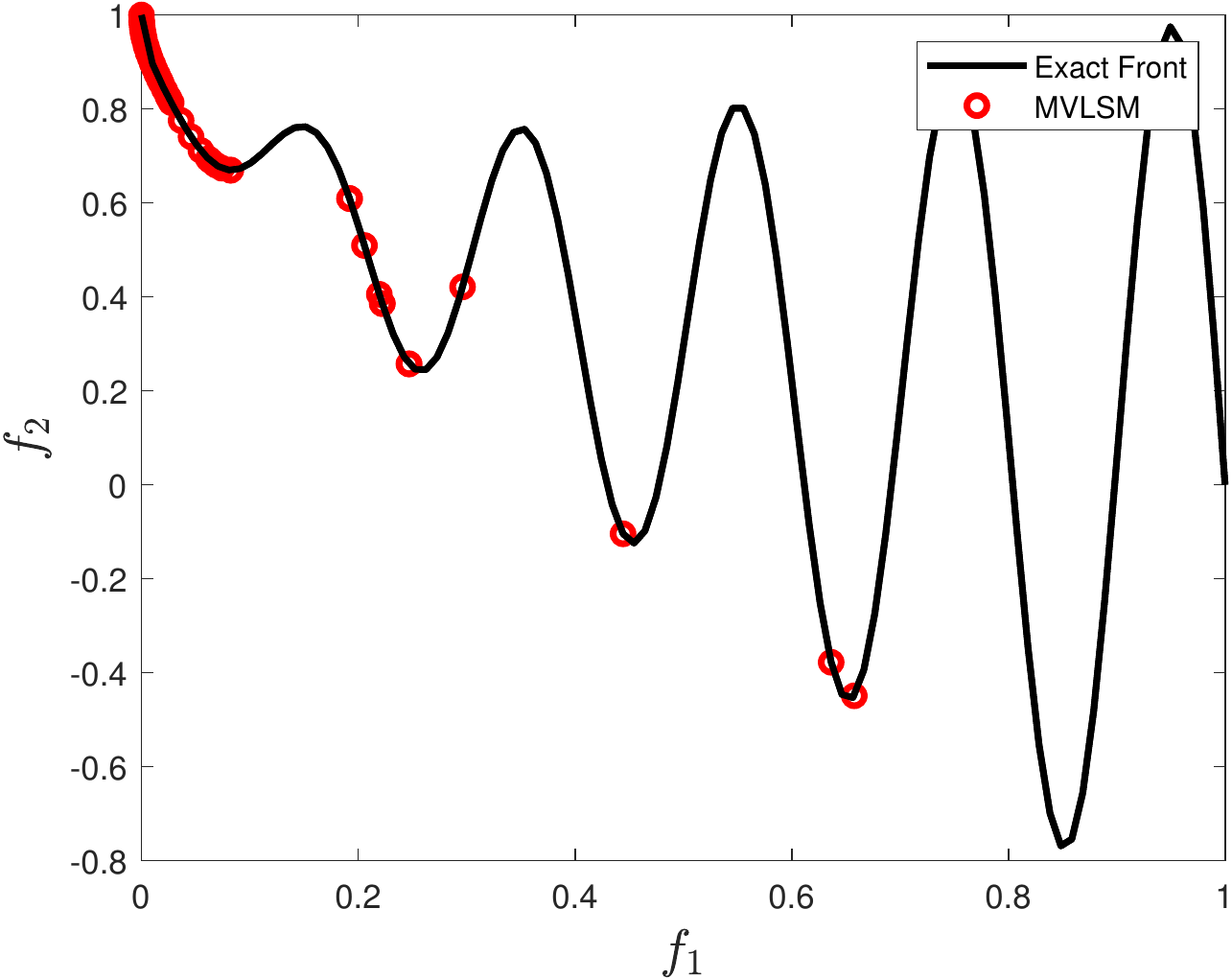}}\\
\cite[ZDT 3]{deb2001} &\\
&\\
$n=4$ & \\
$r=2$ &\\
&\\
&\\
$T=0.0123$ sec &\\
$\overline{k}=3.7$ &\\
&\\
\hline
%------PM25
 &\multirow{2}{*}{\includegraphics[scale=0.45]{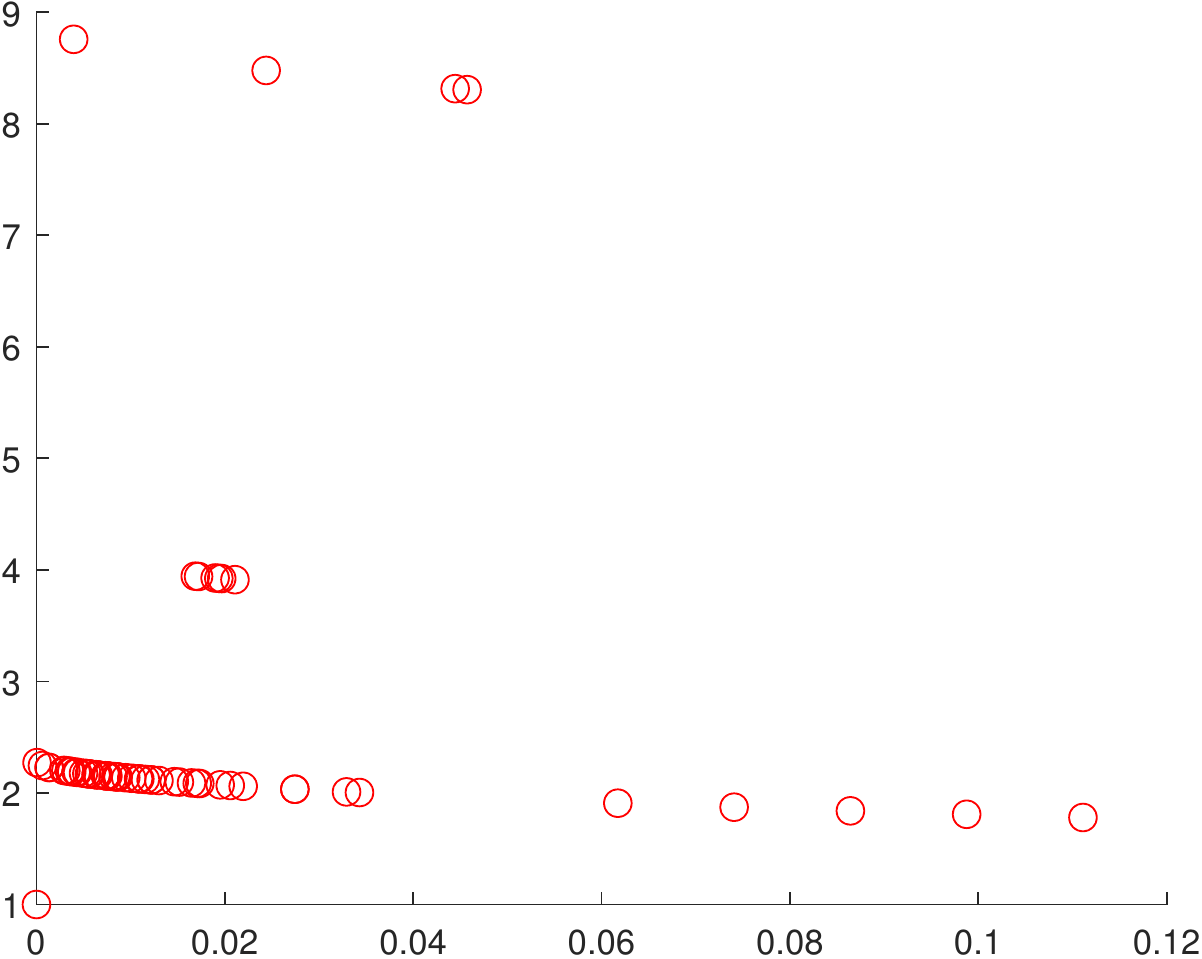}}\\
\cite[ZDT 4]{deb2001} &\\
& \\
&\\
$n=4$ &  \\
$r=2$ &\\
&\\
 &\\
$T=0.0392$ sec & \\
$\overline{k}=7.9$ & \\
&\\
\hline
\end{tabular}
\caption{Results for problems with dimension $n=4$.}
\label{table7}
\end{table}
%-------------------
\section{Conclusion}
\label{sec:conclusion}
In this paper, integral global optimality conditions are extended to multiobjective optimization problems from single objective case by using weighted scalarization techniques.  These conditions of optimality via integration can be a powerful tool to deal with several optimization problems of practical nature that appear in diverse areas of knowledge. Based on the theoretical results using Chebyshev scalarization, we proposed an algorithm to build an approximation of the  weak Pareto front.  The algorithm proposed was implemented in {\sc Matlab} and its good performance was illustrated  by solving a set of  unconstrained and box constrained problems with continuous variables and dimension at most $4$. 

The integral optimality conditions are interesting, among other reasons, because they can be applied even in the non-smooth case since no kind of derivative (or sub-derivative) is used. On the other hand, these conditions are stated in terms of multiple integrals, which may narrow applying such a theory to problems with many variables. Future research topics include implementing efficient methods to compute integrals with many variables and smarter choices of the weights to get points well spread in the (weak) Pareto front approximation. Also, we intend to study integral optimality conditions using other scalarization techniques.

%----------------------
\subsubsection*{Acknowledgments.}
{\small
The authors are thankful to Fernanda Maria Pereira, Valeriano Antunes de Oliveira  and to the anonymous referees whose suggestions led to improvements in the paper. The first author was partially supported by  CAPES - Brazil and Funda\c c\~ao para a Ci\^encia e a Tecnologia (FCT) through the projects PTDC/MAT-APL/28400/2017, UIDB/00297/2020, UI/BD/151246/2021, and UIDP/00297/2020 (CMA), Portugal}. The third author was partially supported by the European Regional Development Fund (ERDF) and by the Ministry of Economy, Knowledge, Business and University, of the Junta de Andaluc\'ia - Spain, within the framework of the FEDER Andalucía 2014-2020 operational program (UPO-1381297).

%-----------------------------------
\bibliographystyle{abbrv}
\bibliography{refs}

\end{document}